\documentclass[11pt,oneside,reqno]{article}
\usepackage[margin=1.3in]{geometry}

\usepackage{amsmath,amsthm,amssymb,color}
\usepackage[numbers]{natbib}
\usepackage{booktabs}

\RequirePackage{amsthm,amsmath,amsfonts,amssymb,mathrsfs,dsfont,mathtools,thmtools}
\RequirePackage{natbib}
\RequirePackage[colorlinks,citecolor=blue,urlcolor=blue,linkcolor=blue]{hyperref}
\usepackage[capitalize]{cleveref}
\RequirePackage{graphicx}

\crefname{equation}{}{}
\crefformat{equation}{\textup{(#2#1#3)}}
\crefrangeformat{equation}{\textup{(#3#1#4)--(#5#2#6)}}
\crefdefaultlabelformat{#2\textup{#1}#3}
\crefname{lemma}{Lemma}{Lemmas}
\crefname{page}{p.}{pp.}
\usepackage{bbm}
\usepackage{mathrsfs}
\usepackage{graphicx}
\usepackage{tikz}
\usepackage{enumitem}
\usetikzlibrary{arrows, automata}
\usetikzlibrary{calc}
\usetikzlibrary{positioning}

\numberwithin{equation}{section}
\allowdisplaybreaks[4]

\theoremstyle{plain}
\newtheorem{theorem}{Theorem}[section]
\newtheorem{proposition}{Proposition}[section]
\newtheorem{lemma}{Lemma}[section]
\newtheorem{corollary}{Corollary}[section]

\theoremstyle{definition}

\newtheorem{example}{Example}[section]
\newtheorem{remark}{Remark}[section]

\makeatletter

\newcount\minute
\newcount\hour
\newcount\hourMins
\def\now{%
\minute=\time%
\hour=\time \divide \hour by 60%
\hourMins=\hour \multiply\hourMins by 60%
\advance\minute by -\hourMins%
\zeroPadTwo{\the\hour}:\zeroPadTwo{\the\minute}%
}
\def\zeroPadTwo#1{\ifnum #1<10 0\fi#1}

\renewcommand{\cite}{\citet}

\def\^#1{\ifmmode {\mathaccent"705E #1} \else {\accent94 #1} \fi}
\def\~#1{\ifmmode {\mathaccent"707E #1} \else {\accent"7E #1} \fi}

\def\*#1{#1^\ast}
\edef\-#1{\noexpand\ifmmode {\noexpand\bar{#1}} \noexpand\else \-#1\noexpand\fi}
\def\>#1{\vec{#1}}
\def\.#1{\dot{#1}}

\def\wt#1{\widetilde{#1}}

\def\atop{\@@atop}
\def\*#1{\mathscr{#1}}

\renewcommand{\leq}{\leqslant}
\renewcommand{\geq}{\geqslant}
\newcommand{\eps}{\varepsilon}
\renewcommand{\eps}{\varepsilon}

\newcommand{\eq}{\eqref}

\newcommand{\diag}{{\mathop{\mathrm{diag}}}}

\newcommand{\IE}{\mathbbm{E}}

\newcommand{\Var}{\mathop{\mathrm{Var}}\nolimits}

\def\be#1{\begin{equation*}#1\end{equation*}}
\def\ben#1{\begin{equation}#1\end{equation}}
\def\bes#1{\begin{equation*}\begin{split}#1\end{split}\end{equation*}}
\def\besn#1{\begin{equation}\begin{split}#1\end{split}\end{equation}}

\def\bm#1{\begin{multline*}#1\end{multline*}}

\def\ba#1{\begin{align*}#1\end{align*}}
\def\ban#1{\begin{align}#1\end{align}}

\def\norm#1{\Vert#1\Vert}

\def\lnorm#1{\left\Vert#1\right\Vert}

\def\mid{\vert}

\def\beqn#1\eeqn{\begin{align}#1\end{align}}
\def\beq#1\eeq{\begin{align*}#1\end{align*}}

\usepackage{graphicx}
\usepackage{latexsym}
\usepackage{amsmath,amsthm,amssymb,amscd}
\usepackage{epsf,amsmath}

\def\E{{\IE}}

\newcommand{\ul}[1]{\underline{#1}}
\newcommand{\ol}[1]{\overline{#1}}
\newcommand{\mcl}[1]{\mathcal{#1}}
\newcommand{\mf}[1]{\mathfrak{#1}}

\DeclareMathOperator{\Inf}{Inf}

\renewcommand\section{\@startsection {section}{1}{\z@}%
{-3.5ex \@plus -1ex \@minus -.2ex}%
{1.3ex \@plus.2ex}%
{\center\small\sc\mathversion{bold}}}
\def\subsection#1{\@startsection {subsection}{2}{0pt}%
{-3.5ex \@plus -1ex \@minus -.2ex}%
{1ex \@plus.2ex}%
{\bf\mathversion{bold}}{#1}}

\def\subsubsection#1{\@startsection{subsubsection}{3}{0pt}%
{\medskipamount}%
{-10pt}%
{\normalsize\itshape}{\kern-2.2ex. #1.}}

\def\blfootnote{\xdef\@thefnmark{}\@footnotetext}

\makeatother

\begin{document}

\title{From $p$-Wasserstein Bounds to Moderate Deviations}
\author{Xiao Fang and Yuta Koike}
\date{\it The Chinese University of Hong Kong and The University of Tokyo} %\\[2ex]  \today %, \now}
\maketitle

\noindent{\bf Abstract:} 
We use a new method via $p$-Wasserstein bounds to prove Cram\'er-type moderate deviations in (multivariate) normal approximations. 
In the classical setting that $W$ is a standardized sum of $n$ independent and identically distributed (i.i.d.) random variables with sub-exponential tails, our method recovers the optimal range of $0\leq x=o(n^{1/6})$ and the near optimal error rate $O(1)(1+x)(\log n+x^2)/\sqrt{n}$ for $P(W>x)/(1-\Phi(x))\to 1$, where $\Phi$ is the standard normal distribution function. 
Our method also works for dependent random variables (vectors) and we give applications to the combinatorial central limit theorem, Wiener chaos, homogeneous sums and local dependence. 
The key step of our method is to show that the $p$-Wasserstein distance between the distribution of the random variable (vector) of interest and a normal distribution grows like $O(p^\alpha \Delta)$, $1\leq p\leq p_0$, for some constants $\alpha, \Delta$ and $p_0$. 
In the above i.i.d.\ setting, $\alpha=1, \Delta=1/\sqrt{n}, p_0=n^{1/3}$. 
For this purpose, we obtain general $p$-Wasserstein bounds in (multivariate) normal approximations using Stein's method.

\medskip



{
  \hypersetup{linkcolor=black}
  \tableofcontents
}

\section{Introduction}

Moderate deviations date back to
\cite{Cr38} who obtained expansions for tail probabilities for sums of independent
 random variables about the normal distribution. For independent and identically distributed (i.i.d.) random
 variables $X_1, \cdots, X_n$ with $\E X_1=0$ and $\Var(X_1)= 1$ such that $\E e^{|X_1|/b}\leq C < \infty$ for some $b>0$, it follows from \cite[Ch.8, Eq.(2.41)]{Pe75} that
 \ben{\label{eq:cramer}
 \left| \frac{P(W>x)}{P(Z>x)}-1\right|=O(1) (1+x^3)/\sqrt{n}
}
 for $0\leq x\leq O(1) n^{1/6}$, where $W=(X_1+\cdots +X_n)/ \sqrt{n}$, $Z\sim N(0,1)$ and $O(1)$ is bounded by a constant that depends on $b$ and $C$. %%depends on $c$ and $t_0$.
 The range $0\leq x\leq O(1) n^{1/6}$ and the order of the error term $O(1) (1+x^3)/\sqrt{n}$ are optimal. \cite{Vo67} obtained a multi-dimensional generalization of the result of \cite{Cr38} for sums of independent random vectors.

The classical proof of \eq{eq:cramer} depends on the conjugate method, which relies heavily on the independence assumption. A related method is by controlling the cumulants of the random vector of interest. See \cite{SaSt91}. In dimension one, \cite{ChFaSh13} developed Stein's method (\cite{St72}) to obtain Cram\'er-type moderate deviation results for dependent random variables. They needed a boundedness condition, which corresponds to assuming $|X_i|\leq b$ for an absolute constant $b$ in the above i.i.d.\ setting. Recently, \cite{LZ21} relaxed the boundedness condition and obtained results for sums of locally dependent random variables and for the combinatorial central limit theorem (CLT).

In this paper, we use a new method via $p$-Wasserstein bounds to prove Cram\'er-type moderate deviations. For two probability measures $\mu$ and $\nu$ on $\mathbb{R}^d$, $d\geq 1$, their $p$-Wasserstein distance, $p\geq 1$, is defined by
\ben{\label{eq:defpwass}
\mcl{W}_p(\mu, \nu)=\left(\inf_\pi\int_{\mathbb{R}^d\times \mathbb{R}^d} |x-y|^p \pi(dx, dy) \right)^{1/p},
}
where $|\cdot|$ denotes the Euclidean norm and $\pi$ is a measure on $\mathbb{R}^d\times \mathbb{R}^d$ with marginals $\mu$ and $\nu$.
For two random vectors $X, Y\in \mathbb{R}^d$, we also write $\mcl{W}_p(X, Y)=\mcl{W}_p(\mcl{L}(X), \mcl{L}(Y))$. The key idea of our method, explained in more detail in \cref{sec:approach}, is that for a random variable $W$ of interest and a standard normal variable $Z$, if we can show
\ben{\label{eq:pwass}
\mcl{W}_p(W, Z)\leq \frac{Cp}{\sqrt{n}}
}
for all $1\leq p\leq n^{1/3}$ and an absolute constant $C$, then, by a smoothing argument, we can recover the optimal range $0\leq x= o(n^{1/6})$ for the relative error $|P(W>x)/P(Z>x)-1|$ to vanish and obtain nearly optimal error rate $O(1)(1+x)(1+\log n+x^2)/\sqrt{n}$ subject to the logarithmic term (cf. \eq{eq:cramer}). 
This method enables us to prove moderate deviation results for dependent random variables as long as we can prove results similar to \eq{eq:pwass} and we give applications to the combinatorial CLT, Wiener chaos, and homogeneous sums in \cref{sec:application}.
The method also works for multi-dimensional approximations (cf. \cref{sec:multi,sec:local}). 

It is well known that classical Cram\'er-type moderate deviation results can be used to prove strong approximation results. See, for example, \cite[Eq.(2.6)]{KMT75} and the survey by \cite{MaZh12}. As far as we know, this is the first time that the reverse direction is explored.
It is made possible by recent advances in $p$-Wasserstein bounds.
In particular, we adapt the approach (cf. \cref{sec:proof-p-wass}) of \cite{Bo20} to obtain $p$-Wasserstein bounds for general dependent random vectors.
See \cref{t1,t3} for the results via (generalized) exchangeable pairs and \cref{thm:local3} for local dependence.

Here, we introduce some of the notations to be used in the statement of results. More notations will be introduced when they are needed in the proofs. 
$|\cdot|$ denotes the Euclidean norm, $\norm{\cdot}_{H.S.}$ denotes the Hilbert-Schmidt norm and $\norm{\cdot}_{op}$ denotes the operator norm. 
$\otimes$ denotes the tensor product.
For a random vector $X$ and $p>0$, we set $\|X\|_p:=(\E|X|^p)^{1/p}$. 
For a random matrix $Y$ and $p>0$, we set $\norm{Y}_p:=(\E\norm{Y}_{H.S.}^p)^{1/p}$.
%For the Orlicz function $\psi_\alpha: [0,\infty]\to [0,\infty)$, $\alpha>0$, defined as
For the function $\psi_\alpha: [0,\infty]\to [0,\infty)$, $\alpha>0$, defined as 
%\blue{[some people seem to require Orlicz functions to be convex]}
\be{
\psi_\alpha(x):=\exp(x^\alpha)-1,
}
the Orlicz (quasi-)norm of a random vector $X$ is defined as
\ben{\label{def:orlicz}
\norm{X}_{\psi_\alpha}:=\inf\{t>0: \E \psi_\alpha(|X|/t)\leq 1\}.
}
Unless otherwise stated, we use $c$ and $C$ to denote positive absolute constants, which may differ in different expressions.
For a positive integer $q$, we set $[q]:=\{1,\dots,q\}$. 
For a finite set $S$, we denote by $|S|$ the cardinality of $S$.

\section{Our approach}\label{sec:approach}

\subsection{$p$-Wasserstein bounds}

The first step in our approach is proving a $p$-Wasserstein bound between the distribution of the random vector of interest and a normal distribution. 
% with matching mean and covariance matrix. 
We obtain the following $p$-Wasserstein bound using exchangeable pairs.

\begin{theorem}\label{t1}
Let $(W, W')$ be an exchangeable pair of $d$-dimensional random vectors satisfying the approximate linearity condition
\ben{\label{1}
\E[W'-W|\mathcal{G}]=-\Lambda (W+R)
}
for some invertible $d\times d$ matrix $\Lambda$, $d$-dimensional random vector $R$ and $\sigma$-algebra $\mathcal{G}$ containing $\sigma(W)$. 
Assume $\Lambda=\lambda I_d$ for some $\lambda>0$ (see \cref{t3} for a more general case).
Assume that $\E|W|^p<\infty$ for some $p\geq1$ and $\E |D|^4<\infty$, where $D=W'-W$. 
Then we have 
\ban{
\mathcal{W}_p(W,Z)
&\leq C\int_0^\infty e^{-t}\left( \|R_t\|_p
+\frac{\|E\|_p}{\eta_t(p)}
+\min\left\{\frac{\sqrt {d}}{\eta_t(p)}, \frac{\|\E[D^{\otimes 2}|D|^21_{\{|D|\leq\eta_t(p)\}}\mid \mcl{G}]\|_p}{\lambda \eta^{3}_t(p)}\right\} \right)dt\label{abst-est}
\\
&\leq C\left( \int_0^\infty e^{-t}\|R_t\|_pdt
+\sqrt{p}\|E\|_p
+p d^{1/4}\sqrt{\frac{\|\E[D^{\otimes2}|D|^2\mid\mcl{G}]\|_p}{\lambda}} \right),
\label{027}
}
where $Z\sim N(0,I_d)$ is a $d$-dimensional standard Gaussian vector, 
$\eta_t(p):=\sqrt{(e^{2t}-1)/p}$,
\ben{\label{def:rt}
R_t:=R+\E[\Lambda^{-1}D1_{\{|D|>\eta_t(p)\}}\mid\mcl{G}],\qquad
E:=\frac{1}{2}\E[\Lambda^{-1}D\otimes D\mid\mathcal{G}]-I_d,
}
and $C$ is an absolute constant.

\end{theorem}

We defer the proof of \cref{t1} to \cref{sec:proof-p-wass}. The proof heavily relies on the techniques developed in \cite{Bo20}. However, the concrete error bound
%using a symmetry argument for exchangeable pairs 
and the explicit dependence on $p$ that yields optimal moderate deviation results are new.
Such $p$-Wasserstein bounds can also be obtained under other dependency structures, e.g.,  generalized exchangeable pairs (cf. \cref{t3}) and local dependence (cf. \cref{thm:local3}). 
%In \cref{sec:projp}, we obtain a bound for the projected $p$-Wasserstein distance. 

Next, we give a corollary of \cref{t1} in dimension one.

\begin{corollary}[The case $d=1$]\label{c1}
Under the setting of \cref{t1}, assume $d=1$.
We have
\ben{\label{020}
\mcl{W}_p(W,Z)\leq C \left(\norm{R}_p +\sqrt{p} \norm{E}_p+ p \sqrt{\lambda^{-1}\norm{\E[D^4|\mcl{G}]}_p} \right).
}
\end{corollary}

\begin{proof}[Proof of \cref{c1}]
The corollary is a direct consequence of \cref{t1} except that we bound the additional term from $R_t$ by
\bes{
&C\int_0^{\infty} e^{-t} \norm{\E[\lambda^{-1} D1_{\{|D|>\sqrt{(e^{2t}-1)/p}\}}|\mcl{G}]}_p dt\\
\leq& C\sqrt{p}\lambda^{-1}\norm{\E[D^2|\mcl{G}]}_p\int_0^\eps \frac{e^{-t}}{\sqrt{e^{2t}-1}}  dt+Cp^{3/2}\lambda^{-1}\norm{\E[D^4|\mcl{G}]}_p\int_\eps^\infty  \frac{e^{-t}}{(e^{2t}-1)^{3/2}}dt\\
\leq& C\sqrt{p}\norm{E}_p+ C\sqrt{p}\int_0^\eps \frac{e^{-t}}{\sqrt{e^{2t}-1}} dt+Cp^{3/2}\lambda^{-1}\norm{\E[D^4|\mcl{G}]}_p\int_\eps^\infty  \frac{e^{-t} }{(e^{2t}-1)^{3/2}}dt ,
}
which is bounded by the summation of second and third error terms in \eq{020} by choosing an appropriate $\eps$ as at the end of the proof of \cref{t1}.
\end{proof}

\subsection{From $p$-Wasserstein bounds to moderate deviations in dimension one}

The next step in our approach is proving moderate deviation results using $p$-Wasserstein bounds. The following result enables such transition in dimension one. In most of our applications of the following result, $r_0=\alpha_1=1$. See \cref{thm:multiMD} for a multi-dimensional result. 

\begin{theorem}\label{t4}
Let $W$ be a one-dimensional random variable and $Z$ a standard normal variable. Suppose that 
\be{
\mcl{W}_p(W, Z)\leq A \max_{1\leq r\leq r_0}p^{\alpha_r} \Delta_r \ \text{for}\ 1\leq p\leq  p_0
}
with some constants $\alpha_1,\dots,\alpha_{r_0}\geq 0$, $A>0$, $p_0\geq1$ and $\Delta_1,\dots,\Delta_{r_0}>0$. 
Suppose also that $\ol\Delta:=\max_{1\leq r\leq r_0}\Delta_r$ satisfies $|\log\ol\Delta|\leq p_0/2$. 
Then there exists a positive constant $C$ depending only on $\alpha_1,\dots,\alpha_{r_0}$ and $A$ such that 
\ben{\label{021}
\left|\frac{P(W>x)}{P(Z>x)}-1\right|\leq C(1+x)\left\{\max_{1\leq r\leq r_0}(|\log\ol\Delta|+x^2)^{\alpha_r}\Delta_r+\ol\Delta\right\}
}
for all $0\leq x\leq \sqrt{p_0}\wedge\min_{r=1,\dots,r_0}\Delta_r^{-1/(2\alpha_r+1)}$.
\end{theorem} 

We remark that because $\mcl{W}_p(W,Z)$ increases in $p$, to apply \cref{t4}, we only need to verify the upper bound on $\mcl{W}_p(W,Z)$ for sufficiently large $p$, for example, for $p\geq 2$ in our applications.

\begin{proof}[Proof of Theorem \ref{t4}]
In this proof, we use $C$ to denote positive constant, which depends only on $\alpha_1,\dots,\alpha_{r_0}$ and $A$ and may be different in different expressions. 
First we prove the claim when $\ol\Delta<1/e$. 
Set 
\be{
p=\log(1/\ol\Delta)+\frac{x^2}{2}, \quad \eps=A \max_{1\leq r\leq r_0}p^{\alpha_r} \Delta_r e.
} 
Because $|\log \ol\Delta|\leq p_0/2$ and $x\leq \sqrt{p_0}$, we have $p\leq p_0$.

Without loss of generality, we may take $W$ and $Z$ so that $\|W-Z\|_p=\mcl{W}_p(W, Z)$. Then
\ba{
P(W>x)
&\leq P(Z>x-\eps)+P(|W-Z|>\eps)\\
&=P(Z>x)+P(x-\eps<Z\leq x)+P(|W-Z|>\eps).
}
Let $\phi(\cdot)$ denote the standard normal density function. 
Since
\ba{
P(x-\eps<Z\leq x)
=\int_{x-\eps}^x\phi(z)dz\leq\phi((x-\eps)\vee0)\eps
}
and
\ba{
P(|W-Z|>\eps)\leq (A \max_{1\leq r\leq r_0}p^{\alpha_r} \Delta_r/\eps)^p=e^{-p}=\ol\Delta e^{-x^2/2},
}
we obtain
\ba{
P(W>x)
&\leq P(Z>x)+\phi((x-\eps)\vee0)\eps+\ol\Delta e^{-x^2/2}.
}
Similarly, we deduce
\ba{
P(Z>x)
&=P(Z>x+\eps)+P(x<Z\leq x+\eps)\\
&\leq P(W>x)+P(|W-Z|>\eps)+P(x<Z\leq x+\eps)\\
&\leq P(W>x)+\phi(x)\eps+\ol\Delta e^{-x^2/2}.
}
Consequently, we obtain
\ben{\label{022}
|P(W>x)-P(Z>x)|\leq\phi((x-\eps)\vee0)\eps+\ol\Delta e^{-x^2/2}.
}
Observe that
\besn{\label{023}
\eps&\leq C\max_{1\leq r\leq r_0}\Delta_r(\{\log(1/\ol\Delta)\}^{\alpha_r}+x^{2\alpha_r})\\
&\leq C\max_{1\leq r\leq r_0}\Delta_r(\{\log(1/\Delta_r)\}^{\alpha_r}+\Delta_r^{-2\alpha_r/(2\alpha_r+1)})\\
&\leq C\max_{1\leq r\leq r_0}\Delta_r^{1-2\alpha_r/(2\alpha_r+1)}
=C\max_{1\leq r\leq r_0}\Delta_r^{1/(2\alpha_r+1)}.
}
If $x\geq\eps$, we have
\ba{
\phi((x-\eps)\vee0)\leq\phi(x)e^{x\eps}\leq C\phi(x).
}
Birnbaum's inequality yields
\ben{\label{024}
\frac{\phi(x)}{P(Z>x)}\leq\frac{2}{\sqrt{4+x^2}-x}=\frac{\sqrt{4+x^2}+x}{2}\leq1+x.
}
Hence
\ba{
\left|\frac{P(W>x)}{P(Z>x)}-1\right|\leq C(1+x)(\eps+\ol\Delta)
\leq C(1+x)\{\max_{1\leq r\leq r_0}(|\log\ol\Delta|+x^2)^{\alpha_r}\Delta_r+\ol\Delta\}.
}
If $x\leq\eps$, we have by \eqref{023} and \eqref{024}
\ba{
\frac{1}{P(Z>x)}\leq\sqrt{2\pi}(1+\eps)e^{\eps^2/2}\leq C.
}
Combining this with \eqref{022} gives
\ba{
\left|\frac{P(W>x)}{P(Z>x)}-1\right|
\leq C(\eps+\ol\Delta)
\leq C(1+x)\{\max_{1\leq r\leq r_0}(|\log\ol\Delta|+x^2)^{\alpha_r}\Delta_r+\ol\Delta\}.
}
So we complete the proof of \eqref{021}.

It remains to prove \eqref{021} when $\ol\Delta\geq 1/e$. In this case, we have $x\leq e$ and thus 
\[
\frac{1}{P(Z>x)}\leq (1+e)\sqrt{2\pi e^{e^2}}
\]
by \eqref{024}. Hence \eqref{021} holds with $C\geq e(1+e)\sqrt{2\pi e^{e^2}}$.
\end{proof}

\subsection{Sums of independent random variables}\label{subsec:indep}

Finally, we illustrate our approach in the classical setting of sums of independent random variables.

Let $W=\frac{1}{\sqrt{n}}\sum_{i=1}^n X_i$, where $\{X_1,\dots, X_n\}$ are independent with $\E X_i=0$ for all $i$ and $\Var(W)=1$.
Suppose 
\ben{\label{035}
b:=\max_{1\leq i\leq n} \norm{X_i}_{\psi_1},
}
where $\norm{\cdot}_{\psi_1}$ is the Orlicz norm defined in \eq{def:orlicz}. This is equivalent to $b$ being the smallest positive constant such that $\E e^{|X_i|/b}\leq 2$ for all $i$.
Let $Z\sim N(0,1)$. To apply \cref{t1}, we construct an exchangeable pair (which is standard in Stein's method) as follows.
Let $I$ be a uniform random index from $\{1,\dots, n\}$ and independent of everything else.
Let $\{X_1',\dots, X_n'\}$ be an independent copy of $\{X_1,\dots, X_n\}$.
Let 
\be{
W'=W-\frac{1}{\sqrt{n}}X_I+\frac{1}{\sqrt{n}}X_I'=:W+D.
}
Let $\mcl{G}=\sigma(X_1,\dots, X_n)$.
It is straightforward to verify that
\be{
\E(D|\mcl{G})=-\frac{W}{n}.
}
Therefore, we can apply \cref{t1} with $R=0$ and $\lambda=1/n$ to bound $\mcl{W}_p(W, Z)$.

We have
\ba{
\|R_t\|_p\leq\norm{\sum_{i=1}^nY_i1_{\{|Y_i|>\eta_t(p)\}}}_p,\qquad
\norm{E}_p\leq \norm{\sum_{i=1}^n(Y_i^2-\E Y_i^2)}_p,
}
and
\ba{
\lambda^{-1} \norm{\E[D^41_{\{|D|\leq\eta_t(p)\}}|\mcl{G}]}_p
&\leq\norm{\sum_{i=1}^nY_i^41_{\{|Y_i|\leq\eta_t(p)\}}}_p\\
&\leq\sum_{i=1}^n\E Y_i^4+\norm{\sum_{i=1}^n(Y_i^41_{\{|Y_i|\leq\eta_t(p)\}}-\E[Y_i^41_{\{|Y_i|\leq\eta_t(p)\}}])}_p,
}
where $Y_i=(X_i'-X_i)/\sqrt n$. 
We employ the following lemma to bound these quantities. See \cite[Theorem 3.1 and Remark 3.1]{KuCh20} for a related result in dimension one and the literature on such concentration inequalities for sub-Weibull distributions.
\begin{lemma}\label{lem:weibull}
Let $\xi_1,\dots,\xi_n$ be independent random vectors in $\mathbb R^d$ such that $\max_{i=1,\dots,n}\|\xi_i\|_{\psi_\alpha}\leq M$ for some $M>0$ and $\alpha\in(0,1]$. 
Then, there is a constant $C_\alpha>0$ depending only on $\alpha$ such that, for any $p\geq2$ and any real numbers $a_1,\dots,a_n$, 
\[
\left\|\sum_{i=1}^na_i(\xi_i-\E\xi_i)\right\|_p\leq C_\alpha M\left(\sqrt{p\sum_{i=1}^na_i^2}+p^{1/\alpha}\max_{1\leq i\leq n}|a_i|\right).
\]
\end{lemma}

\begin{proof}
First, by symmetrization, we have
\[
\left\|\sum_{i=1}^na_i(\xi_i-\E\xi_i)\right\|_p\leq2\left\|\sum_{i=1}^na_i\epsilon_i\xi_i\right\|_p,
\]
where $\epsilon_1,\dots,\epsilon_n$ are i.i.d.~Rademacher variables independent of everything else. 
Next, let $\zeta$ be a symmetric random variable such that $P(|\zeta|> t)=e^{-t^\alpha}$ for all $t\geq0$. Then we have $P(|\epsilon_i\xi_i|> t)\leq2\exp(-(t/M)^\alpha)=2P(M|\zeta|> t)$ for all $i=1,\dots,n$ and $t>0$. Thus, by Theorem 3.2.2 in \cite{KwWo92}, 
\[
P\left(\left|\sum_{i=1}^na_i\epsilon_i\xi_i\right|>t\right)\leq 48P\left(6 M\left|\sum_{i=1}^na_i\zeta_i\right|>t\right)
\]
for any $t>0$, where $\zeta_1,\dots,\zeta_n$ are independent copies of $\zeta$. This particularly implies that
\[
\left\|\sum_{i=1}^na_i\epsilon_i\xi_i\right\|_p\leq CM\left\|\sum_{i=1}^na_i\zeta_i\right\|_p.
\]
Finally, by Corollary 1.2 in \cite{Bo15},
\[
\left\|\sum_{i=1}^na_i\zeta_i\right\|_p\leq L_\alpha\left(\sqrt{p\sum_{i=1}^na_i^2}+p^{1/\alpha}\max_{i=1,\dots,n}|a_i|\right),
\]
where $L_\alpha>0$ depends only on $\alpha$. All together, we obtain the desired result. 
\end{proof}

Now, for any $r\geq1$, from $b:=\max_{1\leq i\leq n}\norm{X_i}_{\psi_1}$ and the equivalence of sub-exponential tails and linear growth of $L^r$-norms (cf. \cite[Proposition 2.7.1]{Ve18}),
\ba{
\|Y_i1_{\{|Y_i|>\eta_t(p)\}}\|_r&\leq\eta_t^{-1}(p)(\E Y_i^{2r})^{1/r}\leq Cr^2\eta_t^{-1}(p)b^2/n,&
\|Y_i^2\|_r&\leq Cr^2 b^2/n,
}
and
\[
\|Y_i^41_{\{|Y_i|\leq\eta_t(p)\}}\|_r\leq\eta_t^2(p)\|Y_i^2\|_r\leq Cr^2\eta_t^2(p) b^2/n.
\]
Hence, $\|Y_i1_{\{|Y_i|>\eta_t(p)\}}\|_{\psi_{1/2}}\leq C\eta_t^{-1}(p)b^2/n$, $\|Y_i^2\|_{\psi_{1/2}}\leq Cb^2/n$ and $\|Y_i^41_{\{|Y_i|\leq\eta_t(p)\}}\|_{\psi_{1/2}}\leq C\eta_t^2(p) b^2/n$. So we obtain by \cref{lem:weibull}, for $p\geq 2$,
\ba{
\int_0^\infty e^{-t}\|R_t\|_pdt&\leq C\frac{\sqrt{np}+p^2}{n}\int_0^\infty\frac{e^{-t}\sqrt p}{\sqrt{e^{2t}-1}}b^2dt\leq C(\frac{p}{\sqrt n}+\frac{p^{5/2}}{n})b^2,\\
\sqrt p\norm{E}_p&\leq C(\frac{p}{\sqrt n}+\frac{p^{5/2}}{n})b^2,
}
and
\ba{
&\int_0^\infty e^{-t}\min\left\{\frac{1}{\eta_t(p)},\frac{\|\E[D^41_{\{|D|\leq\eta_t(p)\}}\mid \mcl{G}]\|_p}{\lambda \eta_t^3(p)}\right\}dt\\
&\leq\int_0^\infty e^{-t}\min\left\{\frac{\sqrt{p}}{\sqrt{e^{2t}-1}},\frac{Cp^{3/2}b^4}{n(e^{2t}-1)^{3/2}}\right\}dt
+C\int_0^\infty e^{-t}\frac{p/\sqrt n+p^{5/2}/n}{\sqrt{e^{2t}-1}}b^2dt\\
&\leq C(\frac{p}{\sqrt n}+\frac{p^{5/2}}{n})b^2.
}
Here, we evaluate the integrals as in the proof of \cref{t1}. 
Consequently, from \eq{abst-est},
\ben{\label{034}
\mcl{W}_p(W,Z)\leq C(\frac{p}{\sqrt n}+\frac{p^{5/2}}{n}) b^2,    \quad \forall\ p\geq 2.
}
Note that $\Var(W)=1\leq C b^2$. Therefore, we can apply \cref{t4} with $r_0=\alpha_1=1$, $\Delta_1=b^2/\sqrt{n}$ and $p_0=(\sqrt{n}/b^2)^{2/3}$, which implies that:
\begin{corollary}\label{cor:indep}
Let $W=\frac{1}{\sqrt{n}}\sum_{i=1}^n X_i$, where $\{X_1,\dots, X_n\}$ are independent with $\E X_i=0$ for all $i$, $\Var(W)=1$ and $b:=\max_{1\leq i\leq n} \norm{X_i}_{\psi_1}$.
Then there exist positive absolute constants $c$ and $C$ such that
\be{
\left|\frac{P(W>x)}{P(Z>x)}-1   \right|\leq C \frac{(1+x)(1+|\log (n/b^4)|+x^2)b^2}{\sqrt n}
}
for all $0\leq x\leq (n/b^4)^{1/6}$ and $\frac{b^2}{\sqrt{n}}\leq c$.
\end{corollary}

\begin{remark}\label{rem:log}
\cref{cor:indep} recovers the bound \eq{eq:cramer} when $x\geq \sqrt{\log n}$. 
It seems impossible to avoid the $\log n$ term using our approach because such a term will appear even if we only aim to bound the Kolmogorov distance using $p$-Wasserstein bounds and a smoothing argument.
\end{remark}

An inspection of the proof shows that we can replace the range of $x$ by $0\leq x\leq c_0(n/b^4)^{1/6}$ with any absolute constant $c_0$ (the constant $C$ will then depend on $c_0$). Because our primary interests are vanishing relative errors and the order of magnitude, we will not worry about such absolute constants and state our results in a form that we find convenient.

\section{Applications to Cram\'er-type moderate deviations in dimension one}\label{sec:application}

In this section, we provide more applications in dimension one, including the combinatorial CLT, Wiener chaos and homogeneous sums.

\subsection{Combinatorial CLT}

Let $\mathbb{X}=\{X_{ij}, 1\leq i,j\leq n\}$ be an $n\times n$ array of independent random variables where $n\geq 2$, $\E X_{ij}=c_{ij}$, $\Var(X_{ij})=\sigma_{ij}^2\geq 0$. Assume without loss of generality that (cf. Remark 1.3 of \cite{CF15})
\be{
c_{i\cdot}=c_{\cdot j}=0,
}
where $c_{i\cdot}=\sum_{j=1}^n c_{ij}/n$, $c_{\cdot j}=\sum_{i=1}^n c_{ij}/n$.
Let $\pi$ be a uniform random permutation of $\{1,\dots, n\}$, independent of $\mathbb{X}$, and let 
\ben{\label{007}
S=\sum_{i=1}^n X_{i\pi(i)}.
}
It is known that $\E(S)=0$ and (cf. Theorem 1.1 of \cite{CF15})
\ben{\label{004}
B_n^2:=\Var(S)=\frac{1}{n-1}\sum_{i,j=1}^n c_{ij}^2+\frac{1}{n}\sum_{i,j=1}^n \sigma_{ij}^2,
}
\ben{\label{003}
\sup_{x\in \mathbb{R}}|P(W\leq x)-P(Z\leq x)|\leq \frac{C}{n}\sum_{i,j=1}^n \E\big|\frac{X_{ij}}{B_n}\big|^3,
}
where 
\ben{\label{008}
W=\frac{S}{B_n},
}
and $Z\sim N(0,1)$.
Cram\'er-type moderate deviation results were obtained by \cite{Fr22} and \cite{LZ21}.
Here, we use our approach to prove a version of such moderate deviation results.

\begin{theorem}\label{t5}
Under the above setting, assume
\ben{\label{005}
b:=\max_{1\leq i,j\leq n}\norm{X_{ij}}_{\psi_1}<\infty.
}
Then
there exist positive absolute constants $c$ and $C$ such that, for
\be{
\Delta:=\frac{n^{1/2}b^2}{B_n^2}\leq c,\quad     0\leq x\leq  \Delta^{-1/3},
}
we have
\be{
\left|\frac{P(W>x)}{P(Z>x)}-1   \right|\leq C (1+x) (1+|\log \Delta |+x^2) \Delta.
}
\end{theorem}

%\begin{remark}
%The appearance of the ratio $b/B_n$ in the statement of \cref{t1} is natural in view of \eq{004} and \eq{005}. 
%One might expect a bound in terms of $\frac{n b^3}{B_n^3}$ in view of the Berry-Esseen bound \eq{003}. However, we use the exchangeable pair approach and hence the resulting fourth moment bound is $\sqrt{\frac{n b^4}{B_n^4}}=\frac{n^{1/2}b^2}{B_n^2}$, which is in fact better than $\frac{n b^3}{B_n^3}$ because $B_n^2\leq C n b^2$ from \eq{004}.
%\end{remark}

\begin{remark}
Because \cite{Fr22}'s result is stated under a different condition and he did not provide a rate of convergence, here we only compare our result with that in \cite{LZ21}. In our notation, their bound is $C (1+x^3) \frac{n^{1/2}b^2}{B_n^2}\cdot (\frac{n^{1/2}b}{B_n})^5$. From \eq{004}, we have $B_n^2\leq C n b^2$ and $B_n^2$ in general can be of smaller order than $n b^2$. Therefore, except for the logarithmic term in the error rate, our bound is in general better.
\end{remark}

We prove \cref{t1} via the following $p$-Wasserstein bound between $W$ and $Z$.

\begin{proposition}\label{p2}
Under the assumptions of \cref{t5},
there exists a positive absolute constant $C$ such that
\ben{\label{006}
\mcl{W}_p(W, Z)\leq C(\frac{ p \sqrt{n} }{B_n^2}+\frac{p^{5/2}}{B_n^2} )b^2\quad \forall\ p\geq 2.
}
\end{proposition}

In the following, we prove \cref{t5} using \cref{p2}. The proof of \cref{p2} is deferred to \cref{appendix:comb}.

\begin{proof}[Proof of \cref{t5}]
We apply \cref{t4} with $r_0=\alpha_1=1$ and
\be{
\Delta_1:=\Delta=\frac{n^{1/2}b^2}{B_n^2}, \ p_0=\Delta_1^{-2/3}=(\frac{B_n^2}{n^{1/2} b^2})^{2/3}.
}
The conditions in \cref{t4} are satisfied by choosing $c$ in the statement of \cref{t5} to be sufficiently small and using $B_n^2\leq C n b^2$ from \eq{004} to reduce the bound \eq{006} to $Cpn^{1/2}b^2/B_n^2$ for $2\leq p\leq p_0$.
\end{proof}

\subsection{Moderate deviations on Wiener chaos}\label{sec:wiener}

Let $X$ be an isonormal Gaussian process over a real separable Hilbert space $\mathfrak{H}$. Given an integer $q\geq2$, we consider the $q$-th multiple Wiener--It\^ointegral $W=I_q(f)$ of $f\in\mathfrak{H}^{\odot q}$ with respect to $X$. Here, $\mathfrak{H}^{\odot q}$ denotes the $q$-th symmetric tensor power of $\mathfrak{H}$. 
Here and below, we use standard concepts and notation in Malliavin calculus. We refer to \cite{NoPe12} for all unexplained notation. 

We assume $\Var[W]=q!\|f\|_{\mf H^{\otimes q}}^2=1$ for simplicity. The celebrated fourth moment theorem states that (cf.~Theorem 5.2.6 in \cite{NoPe12})
\[
\sup_{x\in\mathbb R}|P(W\leq x)-P(Z\leq x)|\leq\sqrt{\frac{q-1}{3q}(\E W^4-3)},
\] 
where $Z\sim N(0,1)$. 
\cite{ScTh16} obtained a corresponding Cram\'er-type moderate deviation result. 
Here, we use our approach to prove a version of such moderate deviation results.

To state our result, we need to introduce mixed injective norms of elements in $\mf H^{\odot q}$ which were originally introduced in \cite{La06} (see also \cite{Le11}). 
%We write $[q]=\{1,\dots,q\}$. 
A \textit{partition} of $[q]$ is a collection of nonempty disjoint sets $\{J_1,\dots,J_k\}$ such that $[q]=\bigcup_{l=1}^kJ_l$. We denote by $\Pi_q$ the set of partitions of $[q]$. For any $h\in\mf H^{\odot q}$ and $\mcl J=\{J_1,\dots,J_k\}\in\Pi_q$, define
\[
\|h\|_{\mcl J}:=\sup\{\langle h,u_1\otimes\cdots\otimes u_k\rangle_{\mf H^{\otimes q}}:u_l\in \mf H^{\otimes|J_l|},\|u_l\|_{\mf H^{\otimes|J_l|}}\leq1,l=1,\dots,k\}.
\]
In the remainder of this section, $C_q$ denotes a positive constant, which depends only on $q$ and may be different in different expressions. 

\begin{theorem}\label{thm:wiener}
Under the above setting, let 
\[
\ol\Delta:=\max_{r\in[q-1]}\max_{\mcl J\in\Pi_{2q-2r}}\|f\wt\otimes_rf\|_{\mcl J},
\]
where $f\wt\otimes_rf$ denotes the symmetrization of $f\otimes_rf$ with $\otimes_r$ the $r$-th contraction operator (cf.~\cite[Eq.\ (B.3.1)$\&$(B.4.4)]{NoPe12}).  
If 
\ben{\label{wiener-x-range}
0\leq x\leq\min_{r\in[q-1]}\min_{\mcl J\in\Pi_{2q-2r}}\|f\wt\otimes_rf\|_{\mcl J}^{-1/(|\mcl J|+2)},
}
then
\ben{\label{eq:wiener}
\left|\frac{P(W>x)}{P(Z>x)}-1\right|\leq C_q(1+x)\left\{\max_{r\in[q-1]}\max_{\mcl J\in\Pi_{2q-2r}}(|\log\ol\Delta|+x^2)^{\frac{1+|\mcl J|}{2}}\|f\wt\otimes_rf\|_{\mcl J}+\ol\Delta\right\}.
}
\end{theorem}
The proof of \cref{thm:wiener} is deferred to \cref{appendix:wiener}. 
\begin{remark}[Optimality on the range of $x$]
Condition \eqref{wiener-x-range} is sharp when $q=2$. To see this, assume that $\mf H$ is infinite-dimensional and let $(e_i)_{i=1}^\infty$ be an orthonormal basis of $\mf H$. Taking $f=\frac{1}{\sqrt{2n}}\sum_{i=1}^ne_i^{\otimes 2}$, we obtain $W=\frac{1}{\sqrt{2n}}\sum_{i=1}^n(X(e_i)^2-1)$ (cf.~Theorem 2.7.7 in \cite{NoPe12}). Since $X(e_i)$ are i.i.d.~standard normal variables, $W$ is a sum of i.i.d.~random variables with the centered $\chi^2$-distribution with 1 degree of freedom. 
Meanwhile, since $|\langle \sum_{i=1}^ne_i^{\otimes 2},u_1\otimes u_2\rangle_{\mf H^{\otimes2}}|\leq\|u_1\|_{\mf H}\|u_2\|_{\mf H}$ for any $u_1,u_2\in\mf H$ by Bessel's inequality and the equality can be attained,
\ba{
\|f\wt\otimes_1f\|_{\{1\},\{1\}}=\frac{1}{2n}\left\|\sum_{i=1}^ne_i^{\otimes 2}\right\|_{\{1\},\{1\}}=\frac{1}{2n}.
}
Also, 
\ba{
\|f\wt\otimes_1f\|_{\{1,2\}}=\frac{1}{2n}\left\|\sum_{i=1}^ne_i^{\otimes 2}\right\|_{\mf H^{\otimes2}}=\frac{1}{2\sqrt n}.
}
Thus, \eqref{wiener-x-range} is rewritten as $0\leq x\leq\min\{(2n)^{1/4},(4n)^{1/6}\}$. In view of Theorem 2 in \cite[Chapter VIII]{Pe75}, this condition is sharp to obtain a bound like \eqref{eq:wiener}. 

When $q>2$, it is unclear whether \eqref{wiener-x-range} is sharp or not. 
By an analogous argument to the above but using Theorem 2 in \cite{Li61}, we can show that $x$ must satisfy $x=O(\ul\Delta^{-1/(2q-2)-\eps})$ for any $\eps>0$, where $\ul\Delta:=\|f\wt\otimes_1f\|_{\{1\},\dots,\{2q-2\}}$. However, \eqref{wiener-x-range} requires at least $x=O(\ul\Delta^{-1/(2q)})$. 
\end{remark}
%\blue{Moved the above remark to this place following your suggestion}

Next, we make connections to the fourth moment theorem.
For any $\mcl J\in\Pi_{2q-2r}$ with $r\in[q-1]$, we have $|\mcl J|\leq2q-2r\leq2q-2$ and
\[
\|f\wt\otimes_rf\|_{\mcl J}\leq\|f\wt\otimes_rf\|_{\mf H^{\otimes(2q-2r)}}\leq\|f\otimes_rf\|_{\mf H^{\otimes(2q-2r)}}\leq\norm{f}_{\mf H^{\otimes q}}^2= 1/q!,
\]
where
the first inequality is from $\|h\|_{\mcl J}\leq \|h\|_{\mf H^{\otimes(2q-2r)}}$ for any $h\in \mf H^{\odot (2q-2r)}$, the second inequality is from the definition of symmetrization and the triangle inequality,
the third inequality follows by the Cauchy--Schwarz inequality.
Therefore, noting that the function $(0,1)\ni \delta \mapsto \delta (y+|\log \delta |)^{(2q-1)/2}\in(0,\infty)$ is increasing for any $y\geq (2q-1)/2$, we particularly obtain by \cref{thm:wiener} 
\besn{\label{wiener-simple}
\left|\frac{P(W>x)}{P(Z>x)}-1\right|&\leq C_q(1+x)(1+|\log\ol\Delta|+x^2)^{(2q-1)/2}\ol\Delta   \\
&\leq C_q(1+x)(1+|\log\Delta|+x^2)^{(2q-1)/2}\Delta
}
%\red{the last inequality is not exactly by the fact that $x(|\log x|\vee 1)$ is increasing?} 
%\blue{Yes, I overlooked the power of log.}
for all $0\leq x\leq\Delta^{-1/(2q)}$, where
\[
\Delta:=\max_{r\in[q-1]}\|f\otimes_rf\|_{\mathfrak H^{\otimes(2q-2r)}}.
\]
%This and Eq.(5.2.6) of \cite{NoPe12} give 
From \cite[Eq.\ (5.2.6)]{NoPe12}, we have $\Delta\leq C_q\sqrt{\E W^4-3}$. Therefore, we obtain a Cram\'er-type moderate deviation result for the fourth moment theorem:
\begin{corollary}
Under the above setting,
\be{%\label{eq:wiener}
\left|\frac{P(W>x)}{P(Z>x)}-1\right|\leq C_q(1+x)(1+|\log\kappa_4(W)|+x^2)^{(2q-1)/2}\sqrt{\kappa_4(W)}
}
for all $0\leq x\leq\kappa_4(W)^{-1/(4q)}$, where $\kappa_4(W)=\E W^4-3$ is the fourth cumulant of $W$. 
\end{corollary}

\begin{remark}[Comparison with \cite{ScTh16}]
Using the method of cumulants, \cite{ScTh16} give in their Theorem 5 a Cram\'er-type moderate deviation result for multiple Wiener-It\^ointegrals in the following form: Let
\[
\alpha(q):=\begin{cases}
(q+2)/(3q+2) & \text{if $q$ is even},\\
(q^2-q-1)/(q(3q-5)) & \text{if $q$ is odd}.
\end{cases}
\]
Then, there are constants $c_0,c_1,c_2>0$ depending only on $q$ such that, for $\Delta^{-\alpha(q)}\geq c_0$ and $0\leq x\leq c_1\Delta^{-\alpha(q)/(q-1)}$,
\ben{\label{eq:scth}
\left|\log\frac{P(W> x)}{P(Z>x)}\right|\leq c_2(1+x^3)\Delta^{\alpha(q)/(q-1)}.
}
%Since $|e^y-1|\leq 2|y|$ for all $|y|\leq1$, this inequality implies that 
%\ben{\label{eq:scth}
%\left|\frac{P(W> x)}{P(Z>x)}-1\right|\leq C_q(1+x^3)\Delta^{\alpha(q)/(q-1)}
%}
%for $0\leq x\leq c_1\Delta^{-\alpha(q)/(q-1)}$. 
On the other hand, by the inequality $|\log(1+y)|\leq2|y|$ for $|y|\leq1/2$, our simplified bound \eqref{wiener-simple} implies that there are constants $c'_0,c'_1,c'_2>0$ depending only on $q$ such that, for $\Delta\leq c'_0$ and $0\leq x\leq c'_1\Delta^{-1/(2q)}$,
\[
\left|\log\frac{P(W> x)}{P(Z>x)}\right|\leq c_2'(1+x)(1+|\log\Delta|+x^2)^{(2q-1)/2}\Delta.
\]
%We compare \eqref{eq:scth} with our simplified bound \eqref{wiener-simple}. We focus on the case $\Delta\leq1$ because otherwise the bounds become trivial in view of the proof of \cref{t4}.  
We compare this bound with \eqref{eq:scth}. 
Note that $\Delta\leq1$. 
Then, since we can easily check that $\alpha(q)+1/(2q)<1/2$ if and only if $q\geq5$, Theorem 5 in \cite{ScTh16} imposes a weaker condition on $x$ than ours when $q<5$. However, note that we need $x^3\Delta^{\alpha(q)/(q-1)}=o(1)$ to get a vanishing bound in \eqref{eq:scth}. This condition is always stronger than our condition $x^{2q}\Delta=o(1)$ because $\alpha(q)\leq1/2$. Moreover, under the condition $x^3\Delta^{\alpha(q)/(q-1)}=o(1)$, we always have $x^{2q}\Delta=o(x^3\Delta^{\alpha(q)/(q-1)})$ since 
\[
\frac{\alpha(q)}{3(q-1)}\leq\frac{q-1-\alpha(q)}{(2q-3)(q-1)}.
\] 
So our bound always gives a better rate of convergence to 0 than \eqref{eq:scth}.
\end{remark}

%\red{Is it better to move the following remark to immediately below Theorem 3.2?}

\subsection{Homogeneous sums}

Let $X_1,\dots,X_n$ be independent random variables with mean 0 and variance 1. 
We consider a multilinear homogeneous sum of these variables, i.e.~a random variable of the form
\[
W=\sum_{i_1,\dots,i_q=1}^nf(i_1,\dots,i_q)X_{i_1}\cdots X_{i_q},
\]
where $q\geq2$ and $f:[n]^q\to\mathbb R$ is a symmetric function with vanishing diagonals (i.e.~$f(i_1,\dots,i_q)=0$ whenever $i_r=i_s$ for some indices $r\ne s$). $W$ has mean 0 by assumption. For simplicity, we assume that $W$ has variance 1, i.e.
\[
\Var[W]=q!\sum_{i_1,\dots,i_q=1}^nf(i_1,\dots,i_q)^2=1.
\]
$W$ is a prominent example of degenerate $U$-statistics of order $q$, and limit theorems for such statistics have been well-studied in the literature. 
In particular, the prominent work of \cite{deJo90} established the following sufficient conditions for the asymptotic normality: $W$ converges in law to $N(0,1)$ if the following conditions are satisfied:
\begin{enumerate}[label=(\roman*)]

\item The fourth cumulant of $W$ converges to 0. That is, $\E W^4$ converges to 3. 

\item The maximal influence
\[
\mcl M(f):=\max_{i\in[n]}\sum_{i_2,\dots,i_q=1}^nf(i,i_2,\dots,i_q)^2
\]
converges to 0.

\end{enumerate}
Corresponding absolute error bounds were investigated in e.g.~\cite{NoPeRe10,DoPe17} and \cite{FaKo22}. For example, Corollary 2.1 in \cite{FaKo22} gives the following optimal 1-Wasserstein bound (throughout this section, $C_q$ denotes a constant, which depends only on $q$ and may be different in different expressions):
\[
\mcl W_1(W,Z)\leq C_q\sqrt{|\E W^4-3|+\left(\max_{i\in[n]}\E X_i^4\right)^q\mcl M(f)},
\]
where $Z\sim N(0,1)$. 
However, to our knowledge, no relative error bound for this type of CLT is available in the literature (but see \cref{rem:homo2}). Using our approach, we can obtain such a bound as follows:
\begin{theorem}\label{thm:dejong}
Under the above setting, assume that there exits a constant $K\geq1$ such that $\|X_i\|_{\psi_2}\leq K$ for all $i\in[n]$. Let
\[
M:=\max_{i\in[n]}\E X_i^4,\quad
\Delta:=K^{2q}\sqrt{|\E W^4-3|+M^q\mcl M(f)(1\vee|\log\mcl M(f)|^{2q-2})},
\]
and assume $\Delta<1$. 
%If 
%\ben{\label{dejong-x-range}
%0\leq x\leq\Delta^{-\frac{1}{2q+1}},
%}
Then, for all $0\leq x\leq\Delta^{-\frac{1}{2q+1}}$,
\ben{\label{eq:dejong}
\left|\frac{P(W>x)}{P(Z>x)}-1\right|
\leq C_q(1+x)(|\log\Delta|+x^2)^q\Delta.
}
\end{theorem}
Although \cref{thm:dejong} is the first moderate deviation result corresponding to \cite{deJo90}'s CLT for homogeneous sums in the literature, its optimality is unclear. 
For the case of $q=2$ and $|X_i|\leq K$ a.s., we can obtain the following optimal result. Its proof is a straightforward but very tedious modification of the proof of \cref{thm:dejong} and we leave it to the supplementary material. The proof technique would work for general $q$ if we introduce appropriate notation, but computation of mixed injective norms becomes extremely complicated. We do not pursue it further in this paper.

%When $q=2$, $W$ has sub-exponential tails under the assumptions of \cref{thm:dejong}, so one may expect a \blue{vanishing} relative error bound would be available for $x=\blue{o}(\Delta^{-1/3})$. However, \cref{thm:dejong} requires a stronger condition $x=\blue{o}(\Delta^{-1/5})$. In fact, when $q=2$ and $X_i$ are bounded, we can refine \cref{thm:dejong} as follows.
% 
\begin{theorem}\label{thm:qf}
Under the above setting, assume that $q=2$ and there exists a constant $K\geq1$ such that $|X_i|\leq K$ a.s.~for all $i\in[n]$. 
Set $F=(f(i,j))_{1\leq i,j\leq n}$. 
Then, there exists a positive absolute constant $C$ such that
\ben{\label{eq:qf}
\left|\frac{P(W>x)}{P(Z>x)}-1\right|
\leq CK^4(1+x)(|\log\|F\|_{op}|+x^2)\|F\|_{op}
}
for all $0\leq x\leq\norm{F}_{op}^{-1/3}$. 
\end{theorem}

\begin{remark}[Optimality of \cref{thm:qf}]
The error bound and the range of $x$ in \cref{thm:qf} are optimal. 
To see this, assume that $n$ is even and $X_i$ are i.i.d.~with $\E X_i^3\neq0$. Define the function $f$ as
\[
f(i,j)=
\begin{cases}
1/\sqrt{2n} & \text{if $\{i,j\}=\{2k-1,2k\}$ for some positive integer $k$},\\
0 & \text{otherwise}.
\end{cases}
\]
Then we have 
\[
W=\sum_{k=1}^{n/2}\frac{X_{2k-1}X_{2k}+X_{2k}X_{2k-1}}{\sqrt{2n}}
=\frac{1}{\sqrt{n/2}}\sum_{k=1}^{n/2}X_{2k-1}X_{2k}.
\]
So $W$ is a normalized sum of $n/2$ i.i.d.~random variables with mean 0 and variance 1. 
Since $\E[X_{2k-1}^3X_{2k}^3]=(\E X_1^3)^2\neq0$, we need the condition $x=o(n^{1/6})$ to get a vanishing relative error bound, and in this case the optimal bound is of the form $c(1+x^3)/\sqrt n$ for some constant $c>0$. This result is recovered by \cref{thm:qf} when $x\geq\sqrt{\log n}$ since $\|F\|_{op}=O(n^{-1/2})$. 
\end{remark}

\begin{remark}[Comparison with \cite{SaSt91}]\label{rem:homo2}
\cite{SaSt91} give Cram\'er-type moderate deviation results for polynomial forms of independent random variables in their Theorem 5.1 using the method of cumulants. 
Their result is in terms of %\blue{$\tilde\Delta$ was removed}
\[
\max_{\begin{subarray}{c}
r,s\in[q]\\
r+s=q
\end{subarray}}\sqrt{\left(\max_{i_1,\dots,i_r\in[n]}\sum_{i_{r+1},\dots,i_{q}=1}^n|f(i_1,\dots,i_q)|\right)\left(\max_{i_1,\dots,i_s\in[n]}\sum_{i_{s+1},\dots,i_{q}=1}^n|f(i_1,\dots,i_q)|\right)}
\]
and is not directly comparable with the fouth-moment-fluence bound in \cref{thm:dejong}.
Therefore, we only compare their result with ours in the setting of \cref{thm:qf}. 
Suppose that $X_1,\dots,X_n$ are i.i.d. Then, under the assumptions of \cref{thm:qf}, \cite[Theorem 5.1]{SaSt91} leads to a bound of the form $C K^4 (1+x^3)\|F\|_{op,\infty}$, where $\|F\|_{op,\infty}$ is the $\ell_\infty$-operator norm of $F$: $\|F\|_{op,\infty}:=\max_{1\leq i\leq n}\sum_{j=1}^n|f(i,j)|.$ 
% This follows from the simplification after Eq.(2.11) of \cite{SaSt91}.
Since $\norm{F}_{op}\leq\norm{F}_{op,\infty}$, our bound is better except for the logarithmic term in the error rate.
\end{remark}

\cref{thm:dejong} is a straightforward consequence of the following $p$-Wasserstein bound and \cref{t4}: 
\begin{proposition}\label{prop:homo}
Under the assumptions of \cref{thm:dejong}, for any $2\leq p\leq\mcl M(f)^{-1/2}$, 
\ba{
\mcl W_p(W,Z)
%&\leq C_qK^{2q}\left(
%p^{q}\sqrt{\mcl M(f)}|\log\mcl M(f)|^{q-1}
%+\max_{r\in[q-1]}\max_{\mcl J\in\Pi_{2q-2r}}p^{(1+|\mcl J|)/2}\|f\wt\otimes_rf\|_{\mcl J}
%\right)\label{homo-wass}\\
&\leq C_qp^q\Delta.%\label{dejong-wass}
}
\end{proposition}

The proof of \cref{prop:homo} is deferred to \cref{appendix:homo}. 
\begin{proof}[Proof of \cref{thm:dejong}]
We first note that $\mcl M(f)\leq\sum_{i_1,\dots,i_q=1}^nf(i_1,\dots,i_q)^2=1/q!\leq1/2$. 
%\red{So $1\vee$ can be omitted in \cref{thm:dejong}?} \blue{$1\vee$ is necessary to ensure $\Delta\geq\sqrt{\mcl M(f)}$.}
We apply \cref{t4} with $r_0=1$, $\alpha_1=q$, $\Delta_1=\Delta$ and $p_0=\mcl M(f)^{-1/2}$. Then, it remains to check $|\log\Delta|\leq p_0/2$ and $\sqrt p_0\geq\Delta^{-1/(2q+1)}$. Since $M\geq(\E X_1^2)^2=1$, we have $\Delta\geq\sqrt{\mcl M(f)}$. This and the assumption $\Delta<1$ give the desired result. 
\end{proof}

%\red{The symbol $q$ in homogeneous sums conflicts with the $q$ in \cref{t4}, although it only matters in the above short proof. Need to do some change?}

%\blue{Maybe unnecessary. (I should have used different notation in \cref{t4} ... It's not $q$ even in the proof of \cref{thm:wiener} ...)}

\section{Moderate deviations in multi-dimensions}\label{sec:multi}

In this section, we study moderate deviations in multi-dimensions.
We first apply \cref{t1} to obtain a $p$-Wasserstein bound for multivariate normal approximation of sums of independent random vectors. All the proofs for the results in this section are deferred to \cref{appendix:multi}.

\begin{theorem}\label{thm:multip}
Let $W=n^{-1/2}\sum_{i=1}^n X_i\in \mathbb{R}^d$, where $\{X_1,\dots, X_n\}$ are independent, $\E(X_i)=0$ for all $i$, and $\Var(W)=I_d$.
Suppose $\norm{X_i}_{\psi_1}\leq b$ for all $1\leq i\leq n$. Let $Z\sim N(0, I_d)$.
Then, for any $p\geq 2$, we have
\ben{\label{028}
\mcl{W}_p(W, Z)\leq C (\frac{p d^{1/4}}{\sqrt{n}}+\frac{p^{5/2}}{n})b^2.
}
\end{theorem}

We can use $p$-Wasserstein bounds to obtain moderate deviation results in the multi-dimensional setting. In the following theorem, 
we provide an analogous result as \cref{t4} for $|P(|W|>x)/P(|Z|>x)-1|$. For simplicity, we only state a result corresponding to $r_0=1$ in \cref{t4}, which suffices for the applications we consider.
We remark that our approach can be used to obtain upper bounds on $|P(W\notin A)/P(Z\notin A)-1|$ for more general convex sets $A\subset \mathbb{R}^d$ as long as we have a suitable control on $P(Z\in A^\eps\backslash A^{-\eps})/P(Z\notin A$) for small $\eps>0$, where $A^\eps\backslash A^{-\eps}$ contains all $x\in \mathbb{R}^d$ within distance $\eps$ away from the boundary of $A$.

\begin{theorem}\label{thm:multiMD}
Let $W$ be a $d$-dimensional random vector, $d\geq 2$, and $Z\sim N(0,I_d)$. Suppose
\be{
\mcl{W}_p(W, Z)\leq A p^\alpha \Delta \ \text{for}\ 1\leq p\leq  p_0
}
with some constants $\alpha\geq 0$, $A>0$, $\Delta>0$, $|\log \Delta|\leq p_0/4$ and $\log(\kappa(d))\leq p_0/4$ with $\kappa(d):=2^{(d/2)-1}\Gamma(d/2)$. 
Suppose further that
\ben{\label{029}
d(d\log d)^\alpha \Delta\leq B_1
%d(\log d) \Delta^{2/(2\alpha+1)}\leq B_1
}
and
\ben{\label{030}
d \Delta |\log \Delta|^\alpha\leq B_2,\ \text{if}\ 0<\alpha\leq 1/2.
}
Then there exists a positive constant $C_{A, \alpha, B_1, B_2}$ depending only on $\alpha$, $A$, $B_1$ and $B_2$ such that 
\ben{\label{036}
\left|\frac{P(|W|>x)}{P(|Z|>x)}-1\right|\leq C_{A, \alpha, B_1, B_2}(1+x)(|\log\Delta|+d \log d+x^2)^{\alpha}\Delta
}
for all $0\leq x\leq \min\{\Delta^{-1/(2\alpha+1)}, \sqrt{p_0}\}$.
\end{theorem}

The following Cram\'er-type moderate deviation result for sums of independent random vectors is an easy consequence of \cref{thm:multip}, \cref{thm:multiMD} with $\alpha=1$, $p_0=\Delta^{-2/3}$ and the fact that $d=\E|W|^2\leq C b^2$.

\begin{theorem}\label{thm:chi}
Under the setting of \cref{thm:multip} with $d\geq 2$, 
let
\be{
\Delta:=\frac{d^{1/4} b^2}{\sqrt{n}}.
}
%assume
%\be{
%\frac{d^3(\log d)b^2 \log^2 n}{\sqrt{n}}\leq B
%}
%for some constant $B$.
Then there exist positive absolute constants $c$ and $C$ such that, for
\be{
d^2 (\log d) \Delta\leq c,\quad 0\leq x\leq \Delta^{-1/3},
}
we have
\be{
\left|  \frac{P(|W|>x)}{P(|Z|>x)}-1    \right|\leq C(1+x)(d \log d+|\log \Delta|+x^2)\Delta.
}
\end{theorem}

\begin{remark}
The result in \cref{thm:chi} recovers the optimal range $0\leq x=o(n^{1/6})$ (cf. \cite{Vo67}) for the relative error to vanish.
Although it is known that the error rate can be improved because of the symmetry of Euclidean balls, see, for example, \cite{Vo67} and \cite{FaLiSh21}, their proofs depend on the conjugate method, which relies heavily on the independence assumption.
Our approach works for the dependent case (cf. \cref{thm:local1,thm:local2}).
\end{remark}

\section{Local dependence}\label{sec:local}

A large class of random vectors that can be approximated by a normal distribution exhibits a local dependence structure. Roughly speaking, we assume that the random vector $W$ is a sum of a large number of random vectors $\{X_i\}_{i=1}^n$ and that each $X_i$ is independent of $\{X_j: j\notin A_i\}$ for a relatively small index set $A_i$. Variations of such local dependence structure and normal approximation results with absolute error bounds can be found in, e.g., \cite{BaRi89}, \cite{BaKaRu89} and \cite{ChSh04}. Moderate deviation results (relative error bounds) under local dependence were recently obtained by \cite{LZ21} in dimension one. See \cref{rem:LZ21} for a comparison. 

Throughout this section, we assume $n\geq2$. %\blue{to ensure $\log n$ is bounded away from 0}

\subsection{Bounded case}

We first provide a $p$-Wasserstein bound for multivariate normal approximation of sums of locally dependent, bounded random vectors.

\begin{theorem}\label{thm:local3}
Let $W=n^{-1/2}\sum_{i=1}^n X_i\in \mathbb{R}^d$ with $\E(X_i)=0$ for all $i$ and $\Var(W)=I_d$. We assume that for each $i$, there is a neighborhood $A_i\subset \{1,\dots, n\}$ such that $X_i$ is independent of $\{X_j: j\notin A_i\}$. Assume further that for each $i$ and $j\in A_i$, there exists a second neighborhood $A_{ij}$ such that $\{X_i, X_j\}$ is independent of $\{X_k: k\notin A_{ij}\}$. 
Let 
\be{
B_{ij}:=\{(k,l): k\in \{1,\dots, n\}, l\in A_k, k\ \text{or}\ l\in A_{ij}\}.
}
Suppose
\be{
|X_i|\leq b_n,\  |X_{ij}|\leq b_n', \  |A_i|\leq \theta_1,\   |B_{ij}|\leq \theta_2,
}
where $X_{ij}$ denotes the $j$th component of $X_i$ and $|\cdot|$ denotes the cardinality when applied to a set. 
Then there exist positive absolute constants $c$ and $C$ such that, for
%\be{
%2\leq  p\leq \theta_1 n /\theta_2, \quad \theta_1 b_n \sqrt{p/n}\leq c,
%}
\ben{\label{037}
2\leq p\leq \min\{\frac{\theta_1}{\theta_2}, \frac{c}{\theta_1^2 b_n^2}\} n
}
we have, with $Z\sim N(0, I_d)$,
\ben{\label{025}
\mcl{W}_p(W, Z)\leq Cp \left(\frac{ d  (\theta_1 \theta_2)^{1/2} b_n'^2+\theta_1^2 b_n^3\log n}{\sqrt{n}}\right).
}
\end{theorem}

\begin{remark}
We will adapt the proof of \cref{t1} to prove \cref{thm:local3} in \cref{appendix:local}. 
Without exchangeability, we can not use the symmetry trick in \eq{043}. Therefore, because of the integrability issue of $1/(e^{2t}-1)$ for $t$ near 0, we get an additional logarithmic term in \eq{025} (cf. \cref{sec:proofthm2.1}).
%This seems only work under a boundedness condition. We obtain a third moment bound without exchangeability, hence the $b_n^3 \log n$ term in \eq{025}.
%Also, from the proof, we can see that $b_n^2$ in the first term on the right-hand side of \eq{025} can be replaced by $\max_{i,j}\norm{X_{ij}}^2_\infty$, where $X_{ij}$ denotes the $j$th component of $X_i$.
\end{remark}

Using \cref{thm:local3} together with \cref{t4,thm:multiMD}, we obtain the following moderate deviation result for sums of locally dependent, bounded random vectors.

\begin{theorem}\label{thm:local1}
Under the same condition as in \cref{thm:local3}, 
%let $p_0:=\theta_1 n/\theta_2$.
for $d=1$, there exist positive absolute constants $c$ and $C$ such that, if
\be{
\Delta_1:=\frac{(\theta_1 \theta_2)^{1/2}b_n'^2+\theta_1^2 b_n^3 \log n}{\sqrt{n}}\leq c, %\quad \theta_1 b_n\sqrt{p_0/n}\leq c,
}
then, for $0\leq x\leq \Delta_1^{-1/3}$,
\be{
\left|\frac{P(W>x)}{P(Z>x)}-1    \right|\leq C(1+x)(1+|\log \Delta_1|+x^2)\Delta_1.
}
For $d\geq 2$, let
\be{
\Delta_d:=\frac{d(\theta_1 \theta_2)^{1/2}b_n'^2+\theta_1^2 b_n^3 \log n}{\sqrt{n}}.
}
Then, there exist a positive absolute constants $c$ and $C$ such that, for
$d^2(\log d)\Delta_d\leq c$ and $0\leq x\leq \Delta_d^{-1/3}$, we have
\be{
\left|\frac{P(|W|>x)}{P(|Z|>x)}-1    \right|\leq C(1+x)(|\log \Delta_d|+d\log d+x^2)\Delta_d.
}
\end{theorem}

\begin{proof}[Proof of \cref{thm:local1}]
Note that $d=\E (W^T W)\leq \theta_1 b_n^2$ and $1\leq \theta_1 b_n'^2$.
First consider the case $d=1$.
Let $p_0=\Delta_1^{-2/3}$.
If $\Delta_1$ is sufficently small, then
$|\log \Delta_1|\leq p_0/2$ and moreover, using $d\leq \theta_1 b_n^2$ and $1\leq \theta_1 b_n'^2$,
\be{
p_0=\Delta_1^{-2/3}\leq \min\{ (\frac{\theta_1 n}{ \theta_2})^{1/3} , (\frac{n}{\theta_1^2 b_n^2})^{1/3}  \},
}
which is bounded by the right-hand side of \eq{037}.
\cref{thm:local1} then follows from \cref{t4} with $r_0=\alpha_1=1$ and \cref{thm:local3}.
The case $d\geq 2$ follows by using \cref{thm:multiMD} with $\alpha=1$ instead of \cref{t4}.
\end{proof}

\subsection{Unbounded case}

Next, we consider the unbounded case. We will do truncation and use Bernstein's inequality to control the truncation error. For this purpose, we need to assume that the index set $\{1,\dots, n\}$ can be partitioned into $L$ groups $g_1, \dots, g_L$ such that for each group $g_l$, the summands $\{X_i: i\in g_l\}$ are independent. We give two examples below. 
The next theorem, whose proof is deferred to \cref{appendix:local}, provides a moderate deviation result under this setting.

\begin{theorem}\label{thm:local2}
Under the setting of \cref{thm:local1}, replace the boundedness conditions $|X_i|\leq b_n$ and $|X_{ij}|\leq b_n'$ by $\norm{X_{ij}}_{\psi_1}\leq b$. Assume in addition the above partition condition with $L$ groups.
Let 
\be{
\Delta_d:=\frac{d L  b \log n+d (\theta_1 \theta_2)^{1/2}b^2\log^2 n+d^{3/2}\theta_1^2 b^3 \log^4 n}{\sqrt{n}}.
}
For $d=1$, there exist positive absolute constants $c$ and $C$ such that, if
$\Delta_1\leq c$
and $0\leq x\leq \Delta_1^{-1/3}$, then
\ben{\label{026}
\left|\frac{P(W>x)}{P(Z>x)}-1    \right|\leq C(1+x)(1+|\log \Delta_1|+x^2)\Delta_1.
}
For $d\geq 2$, there exist a positive absolute constants $c$ and $C$ such that, if
$d^2(\log d)\Delta_d\leq c$
and $0\leq x\leq \Delta_d^{-1/3}$, then
\ben{\label{033}
\left|\frac{P(|W|>x)}{P(|Z|>x)}-1    \right|\leq C(1+x)(|\log \Delta_d|+d\log d+x^2)\Delta_d.
}

\end{theorem}

\begin{example}
In $m$-dependence (cf. \cite{HoRo48}), 
it is assumed that $X_i$ is independent of $\{X_j: |i-j|> m\}$.
We obtain the following corollary of \cref{thm:local2} for the case $d=1$.
\begin{corollary}\label{cor:mdep}
Let $\{X_1,\dots, X_n\}$ be a sequence of $m$-dependent random variables with $m\geq 1$, $\E(X_i)=0$ and $\norm{X_i}_{\psi_1}\leq b$. 
Let $W=\frac{1}{\sqrt{n}}\sum_{i=1}^n X_i$. Suppose $\Var(W)=1$.
Let 
\be{
\Delta=\frac{m^2 b^3 \log^4 n}{\sqrt{n}}.
}
Then there exist positive absolute constants $c$ and $C$ such that, for
\be{
\Delta\leq c, \quad 0\leq x\leq \Delta^{-1/3},
}
we have
\be{
\left|\frac{P(W>x)}{P(Z>x)}-1    \right|\leq C(1+x)(1+|\log \Delta|+x^2)\Delta. 
}
\end{corollary}
\begin{proof}[Proof of \cref{cor:mdep}]
Under $m$-dependence, $\{X_1, \dots, X_n\}$ can be partitioned into $L=m+1$ groups such that the $X$'s in each group are independent.
Moreover, the quantities appearing in the statement of \cref{thm:local2} can be taken as
\be{
\theta_1\asymp m, \  \theta_2\asymp m^2.
}
Using $1\leq C m b^2$, we have, $\Delta_1\leq C\Delta$.
%, and for sufficiently large $\Delta$,
%$\theta_1 b \log n\sqrt{p_0/n}$ is sufficiently small and $\Delta^{-1/3}\leq (\theta_1 n/\theta_2)^{1/4}$.
The corollary then follows from \eq{026}.
\end{proof}
\end{example}

\begin{example}
In graph dependency structure (cf. \cite{BaRi89}),
each index $i\in \{1,\dots, n\}$ is represented by a node in a simple graph and $\{X_i: i\in A\}$ is assumed to be independent of $\{X_j: j\in B\}$ if $A$ and $B$ are disconnected.
In such graph dependency structure, if the maximum degree of the dependency graph is $deg^*$, then $L$ can be taken as $L=deg^*+1$. This is because each time we take out a group of independent summands, we can do it in a way that the max degree is decreased by 1. Therefore, \cref{thm:local2} also applies. We omit the straightforward result.
\end{example}

\begin{remark}\label{rem:LZ21}
\cite{LZ21} obtained a moderate deviation result under local dependence in dimension one using a different method. Their result is stated under a more general condition and does not have the additional logarithmic terms. However, the dependence on the neighborhood size and $b$ in their result is worse than ours. For example, under $m$-dependence, the bound using their Theorem 2.1 with $\kappa\asymp m$ and $a_n\asymp \sqrt{n}/(mb)$ is
\be{
\left|\frac{P(W>x)}{P(Z>x)}-1\right|\leq C(1+x^3) \frac{m^9 b^7}{\sqrt{n}},
}
while our bound is (cf. \cref{cor:mdep}), subject to logarithmic terms,
\be{
\left|\frac{P(W>x)}{P(Z>x)}-1\right|\lesssim_{\log} C(1+x^3) \frac{m^2 b^3}{\sqrt{n}}.
}
Moreover, our approach generalizes easily to multi-dimensions.
\end{remark}

\section{Proof of the $p$-Wasserstein bound}\label{sec:proof-p-wass}

In this section, we prove \cref{t1}.
Without loss of generality, we may assume $Z$ is independent of $\mcl{G}$ and $W'$. 

We introduce some notation. %Given a positive integer $N$, we set $[N]:=\{1,\dots,N\}$. 
Let $k\in\mathbb{N}$. Given families of real numbers $a=(a_{i_1,\dots,i_k})_{1\leq i_1,\dots,i_k\leq d}$ and $b=(b_{i_1,\dots,i_k})_{1\leq i_1,\dots,i_k\leq d}$, we set
\[
\langle a,b\rangle:=\sum_{i_1,\dots,i_k=1}^da_{i_1,\dots,i_k}b_{i_1,\dots,i_k},\qquad
|a|:=\sqrt{\langle a,a\rangle}=\sqrt{\sum_{i_1,\dots,i_k=1}^da_{i_1,\dots,i_k}^2}.
\]
Note that, if $k=2$, $\langle a,b\rangle=\langle a,b\rangle_{H.S.}$ and $|a|=\|a\|_{H.S.}$. 
For $x_1,\dots,x_k\in\mathbb{R}^d$, we define
\[
x_1\otimes\cdots\otimes x_k:=(x_{1,i_1}\cdots x_{k,i_k})_{1\leq i_1,\dots,i_k\leq d}.
\]
If $x_1=\cdots=x_d=:x$, we write $x_1\otimes\cdots\otimes x_k=x^{\otimes k}$ for short. 
Also, if a function $f:\mathbb{R}^d\to\mathbb{R}$ is $k$-times differentiable at $w\in\mathbb{R}^d$, we set
\ben{\label{eq:nabla}
\nabla^kf(w):=\left(\frac{\partial^kf}{\partial w_{i_1}\cdots\partial w_{i_k}}(w)\right)_{1\leq i_1,\dots,i_k\leq d}.
}
Given a family of random variables $X=(X_{i_1,\dots,i_k})_{1\leq i_1,\dots,i_k\leq d}$ and $p>0$, we set
\[
\|X\|_p:=\left(\E|X|^p\right)^{1/p}.
\]

We denote by $\phi$ the $d$-dimensional standard normal density. For brevity, we write $\eta_t$ instead of $\eta_t(p)$ throughout this section. 

\subsection{Auxiliary estimates}

For every $t>0$, we set $F_t:=e^{-t}W+\sqrt{1-e^{-2t}}Z$. It is straightforward to check that $F_t$ has a smooth density $f_t$ with respect to $N(0,I_d)$. Moreover, $f_t$ is strictly positive by Lemma 3.1 of \cite{JoSu01}. Therefore, we can define the \textit{score} of $F_t$ with respect to $N(0,I_d)$ by $\rho_t(w)=\nabla \log f_t(w)$, $w\in\mathbb{R}^d$. We use $C$ to denote positive absolute constants, which may differ in different expressions.
\begin{proposition}\label{p1}
Let $p\geq1$ and $t>0$. 
Under the assumptions of \cref{t1}, we have
\ba{
\|\rho_t(F_t)\|_p
\leq C e^{-t}\left(
\|R_t\|_p+\frac{\|E\|_p}{\eta_t}
+\min\left\{\frac{\sqrt {d}}{\eta_t}, \frac{\|\E[D^{\otimes 2}|D|^21_{\{|D|\leq\eta_t\}}\mid \mcl{G}]\|_p}{\lambda\eta_t^{3}}\right\}
\right).
}
\end{proposition}

We need some lemmas to prove \cref{p1}.
\begin{lemma}[Lemma A.1 of \cite{FaKo22}]\label{psd}
Let $Y=(Y_{ij})_{1\leq i, j\leq d}$ be a $d\times d$ positive semidefinite symmetric random matrix. Let $F$ and $G$ be two random variables such that $|F|\leq G$. Suppose that $\E|Y_{ij} F|<\infty$ for all $i,j=1,\dots, d$. 
Let $\mathcal{G}$ be an arbitrary $\sigma$-field.
Then we have
\be{
\norm{\E[YF|\mathcal{G}]}_{H.S.}\leq \norm{\E[YG|\mathcal{G}]}_{H.S.}.
}
\end{lemma}

\begin{lemma}[Lemma A.2 of \cite{FaKo22}]\label{tensor}
Let $Y$ be a random vector in $\mathbb{R}^d$ such that $\E|Y|^k<\infty$ for some integer $k\geq 2$. 
Let $\mathcal{G}$ be an arbitrary $\sigma$-field.
Then
\be{
|\E [Y^{\otimes k}|\mathcal{G}]|\leq \norm{\E[Y^{\otimes2}|Y|^{k-2}|\mathcal{G}]}_{H.S.}.
}
\end{lemma}

\begin{lemma}\label{hyper}
Let $F$ be a random vector in $\mathbb{R}^m$ whose components are of the form $Q(Z_1,\dots,Z_d)$, where $Q$ is a polynomial of degree $\leq k$. 
Then, for every $p>0$,
\[
\|F\|_p\leq\kappa_p^k\|F\|_2,
\]  
where $\kappa_p:=e\sqrt{(p/2-1)\vee1}$.
\end{lemma}

\begin{proof}
Since $|F|^2$ is a polynomial of degree $\leq 2k$ in $Z_1,\dots,Z_d$ by assumption, we have by Theorem 5.10 and Remark 5.11 of \cite{Ja97}
\ba{
\|F\|_p=\||F|^2\|_{p/2}^{1/2}
&\leq(p/2-1)^{k/2}\||F|^2\|_{2}^{1/2}
}
if $p\geq4$. Since we have $\||F|^2\|_{p/2}\leq\||F|^2\|_{2}$ if $p<4$, we obtain
\ben{\label{hyper1}
\|F\|_p\leq\{(p/2-1)\vee1\}^{k/2}\||F|^2\|_{2}^{1/2}.
}
Next, we have by Theorem 5.10 and Remark 5.13 of \cite{Ja97}
\ben{\label{hyper2}
\||F|^2\|_{2}\leq e^{2k}\||F|^2\|_1=e^{2k}\|F\|_2^2.
}
The desired result follows from \eqref{hyper1}--\eqref{hyper2}. 
\end{proof}

Given a bounded measurable function $h:\mathbb{R}^d\to\mathbb{R}$ and $t>0$, we define the function $T_th:\mathbb{R}^d\to\mathbb{R}$ by
\[
T_th(w)=\E h(e^{-t}w+\sqrt{1-e^{-2t}}Z),\qquad w\in\mathbb{R}^d.
\]
One can easily check that $T_th$ is infinitely differentiable and 
\ben{\label{deriv}
\nabla^kT_th(w)=\frac{(-1)^k}{(e^{2t}-1)^{k/2}}\int_{\mathbb{R}^d}h(e^{-t}w+\sqrt{1-e^{-2t}}z)\nabla^k\phi(z)dz,\qquad k=1,2,\dots.
}

\begin{lemma}\label{l1}
Let $X$ and $X'$ be two $d$-dimensional random vectors such that $|X'-X|$ is bounded, and set $Y:=X'-X$. 
Then, for any integer $l\geq0$, bounded measurable function $h:\mathbb{R}^d\to\mathbb{R}$ and $t>0$, we have
\ba{
\langle\nabla^lT_th(X')-\nabla^l T_th(X),Y^{\otimes l}\rangle=\sum_{k=1}^\infty\frac{1}{k!}\langle\nabla^{l+k}T_th(X),Y^{\otimes (l+k)}\rangle\qquad\text{in }L^\infty(P).
}
\end{lemma}

\begin{proof}
By assumption, there is a constant $M>0$ such that $|Y|\leq M$ and $\sup_{x\in\mathbb{R}^d}|h(x)|\leq M$. Using \eqref{deriv}, we deduce
\ba{
|\langle \nabla^kT_th(w),Y^{\otimes k}\rangle|&\leq\frac{M}{(e^{2t}-1)^{k/2}}\int_{\mathbb{R}^d}|\langle\nabla^k\phi(z),Y^{\otimes k}\rangle| dz\\
&\leq\frac{M}{(e^{2t}-1)^{k/2}}\sqrt{\int_{\mathbb{R}^d}\left(\frac{\langle\nabla^k\phi(z),Y^{\otimes k}\rangle}{\phi(z)}\right)^2 \phi(z)dz}.
}
Thus, we have by Lemma 4.3 of \cite{FaRo15}
\ba{
|\langle \nabla^kT_th(w),Y^{\otimes k}\rangle|
&\leq\frac{M}{(e^{2t}-1)^{k/2}}\sqrt{k!|Y^{\otimes k}|^2}
=\frac{M\sqrt{k!}|Y|^k}{(e^{2t}-1)^{k/2}}
\leq\frac{M^{k+1}\sqrt{k!}}{(e^{2t}-1)^{k/2}}.
}
Hence, for any integer $K>0$, we have by Taylor's expansion
\ba{
&\left|\langle\nabla^lT_th(X')-\nabla^l T_th(X),Y^{\otimes l}\rangle-\sum_{k=1}^K\frac{1}{k!}\langle\nabla^{l+k}T_th(X),Y^{\otimes (l+k)}\rangle\right|\\
&\leq\sup_{u\in[0,1]}\frac{1}{(K+1)!}\left|\langle\nabla^{l+K+1}T_th(X+uY),Y^{\otimes (l+K+1)}\rangle\right|
\leq\frac{M^{l+K+1}\sqrt{(l+K+1)!}}{(K+1)!(e^{2t}-1)^{(l+K+1)/2}}.
}
Since the last quantity tends to 0 as $K\to\infty$, we complete the proof. 
\end{proof}

For every $t>0$, let 
\ba{
%\gamma_t&:=\sqrt{e^{2t}-1},&
%\eta_t&:=\gamma_t/\sqrt{p},&
D_t&:=D1_{\{|D|\leq\eta_t\}},&
W_t&:=W+D_t.
}
Note that we have
\be{%\label{w-alpha}
W_t=\begin{cases}
W' & \text{if }|D|\leq\eta_t,\\
W & \text{if }|D|>\eta_t.
\end{cases}
}
One can check that $(W,W_t)$ is an exchangeable pair. In fact, for any $u,v\in\mathbb{R}^d$, we have
\ba{
\E[e^{\sqrt{-1}(u\cdot W+v\cdot W_t)}]
&=\E[e^{\sqrt{-1}(u\cdot W+v\cdot W')};|D|\leq\eta_t]
+\E[e^{\sqrt{-1}(u\cdot W+v\cdot W)};|D|>\eta_t]\\
&=\E[e^{\sqrt{-1}(u\cdot W'+v\cdot W)};|D|\leq\eta_t]
+\E[e^{\sqrt{-1}(u\cdot W+v\cdot W)};|D|>\eta_t]\\
&=\E[e^{\sqrt{-1}(u\cdot W_t+v\cdot W)};|D|\leq\eta_t]
+\E[e^{\sqrt{-1}(u\cdot W_t+v\cdot W)};|D|>\eta_t]\\
&=\E[e^{\sqrt{-1}(u\cdot W_t+v\cdot W)}],
}
where the second equality follows from the exchangeability of $(W,W')$. 
Also, using \eqref{1} and recalling \eqref{def:rt}, one can easily check
\ben{\label{rt}
\E[W_t-W\mid\mcl{G}]
=-\Lambda(W+R_t).
}
Let us set
\ben{\label{def:tau}
\tau_t:=\E\left[\Lambda^{-1}D_t\left(1-\frac{1}{2}\frac{\langle\nabla \phi(Z),D_t\rangle}{\phi(Z)\sqrt{e^{2t}-1}}
+\frac{1}{2}\sum_{k=3}^\infty a_k\frac{(-1)^k\langle\nabla^k \phi(Z),D_t^{\otimes k}\rangle}{\phi(Z)(e^{2t}-1)^{k/2}}\right)\mid \mcl{G}\vee\sigma(Z)\right],
%\tau_t:=e^{-t}\E\left[\Lambda^{-1}D_t\left(1+\frac{1}{2}\frac{1}{\sqrt{e^{2t}-1}}D_t\cdot Z
%+\frac{1}{2}\sum_{\alpha\in\mathbb{Z}_+^d:|\alpha|\geq3}b_{|\alpha|}\frac{|\alpha|!H_\alpha(Z)D_t^\alpha}{\alpha!(e^{2t}-1)^{|\alpha|/2}}\right)\mid W,Z\right]
}
where $a_k:=\frac{1}{k!}-\frac{1}{4(k-2)!}$.
As in the proof of \cref{l1}, one can check that the series inside the conditional expectation in \eqref{def:tau} converges in $L^1(P)$, so $\tau_t$ is well-defined. 

\begin{lemma}\label{lem:tau}
$\E[\tau_t\mid F_t]=0$ for all $t>0$. 
\end{lemma}

\begin{proof}
It suffices to prove $\E[\tau_th(F_t)]=0$ for any bounded measurable function $h:\mathbb{R}^d\to\mathbb{R}$. We have by exchangeability
\[
\E[\Lambda^{-1}D_t\{T_th(W)+T_th(W_t)\}]=0.
\]
Applying \cref{l1}, we obtain
\ben{\label{taylor}
\E\left[\Lambda^{-1}D_t\left\{T_th(W)+\sum_{k=0}^\infty\frac{1}{k!}\langle\nabla^kT_th(W),D_t^{\otimes k}\rangle\right\}\right]=0.
}
Now, we have again by exchangeability  
\ban{\label{043}
\E\left[\Lambda^{-1}D_t\langle\nabla^2T_th(W),D_t^{\otimes 2}\rangle\right]
=-\E\left[\Lambda^{-1}D_t\langle\nabla^2T_th(W_t),D_t^{\otimes 2}\rangle\right].
}
Hence we obtain 
\ba{
\E\left[\Lambda^{-1}D_t\langle\nabla^2T_th(W),D_t^{\otimes 2}\rangle\right]
&=-\frac{1}{2}\E\left[\Lambda^{-1}D_t\langle\nabla^2T_th(W_t)-\nabla^2T_th(W),D_t^{\otimes 2}\rangle\right]\\
&=-\frac{1}{2}\E\left[\Lambda^{-1}D_t\sum_{k=1}^\infty\frac{1}{k!}\langle\nabla^{k+2}T_th(W),D_t^{\otimes (2+k)}\rangle\right]\\
&=-\frac{1}{2}\E\left[\Lambda^{-1}D_t\sum_{k=3}^\infty\frac{1}{(k-2)!}\langle\nabla^{k}T_th(W),D_t^{\otimes k}\rangle\right].
}
Inserting this into \eqref{taylor}, we deduce
\ben{\label{taylor2}
\E\left[\Lambda^{-1}D_t\left\{2T_th(W)+D_t\cdot\nabla T_th(W)+\sum_{k=3}^\infty a_k\langle\nabla^kT_th(W),D_t^{\otimes k}\rangle\right\}\right]=0.
}
%Using \eqref{convert}, we obtain
%\ben{\label{taylor2}
%\E\left[\Lambda^{-1}D_t\left\{2T_th(W)+D_t\cdot\nabla T_th(W)+\sum_{\alpha\in\mathbb{Z}_+^d:|\alpha|\geq3} b_{|\alpha|}\frac{|\alpha|!}{\alpha!}\partial^\alpha T_th(W)D_t^\alpha\right\}\right]=0.
%}
%Now, we have for every $\alpha\in\mathbb{Z}_+^d$
%\ba{
%\partial^\alpha T_th(w)
%&=\frac{(-1)^{|\alpha|}}{(e^{2t}-1)^{|\alpha|/2}}\int_{\mathbb{R}^d}h(e^{-t}w+\sqrt{1-e^{-2t}}z)\partial^\alpha\phi(z)dz\\
%&=\frac{1}{(e^{2t}-1)^{|\alpha|/2}}\E[h(e^{-t}w+\sqrt{1-e^{-2t}}Z)H_\alpha(Z)].
%}
Meanwhile, we have by \eqref{deriv}
\ba{
\nabla^kT_th(w)=\frac{(-1)^k}{(e^{2t}-1)^{k/2}}\E h(e^{-t}w+\sqrt{1-e^{-2t}}Z)\frac{\nabla^k\phi(Z)}{\phi(Z)}.
} 
Inserting this into \eqref{taylor2} and using the definition of $F_t$, we obtain $2\E[\tau_th(F_t)]=0$. Hence we complete the proof.
\end{proof}

\begin{proof}[Proof of \cref{p1}]
Recall 
\be{
a_k:=\frac{1}{k!}-\frac{1}{4(k-2)!},~\kappa_p:=e\sqrt{(p/2-1)\vee1}.
}
We divide the proof into two steps. 

\noindent\textbf{Step 1}. 
We first prove the following inequality:
\ben{\label{step1}
\|\rho_t(F_t)\|_p
\leq e^{-t}\left(\|R_t\|_p+\frac{\kappa_p}{\sqrt{e^{2t}-1}}\|E_t\|_p+\frac{1}{2}\sum_{k=3}^\infty \frac{|a_k|\kappa_p^k\sqrt{k!}}{(e^{2t}-1)^{k/2}}\left\|\E[(\Lambda^{-1}D_t)\otimes D_t^{\otimes k}\mid \mcl{G}]\right\|_p\right),
}
where
\[
E_t:=E-\frac{1}{2}\E[(\Lambda^{-1}D)\otimes D1_{\{|D|>\eta_t\}}\mid\mcl{G}].
\]
We have by Lemma IV.1 of \cite{NoPeSw14} (see also Lemma 2 of \cite{Bo20})
\ben{\label{eq:score}
\rho_t(F_t)=\E\left[e^{-t}W-\frac{e^{-2t}}{\sqrt{1-e^{-2t}}}Z\mid F_t\right]
=e^{-t}\E\left[W-\frac{1}{\sqrt{e^{2t}-1}}Z\mid F_t\right].
}
Hence, \cref{lem:tau} yields
\ba{
\rho_t(F_t)
&=e^{-t}\E\left[W-\frac{1}{\sqrt{e^{2t}-1}}Z+\tau_t\mid F_t\right]\\
&=e^{-t}\E\left[-R_t+\frac{1}{\sqrt{e^{2t}-1}}E_tZ+\frac{1}{2}\sum_{k=3}^\infty a_k\E\left[\Lambda^{-1}D_t\frac{(-1)^k\langle\nabla^k \phi(Z),D_t^{\otimes k}\rangle}{\phi(Z)(e^{2t}-1)^{k/2}}\mid \mcl{G}\vee\sigma(Z)\right]\mid F_t\right].
}
Therefore, we have by the Jensen and Minkowski inequalities 
\begin{multline}\label{step1-eq1}
\|\rho_t(F_t)\|_p
\leq e^{-t}\left(\|R_t\|_p+\frac{1}{\sqrt{e^{2t}-1}}\|E_tZ\|_p\right.\\
\left.+\frac{1}{2}\sum_{k=3}^\infty \frac{|a_k|}{(e^{2t}-1)^{k/2}}\left\|\E\left[\Lambda^{-1}D_t\frac{\langle\nabla^k \phi(Z),D_t^{\otimes k}\rangle}{\phi(Z)}\mid \mcl{G}\vee\sigma(Z)\right]\right\|_p\right).
\end{multline}
%Lemma 3 of \cite{Bo20}
%\ba{
%&\E|\rho_t(F_t)|^p\\
%&\leq e^{-t}\E\left[\left(|R_t|^2+\frac{1\vee(p-1)}{e^{2t}-1}\|E_t\|_{H.S.}^2+\frac{1}{4}\sum_{\alpha\in\mathbb{Z}_+^d:|\alpha|\geq3}1\vee(p-1)^{|\alpha|}b_{|\alpha|}^2\frac{(|\alpha|!)^2|\E[\Lambda^{-1}D_tD_t^\alpha|\mid W]|^2}{\alpha!(e^{2t}-1)^{|\alpha|}}\right)^{p/2}\right]\\
%&=e^{-t}\E\left[\left(|R_t|^2+\frac{1\vee(p-1)}{e^{2t}-1}\|E_t\|_{H.S.}^2+\frac{1}{4}\sum_{k=3}^\infty1\vee(p-1)^{k}k!b_{k}^2\frac{|\E[\Lambda^{-1}D_t\otimes D_t^{\otimes k}|\mid W]|^2}{\alpha!(e^{2t}-1)^{k}}\right)^{p/2}\right]
%}
Now, \cref{hyper} yields
\ba{
\E[|E_tZ|^p\mid \mcl{G}]\leq\left(\kappa_p^2\E[|E_tZ|^2\mid \mcl{G}]\right)^{p/2}
}
and
\ba{
&\E\left[\left|\E\left[\Lambda^{-1}D_t\frac{\langle\nabla^k \phi(Z),D_t^{\otimes k}\rangle}{\phi(Z)}\mid \mcl{G}\vee\sigma(Z)\right]\right|^p\mid \mcl{G}\right]\\
&\leq\left(\kappa_p^{2k}\E\left[\left|\E\left[\Lambda^{-1}D_t\frac{\langle\nabla^k \phi(Z),D_t^{\otimes k}\rangle}{\phi(Z)}\mid \mcl{G}\vee\sigma(Z)\right]\right|^2\mid \mcl{G}\right]\right)^{p/2}.
}
Note that, conditional on $\mcl{G}$, $E_tZ\sim N(0,E_tE_t^{T})$. Thus we have
\[
\E[|E_tZ|^2\mid \mcl{G}]=|E_t|^2.
\]
Meanwhile, we have by Lemma 4.3 of \cite{FaRo15}
\ba{
&\E\left[\left|\E\left[\Lambda^{-1}D_t\frac{\langle\nabla^k\phi(Z),D_t^{\otimes k}\rangle}{\phi(Z)}\mid \mcl{G}\vee\sigma(Z)\right]\right|^2\mid \mcl{G}\right]\\
&=\sum_{j=1}^d\E\left[\left|\frac{\langle\nabla^k\phi(Z),\E[(\Lambda^{-1}D_t)_jD_t^{\otimes k}\mid \mcl{G}]\rangle}{\phi(Z)}\right|^2\mid \mcl{G}\right]\\
&\leq k!\sum_{j=1}^d\left|\E[(\Lambda^{-1}D_t)_jD_t^{\otimes k}\mid \mcl{G}]\right|^2
=k!\left|\E[(\Lambda^{-1}D_t)\otimes D_t^{\otimes k}\mid \mcl{G}]\right|^2.
}
Consequently, we obtain
\ba{
\|E_tZ\|_p\leq\kappa_p\|E_t\|_p
}
and
\ba{
\left\|\E\left[\Lambda^{-1}D_t\frac{\langle\nabla^k \phi(Z),D_t^{\otimes k}\rangle}{\phi(Z)}\mid \mcl{G}\vee\sigma(Z)\right]\right\|_p
\leq\kappa_p^k\sqrt{k!}\left\|\E[(\Lambda^{-1}D_t)\otimes D_t^{\otimes k}\mid \mcl{G}]\right\|_p.
}
Inserting these estimates into \eqref{step1-eq1}, we obtain \eqref{step1}.

%\noindent\textbf{Step 2}. 
%We prove \eqref{n1} in this step. We have
%\ba{
%|R_t|&\leq|R|+\gamma_t^{-3}\E[|\Lambda^{-1}D||D|^3\mid\mcl{G}],&
%|E_t|&\leq|E|+\gamma_t^{-2}\E[|\Lambda^{-1}D||D|^3\mid \mcl{G}],
%}
%and
%\ba{
%|\E[(\Lambda^{-1}D_t)\otimes D_t^{\otimes k}\mid \mcl{G}]|
%\leq\gamma_t^{k-3}\E[|\Lambda^{-1}D||D|^3\mid \mcl{G}]
%}
%for any integer $k\geq3$. Inserting these estimates into \eqref{step1}, we obtain
%\ba{
%\|\rho_t(F_t)\|_p
%\leq e^{-t}\left(
%\|R\|_p+\frac{1}{\sqrt{e^{2t}-1}}\|E\|_p
%+(2+C_p)\frac{\|\E[|\Lambda^{-1}D||D|^3\mid \mcl{G}]\|_p}{(e^{2t}-1)^{3/2}}
%\right),
%}

\noindent\textbf{Step 2}. 
We have by \cref{psd}
\ba{
|E_t|&\leq|E|+(2\lambda)^{-1}\min\{|\E[D^{\otimes2}\mid \mcl{G}]|,\eta_t^{-2}|\E[D^{\otimes2}|D|^2\mid \mcl{G}]|\}\\
&\leq|E|+(2\lambda)^{-1}\min\{2\lambda (|E|+\sqrt d),\eta_t^{-2}|\E[D^{\otimes2}|D|^2\mid \mcl{G}]|\}\\
&\leq2|E|+(2\lambda)^{-1}\min\{2\lambda\sqrt d,\eta_t^{-2}|\E[D^{\otimes2}|D|^2\mid \mcl{G}]|\}.
}
We also have by \cref{psd,tensor}
\ba{
|\E[D_t^{\otimes (k+1)}\mid \mcl{G}]|
&\leq|\E[D_t^{\otimes 2}|D_t|^{k-1}\mid \mcl{G}]|
=|\E[D^{\otimes 2}|D|^{k-1}1_{\{|D|\leq\eta_t\}}\mid \mcl{G}]|\\
&\leq\min\{\eta_t^{k-1}|\E[D^{\otimes 2}\mid \mcl{G}]|,\eta_t^{k-3}|\E[D^{\otimes 2}|D|^2\mid \mcl{G}]|\}\\
&\leq 2\lambda\eta_t^{k-1}|E|+\min\{2\lambda\eta_t^{k-1}\sqrt{d},\eta_t^{k-3}|\E[D^{\otimes 2}|D|^2\mid \mcl{G}]|\}.
}
Inserting these estimates into \eqref{step1} and noting $\kappa_p\leq e\sqrt p$ as well as $\sum_{k=3}^\infty |a_k|e^k\sqrt{k!}<\infty$, we obtain the desired result. 
\end{proof}

\subsection{Proof of \cref{t1}}\label{sec:proofthm2.1}

%As in \cite{Bo20}, we may assume that the law of $W$ has a density $h$ with respect to $N(0,I_d)$ such that $h=\eta+f$ with $\eta>0$ a constant and $f$ a compactly supported $C^\infty$ function. 
%\blue{I will formally verify this claim.} 
%\red{Seems OK as far as I checked following Sec.8 of Bonis'20.} 
%If not, we follow \cite[Section 8]{Bo20} and define
%\be{
%\tilde W_t=I Z'+(1-I)(U+W_t^R),
%}
%where $I\sim \text{Bernoulli}(\eps)$, $\eps>0$, $Z'\sim N(0,I_d)$, $U=\eps N$, $N\in \mathbb{R}^d$ is a random vector with smooth density and taking values in the ball of radius 1, $W_t^R$ is the orthogonal projection of $W_t$ on the Euclidean ball centered at 0 and with radius $R>0$, and $(W_t)_{t\geq 0}, Z', N, I$ are independent. Obtain a similar result as \eq{step1} for the exchangeable pair $(\tilde W_0, \tilde W_t)$ and show that $\norm{\rho_t(\tilde F_t)}_p$ is bounded by the right-hand side of \eq{step1} plus an error term that vanishes after integration $\int_0^\infty$ as $\eps\to 0, R\to \infty$.
By Eq.(3.8) of \cite{LeNoPe15},
\ben{\label{ov-est}
\mcl{W}_p(W,Z)\leq\int_0^\infty\|\rho_t(F_t)\|_pdt,\quad p\geq 1.
}
Strictly speaking, this bound was only proved when $W$ has a bounded $C^\infty$ density $h$ with respect to $N(0,I_d)$ such that $h\geq\eta$ for some constant $\eta>0$ and $|\nabla h|$ is bounded (cf.~Eq.(32) of \cite{OV00}). However, this restriction can be removed by a similar argument as in Section 8 of \cite{Bo20}. For completeness, we give a formal proof in \cref{appendix:ov-est} of the supplementary material.

\eqref{abst-est} follows by combining \eq{ov-est} with \cref{p1}. 
%Therefore, we obtain by \cref{p1}
%\bmn{\label{ov-est}
%\mcl{W}_p(W,Z)\leq C\int_0^\infty e^{-t}\|R_t\|_pdt
%+C\sqrt{p} \|E\|_p\int_0^\infty\frac{e^{-t}}{\sqrt{e^{2t}-1}}dt\\
%+C \int_0^\infty e^{-t}\min\left\{\frac{\sqrt{pd}}{\sqrt{e^{2t}-1}},\frac{p^{3/2}\|\E[D^{\otimes 2}|D|^2\mid \mcl{G}]\|_p}{\lambda(e^{2t}-1)^{3/2}}\right\}dt.
%}

Next, take $\eps>0$ arbitrarily. We have
\ba{
&\int_0^\infty e^{-t}\min\left\{\frac{\sqrt{d}}{\eta_t}, \frac{\|\E[D^{\otimes 2}|D|^2\mid \mcl{G}]\|_p}{\lambda\eta_t^{3}}\right\}dt\\
&\leq \sqrt{pd}\int_0^{\eps}\frac{e^{-t}}{\sqrt{e^{2t}-1}}dt
+\frac{p^{3/2}\|\E[D^{\otimes 2}|D|^2\mid \mcl{G}]\|_p}{\lambda}\int_{\eps}^\infty \frac{e^{-t}}{(e^{2t}-1)^{3/2}}dt.
}
Since
\ba{
\int_0^{\eps} \frac{e^{-t}}{\sqrt{e^{2t}-1}}dt
\leq\int_0^{\eps} \frac{1}{\sqrt{2t}}dt=\sqrt{2\eps}
}
and
\ba{
\int_{\eps}^\infty \frac{e^{-t}}{(e^{2t}-1)^{3/2}}dt
\leq\int_{\eps}^\infty \frac{1}{(2t)^{3/2}}dt=\frac{1}{\sqrt{2\eps}},
}
taking 
\[
\eps=\frac{p \|\E[D^{\otimes2}|D|^2\mid\mcl{G}]\|_p}{2\sqrt{d}\lambda},
\]
we obtain
\ba{
\int_0^\infty e^{-t}\min\left\{\frac{\sqrt{d}}{\eta_t}, \frac{\|\E[D^{\otimes 2}|D|^2\mid \mcl{G}]\|_p}{\lambda\eta_t^{3}}\right\}dt
\leq C p d^{1/4}\sqrt{\frac{\|\E[D^{\otimes2}|D|^2\mid\mcl{G}]\|_p}{\lambda}}.
}
Also, observe that
\[
\int_0^\infty\frac{e^{-t}}{\sqrt{e^{2t}-1}}dt=\frac{1}{2}\int_0^1 \frac{1}{\sqrt{1-x}}dx=1.
%\quad\text{and}\quad \int_0^\infty e^{-t}dt=1.
\]
Inserting these estimates into \eqref{abst-est}, we obtain \eq{027}.

\section{More proofs}\label{sec:appendix}

\subsection{Generalized exchangeable pairs}\label{appendix:generalexch}

Here we record a $p$-Wasserstein bound for generalized exchangeable pairs.
Let $\mathcal{X}$ be a general space and suppose $(X, X')$ is an exchangeable pair of $\mathcal{X}$-valued random variables. Let $W:=W(X)\in \mathbb{R}^d$ be the random vector of interest, $W':=W(X')$ and $D:=W'-W$.
Suppose there exists an antisymmetric function $G:=G(X, X')\in \mathbb{R}^d$ (i.e., $G(X, X')=-G(X', X)$ a.s.) such that
\ben{\label{f01}
\E(G\mid\sigma(X))=-(W+R).
}
Suppose the law of $W$ is approximately $N(0, I_d)$ and
%(\red{$I_d$ may be changed to an invertible matrix, but I guess it is ok not to consider that.} \blue{I agree with you.}) and $Z\sim N(0, I_d)$, 
we are interested in bounding
\be{
\mathcal{W}_p (W, Z).
}

The formulation \eq{f01} with $d=1$ was first proposed by \cite{Ch07} for concentration inequalities (see also \cite{Zh19} for Kolmogorov bounds). In Corollary 2.11 of \cite{Do20} for 1-Wasserstein bounds, he considered the case $d=1$, $W=\sum_{l=1}^m W_l$ and $\E[W_l'-W_l\mid X]=-\lambda_l W_l$. In this case, we can choose $G$ in \eq{f01} to be $G=\sum_{l=1}^m \frac{W_l'-W_l}{\lambda_l}$. For $d>1$, the setting of \cite{ReRo09} corresponds to $G=\Lambda^{-1} (W'-W)$. 

\begin{theorem}\label{t3}
Under the above setting,
assume that $\E|W|^p<\infty$ for some $p\geq1$ and $\E|G||D|^3<\infty$.  
Then we have
\ba{
\mathcal{W}_p(W,Z)
&\leq C\left( \int_0^\infty e^{-t}\|R_t\|_pdt
+\sqrt{p}\|E\|_p
+p \sqrt{\|\E[|G||D| \mid \sigma(X)] \|_p \|\E[|G||D|^3\mid \sigma(X)]\|_p} \right),
}
where $Z\sim N(0,I_d)$ is a $d$-dimensional standard Gaussian vector, 
\be{
R_t:=R+\E[G 1_{\{|D|>\sqrt{(e^{2t}-1)/p}\}}\mid \sigma(X)],\qquad
E:=\frac{1}{2}\E[G\otimes D\mid \sigma(X)]-I_d,
}
and $C$ is an absolute constant.
\end{theorem}
%\blue{I thought you wrote $\sigma(X)$ to avoid confusion of $\mid X$ with $|X|$. $\E[|G||D|^3\mid X]$ looks confusing with $\E[|G||D|^3|X|]$.}

\begin{proof}[Proof of \cref{t3}]
The proof is a straightforward modification of that of \cref{t1}. We use the notation therein. 
Let 
\be{
G_t:=G 1_{\{|D|\leq \eta_t\}}
}
We start from the identity
\be{
\E[G_t\{T_t h(W)+T_t h(W_t)\}]=0.
}
Following the proof of \cref{p1} except that we change $\Lambda^{-1} D_t$ therein by $G_t$ and use $|\E[Y_1\otimes\cdots \otimes Y_k \mid \sigma(X)]|\leq \E[|Y_1\otimes\cdots \otimes Y_k| \mid\sigma(X)]= \E[|Y_1|\cdots |Y_k| \mid \sigma(X)]$ instead of \cref{psd,tensor}, we obtain
\bes{
&\|\rho_t(F_t)\|_p\\
\leq &C e^{-t}\left(
\|R_t\|_p+\frac{\sqrt{p}}{\sqrt{e^{2t}-1}}\|E\|_p
+\min\left\{\frac{\sqrt {p}\norm{\E[|G||D| \mid \sigma(X)]}_p}{\sqrt{e^{2t}-1}}, \frac{p^{3/2}\|\E[|G||D|^3\mid \sigma(X)]\|_p}{(e^{2t}-1)^{3/2}}\right\}
\right).
}
Then, the theorem follows by optimizing the integration as in the proof of \cref{t1}.
\end{proof}

\subsection{Proof for combinatorial CLT}\label{appendix:comb}
\begin{proof}[Proof of \cref{p2}]
In this proof, we use $C$ to denote positive absolute constants, which may differ in different expressions.

\medskip

\textbf{Step 1. The exchangeable pair.}
Let $Y_{ij}=X_{ij}/B_n$ and hence, $W=\sum_{i=1}^n Y_{i\pi(i)}$.
We construct an exchangeable pair $(W, W')$ by uniformly selecting two different indices $I, J\in \{1,\dots, n\}$, independent of $\mathbb{X}$ and $\pi$, and let
\be{
W'=W+D=W-Y_{I \pi(I)}-Y_{J \pi(J)}+Y_{I \pi(J)}+Y_{J \pi(I)}.
}
Let $\mcl{G}=\sigma(\mathbb{X}, \pi)$.
It is know that (cf. Eq. (3.3) of \cite{CF15})
\ben{\label{010}
\E(W'-W|\mcl{G})=-\lambda(W+R),
}
where
\be{
\lambda=\frac{2}{n-1},\quad R=-\frac{1}{n}\sum_{i,j=1}^n Y_{ij}.
}
For $1\leq i\ne j\leq n$, let
\be{
Y^{(ij)}_\pi:=-Y_{i \pi(i)}-Y_{j \pi(j)}+Y_{i \pi(j)}+Y_{j \pi(i)}.
} 
For $t>0$ and $p\geq 2$, let $\eta_t(p)=\sqrt{(e^{2t}-1)/p}$ be as in \cref{t1}.
For any given permutation $\pi$, because of the assumption $\norm{X_{ij}}_{\psi_1}\leq b$, we have, following the same argument as in \cref{subsec:indep} for the independent case, 
\ben{\label{038}
\norm{Y^{(ij)}_\pi 1_{\{|Y^{(ij)}_\pi|>\eta_t(p)\}}}_{\psi_{1/2}}\leq C \eta_t^{-1}(p) \frac{b^2}{B_n^2}, 
}
\ben{\label{039}
\norm{(Y^{(ij)}_\pi)^2}_{\psi_{1/2}}\leq \frac{C b^2}{B_n^2},
}
\be{
\norm{(Y^{(ij)}_\pi)^4 1_{\{|Y^{(ij)}_\pi|\leq \eta_t(p)\}}}_{\psi_{1/2}}\leq C \eta^2_t(p) \frac{b^2}{B_n^2}.
}
We will apply the $p$-Wasserstein bound \eq{abst-est}, which we recall:
\be{
\mathcal{W}_p(W,Z)
\leq C\int_0^\infty e^{-t}\left( \|R_t\|_p
+\frac{\|E\|_p}{\eta_t(p)}
+\min\left\{\frac{1}{\eta_t(p)}, \frac{\|\E[D^4 1_{\{|D|\leq\eta_t(p)\}}\mid \mcl{G}]\|_p}{\lambda \eta^3_t(p)}\right\} \right)dt,
}
where
\be{
R_t:=R+\E[\lambda^{-1}D1_{\{|D|>\eta_t(p)\}}\mid\mcl{G}],\qquad
E:=\frac{1}{2}\E[\lambda^{-1}D^2\mid\mathcal{G}]-1.
}

\medskip

\textbf{Step 2. Bounding $R_t$.}
For the above exchangeable pair, we have
\be{
R_t=-\frac{1}{n}\sum_{i,j=1}^n Y_{ij}+\frac{1}{n}\sum_{1\leq i<j\leq n} Y^{(ij)}_\pi 1_{\{|Y^{(ij)}_\pi|>\eta_t(p)\}}.
}
Because of centering (i.e., $c_{i\cdot}=c_{\cdot j}=0$), we have
\be{
\frac{1}{n}\sum_{i,j=1}^n Y_{ij}=\frac{1}{n} \sum_{i,j=1}^n (Y_{ij}-\E Y_{ij}).
}
From \cref{lem:weibull} and $\norm{Y_{ij}}_{\psi_1}\leq b/B_n$, we have
\be{
\norm{\frac{1}{n}\sum_{i,j=1}^n Y_{ij}}_p\leq \frac{C b}{n  B_n} (\sqrt{pn^2}+p)\leq C(\frac{p\sqrt{n}}{B_n^2}+\frac{p^{5/2}}{B_n^2})b^2,
}
where we used $B_n^2\leq C n b^2$ from \eq{004} in the last inequality.

To deal with the second term in $R_t$, we separate $\sum_{1\leq i<j\leq n}$ into $O(n)$ sums, each sum is over a collection of $O(n)$ disjoint pairs $(i,j)$.
For example, $\{1\leq i<j\leq n\}=\cup_{l=1}^{n-1}(\mcl{I}_l^{(1)}\cup \mcl{I}_l^{(2)})$, where
\be{
\mcl{I}_l^{(1)}=\{1\leq i<j\leq n: j-i=l, i\in \{kl+1,\dots, (k+1)l\}, k\geq 0\ \text{an odd integer}\},
}
\be{
\mcl{I}_l^{(2)}=\{1\leq i<j\leq n: j-i=l, i\in \{kl+1,\dots, (k+1)l\}, k\geq 0\ \text{an even integer}\}.
}
Consider such a sum
\be{
\sum_{(i,j)\in \mcl{I}} Y^{(ij)}_\pi 1_{\{|Y^{(ij)}_\pi|>\eta_t(p)\}}.
}
Conditioning on the unordered pair $\{\pi(i), \pi(j)\}$ for all $(i,j)\in \mcl{I}$, it is a sum of $O(n)$ independent random variables, each with mean 0 and $\norm{\cdot}_{\psi_{1/2}}\leq C\eta_t^{-1}(p) b^2/B_n^2$ (cf. \eq{038}). From \cref{lem:weibull}, we obtain
\be{
\norm{\sum_{(i,j)\in \mcl{I}} Y^{(ij)}_\pi 1_{\{|Y^{(ij)}_\pi|>\eta_t(p)\}}}_p\leq C\eta_t^{-1}(p)\frac{b^2}{B_n^2} (\sqrt{pn}+p^2).
}
Combining the above bounds, we obtain
\be{
\int_0^\infty e^{-t}\norm{R_t}_p dt\leq C(\frac{p\sqrt{n}}{B_n^2}+\frac{p^{5/2}}{B_n^2})b^2.
}

\medskip

\textbf{Step 3. Bounding $E$.}
Note that
\bes{
&E:=\frac{1}{2\lambda}\E[D^2|\mcl{G}]-1\\
=&\frac{1}{2\lambda}\E[D^2|\mcl{G}]-\frac{1}{2\lambda}\E[D^2]+\frac{1}{2\lambda}\E[D^2]-1\\
=&\frac{1}{2n}\sum_{1\leq i< j\leq n}\left[ (Y^{(ij)}_\pi)^2-\E(Y^{(ij)}_\pi)^2 \right]+\frac{1}{2\lambda}\E[D^2]-1\\
=:&H_{21}+H_{22}.
}
From exchangeability and the linearity condition \eq{010}, we obtain
\bes{
H_{22}=&\frac{1}{2\lambda}\E(W'-W)^2-1=\frac{1}{2\lambda}(-2\E[(W'-W)W])-1\\
=&\E(RW)=-\frac{1}{n} \E\left[\sum_{i,j=1}^n Y_{ij} \sum_{k=1}^n Y_{k\pi(k)} \right]
=-\frac{1}{n^2} \E\left[\sum_{i,j=1}^n Y_{ij} \sum_{k,l=1}^n Y_{kl} \right].
}
From \eq{004}, we have
\be{
|H_{22}|= \frac{1}{n^2} \E(\sum_{i,j=1}^n Y_{ij})^2=\frac{1}{n^2} \Var(\sum_{i,j=1}^n Y_{ij})\leq \frac{1}{n}.
}
Now we turn to bounding $H_{21}$.
Write 
\be{
H_{21}=\frac{1}{2n}\sum_{1\leq i< j\leq n}\left[ (Y^{(ij)}_\pi)^2-\E^\pi(Y^{(ij)}_\pi)^2 \right]+\frac{1}{2n}\sum_{1\leq i< j\leq n}\left[ \E^\pi (Y^{(ij)}_\pi)^2-\E(Y^{(ij)}_\pi)^2 \right],
}
where $\E^\pi$ denotes the conditional expectation given the permutation $\pi$.
From a similar argument as in bounding $R_t$ and using \eq{039} for the first term, we obtain
\be{
\norm{\frac{1}{2n}\sum_{1\leq i< j\leq n}\left[ (Y^{(ij)}_\pi)^2-\E^\pi(Y^{(ij)}_\pi)^2 \right]}_p\leq C(\frac{\sqrt{pn}}{B_n^2}+\frac{p^2}{B_n^2})b^2.
}
Now we turn to bounding the second term of $H_{21}$.
Let 
\be{
\xi_{ij}:=  \frac{ \E^\pi(Y^{(ij)}_\pi)^2-\E(Y^{(ij)}_\pi)^2}{n^{3/2} b^2/B_n^2 },
}
and hence,
\be{
\frac{1}{2n}\sum_{1\leq i< j\leq n}\left[ \E^\pi (Y^{(ij)}_\pi)^2-\E(Y^{(ij)}_\pi)^2 \right]=\frac{n^{1/2}b^2}{2B_n^2}\sum_{1\leq i< j\leq n} \xi_{ij}.
}
In the remainder of this step, we show that with $V=\sum_{1\leq i< j\leq n} \xi_{ij}$ and if $p\geq 2$, we have
\ben{\label{014}
\norm{V}_p\leq C(\sqrt{p}+\frac{p}{\sqrt{n}}),
}
and hence
\be{
\int_0^\infty e^{-t} \frac{\norm{E}_p}{\eta_t(p)}dt\leq C(\frac{p\sqrt{n}}{B_n^2}+\frac{p^{5/2}}{B_n^2})b^2,
}
where we used $B_n^2\leq Cnb^2$ again to simplify the upper bound.
To prove \eq{014}, let $h(t)=\E e^{tV}$. We have
\ben{\label{011}
h'(t)=\sum_{1\leq i< j\leq n}\E \xi_{ij} e^{tV}=\sum_{1\leq i< j\leq n}\frac{1}{n(n-1)}\sum_{1\leq k\ne l\leq n}\E\{ \E[\xi_{ij} e^{tV}|\pi(i)=k, \pi(j)=l]\}.
}
It is known that we can define a new permutation $\pi_{ijkl}$ such that it differs from $\pi$ only in absolutely bounded finite number of arguments and (cf. (3.14) of \cite{CF15})
\ben{\label{012}
\mcl{L}(\pi_{ijkl})=\mcl{L}(\pi|\pi(i)=k, \pi(j)=l).
}
Let
\be{
V_{ijkl}=\sum_{1\leq u< v\leq n}\frac{1}{n^{3/2} b^2/B_n^2} \Big[ \E^{\pi_{ijkl}}(Y^{(uv)}_{\pi_{ijkl}})^2   -\E(Y^{(uv)}_\pi)^2 \Big].
}
From its construction and the bound $|\xi_{ij}|\leq C/n^{3/2}$, we have
\ben{\label{013}
|V_{ijkl}-V|\leq C n \frac{1}{n^{3/2}}=\frac{C}{\sqrt{n}}.
}
From \eq{011}, \eq{012} and \eq{013}, we have, for absolutely bounded $|t|/\sqrt{n}$,
\bes{
h'(t)=&\sum_{1\leq i< j\leq n}\frac{1}{n(n-1)}\sum_{1\leq k\ne l\leq n} \E \Big\{ [ e^{tV_{ijkl}}-e^{tV}]\\
&\times \frac{\E(-Y_{i k}-Y_{j l}+Y_{i l}+Y_{j k})^2-\E(-Y_{i \pi(i)}-Y_{j \pi(j)}+Y_{i \pi(j)}+Y_{j \pi(i)})^2}{n^{3/2} b^2/B_n^2} \Big\}
}
and
\be{
|h'(t)|\leq C n^2 \frac{1}{n^2} n^2 \frac{1}{\sqrt{n}} |t| \E[ e^{tV}] \frac{1}{n^{3/2}}\leq C|t| h(t).
}
This implies 
\ben{\label{015}
h(t)=Ee^{tV}\leq e^{Ct^2}\  \text{for absolutely bounded}\  |t|/\sqrt{n}.
}
\eqref{015} means that $V$ is sub-gamma with variance factor $C$ and scale parameter $1/\sqrt{n}$ in the sense of \cite[Section 2.4]{BLM13}. Then, by Theorem 2.3 in \cite{BLM13} and Stirling's formula, 
\[
\norm{V}_p\leq C(\sqrt{p}+p/\sqrt{n}),\quad p\geq 2
\]
which is \eq{014}.
%\blue{I was wrong when I said the second term of $H_{21}$ is just a variance computation so I made the above change.}

\medskip

\textbf{Step 4. Bounding $D^4$.}
We have
\bes{
&\lambda^{-1}\E[D^4 1_{\{|D|\leq \eta_t(p)\}}|\mcl{G}]=\frac{1}{n}\sum_{1\leq i<j\leq n} [(Y^{(ij)}_\pi)^4 1_{\{|Y^{(ij)}_\pi|\leq \eta_t(p)\}}]\\
=&\frac{1}{n}\sum_{1\leq i<j\leq n} \left[(Y^{(ij)}_\pi)^4 1_{\{|Y^{(ij)}_\pi|\leq \eta_t(p)\}}-\E^\pi (Y^{(ij)}_\pi)^4 1_{\{|Y^{(ij)}_\pi|\leq \eta_t(p)\}}\right]\\
&+ \frac{1}{n}\sum_{1\leq i<j\leq n} \E^\pi(Y^{(ij)}_\pi)^4 1_{\{|Y^{(ij)}_\pi|\leq \eta_t(p)\}}.
}
Following a similar argument as in the previous two steps, we obtain
\bes{
&\int_0^\infty e^{-t} \min\left\{\frac{1}{\eta_t(p)}, \frac{\|\E[D^4 1_{\{|D|\leq\eta_t(p)\}}\mid \mcl{G}]\|_p}{\lambda \eta^3_t(p)}\right\} dt\\
\leq &C\int_0^\infty e^{-t}\frac{p\sqrt n/B_n^2+p^{5/2}/B_n^2}{\sqrt{e^{2t}-1}} b^2 dt
+\int_0^\infty e^{-t}\min\left\{\frac{\sqrt{p}}{\sqrt{e^{2t}-1}},\frac{Cp^{3/2} n b^4}{B_n^4(e^{2t}-1)^{3/2}}\right\}dt\\
\leq &C(\frac{p\sqrt{n}}{B_n^2}+\frac{p^{5/2}}{B_n^2})b^2.
}

Combining all the above bounds proves \eq{006}.
\end{proof}

\subsection{Proof for moderate deviations on Wiener chaos}\label{appendix:wiener}

%We turn to the proof of Theorem \ref{thm:wiener}. 
Throughout this subsection, $C_q$ denotes a positive constant, which depends only on $q$ and may be different in different expressions. 
For the proof, in addition to Theorem \ref{t4}, we use \cite{La06}'s sharp moment estimates for Gaussian homogeneous sums. For later use in \cref{appendix:homo}, we state the following generalization obtained in \cite{AdWo15}. 
\begin{lemma}[\cite{AdWo15}, Theorem 1.3]\label{lem:latala}
Let $G$ be a standard Gaussian vector in $\mathbb R^n$. Then, for every polynomial $Q:\mathbb R^n\to\mathbb R$ of degree at most $q$ and every $p\geq2$,
\[
C_q^{-1}\sum_{r=1}^q\sum_{\mcl J\in \Pi_r}p^{|\mcl J|/2}\|\E\nabla^rQ(G)\|_{\mcl J}
\leq\|Q(G)-\E Q(G)\|_p\leq C_q\sum_{r=1}^q\sum_{\mcl J\in \Pi_r}p^{|\mcl J|/2}\|\E\nabla^rQ(G)\|_{\mcl J},
\]
where $\nabla^rQ$ is defined by \eqref{eq:nabla} and we regard $\E\nabla^rQ(G)$ as an element of $(\mathbb R^n)^{\odot r}$.
\end{lemma}

The next result follows from \cref{lem:latala} via a standard approximation argument. 
\begin{lemma}\label{lem:wiener}
For any $h\in\mf H^{\odot q}$ and $p\geq2$,
\ben{\label{wiener-ineq}
\|I_q(h)\|_p\leq C_q\sum_{\mcl J\in \Pi_q}p^{|\mcl J|/2}\|h\|_{\mcl J}.
}
\end{lemma}

\begin{proof}
We prove the claim when $\mf H$ is infinite-dimensional; the finite-dimensional case is similar and easier. 
Let $(e_i)_{i=1}^\infty$ be an orthonormal basis of $\mf H$. Then $(e_{i_1}\otimes\cdots\otimes e_{i_q})_{i_1,\dots,i_q=1}^\infty$ is an orthonormal basis of $\mf H^{\otimes q}$. 
For every $n\in\mathbb N$, define
\[
h_n:=\sum_{i_1,\dots,i_q=1}^na_{i_1,\dots,i_q}e_{i_1}\otimes\dots\otimes e_{i_q},
\]
where $a_{i_1,\dots,i_q}=\langle h,e_{i_1}\otimes\cdots\otimes e_{i_q}\rangle_{\mf H^{\otimes q}}$. Then we have $\|h_n-h\|_{\mf H^{\otimes q}}\to0$ as $n\to\infty$. By hypercontractivity (cf.~Theorem 2.7.2 of \cite{NoPe12}), this implies $\|I_q(h_n)-I_q(h)\|_p\to0$ as $n\to\infty$. 
Also, it is straightforward to check that $\|h_n-h\|_{\mcl J}\to0$ as $n\to\infty$ for all $\mcl J\in\Pi_q$. 
Therefore, it suffices to prove \eqref{wiener-ineq} with $h$ replaced by $h_n$. 

By Theorems 2.7.7 and 2.7.10 in \cite{NoPe12}, we have $I_q(h_n)=Q(X(e_1),\dots,X(e_n))$ for some polynomial $Q:\mathbb R^n\to\mathbb R$ of degree at most $q$. %\red{(without loss of generality, assume $I_q(h_n)=0$)}. 
%\red{all the $h$ below should be changed to $h_n$?} \blue{Yes!}
Then, for any $j_1,\dots,j_r\in[n]$,
\[
\partial_{j_1,\dots,j_r}Q(X(e_1),\dots,X(e_n))=\langle D^rI_q(h_n),e_{j_1}\otimes\cdots\otimes e_{j_r}\rangle_{\mf H^{\otimes r}}. 
\]
Since $\E D^rI_q(h_n)=0$ if $r<q$ and $D^qI_q(h_n)=q! h_n$, we obtain
\[
\E\nabla^r Q(X(e_1),\dots,X(e_n))
=\begin{cases}
0 & \text{if }r<q,\\
q!A & \text{if }r=q,
\end{cases}
\]
where $A=(a_{i_1,\dots,i_q})_{1\leq i_1,\dots,i_q\leq n}$. Regarding $A$ as an element of $(\mathbb R^n)^{\odot q}$, we can easily check that $\|A\|_{\mcl J}=\|h_n\|_{\mcl J}$ for all $\mcl J\in\Pi_q$. Thus, the desired result follows from \cref{lem:latala}.
\end{proof}

\begin{proof}[Proof of Theorem \ref{thm:wiener}]
According to Theorem \ref{t4}, it suffices to prove
\ben{\label{eq:wiener-wass}
\mathcal W_p(W,Z)\leq C_q\max_{r\in[q-1]}\max_{\mcl J\in\Pi_{2q-2r}}p^{(1+|\mcl J|)/2}\|f\wt\otimes_rf\|_{\mcl J}
}
for all $p\geq2$. 
By Proposition 3.7 in \cite{NPS14}, 
\[
\tau(w)=\E[\langle-DL^{-1}W,DW\rangle_{\mathfrak{H}}\mid W=w],\quad w\in\mathbb{R},
\] 
gives a Stein kernel for $W$ (in the sense that it satisfies Eq.(2.3) in \cite{LeNoPe15} with $\nu$ the law of $W$). 
Hence, using the Stein kernel bound for $p$-Wasserstein distance (cf.~Proposition 3.4(ii) in \cite{LeNoPe15}), we obtain
\[
\mathcal W_p(W,Z)\leq C\sqrt p\|\tau(W)-1\|_p.
\] 
By Eq.(5.2.2) in \cite{NoPe12}, 
\ba{
\tau(W)=\frac{1}{q}\| DW\|_{\mathfrak{H}}^2
=1+q\sum_{r=1}^{q-1}(r-1)!\binom{q-1}{r-1}^2I_{2q-2r}(f\wt\otimes_rf). 
}
Thus, by Minkowski's inequality and \cref{lem:wiener},
\ba{
\|\tau(W)-1\|_p\leq C_q\sum_{r=1}^{q-1}\sum_{\mcl J\in\Pi_{2q-2r}}p^{|\mcl J|/2}\|f\wt\otimes_rf\|_{\mcl J}.
}
Consequently, we obtain \eqref{eq:wiener-wass}.
\end{proof}

\subsection{Proof for homogeneous sums}\label{appendix:homo}

Throughout this section, $C$ denotes a positive absolute constant and $C_q$ denotes a positive constant depending only on $q$, respectively. Note that their values may be different in different expressions. Also, given a function $g:[n]^q\to\mathbb R$, we write
\[
\|g\|=\sqrt{\sum_{i_1,\dots,i_q=1}^ng(i_1,\dots,i_q)^2}.
\]
%Recall that, for every $r\in[q-1]$, the $r$-th contraction $f\otimes_rf:[n]^{2q-2r}\to\mathbb R$ is defined as
%\[
%f\otimes_rf(i_1,\dots,i_{2q-2r})=\sum_{j_1,\dots,j_r=1}^nf(i_1,\dots,i_{q-r},j_1,\dots,j_r)f(i_{q-r+1},\dots,i_{2q-2r},j_1,\dots,j_r).
%\]

We will frequently use the following inequality throughout the proof.  
\begin{lemma}[\cite{AdWo15}, Theorem 1.4]\label{lem:aw-est}
Let $X=(X_1,\dots,X_n)$ be a random vector with independent components. Suppose that there is a constant $K>0$ such that $\|X_i\|_{\psi_2}\leq K$ for all $i=1,\dots,n$. 
Then, for every polynomial $Q:\mathbb R^n\to\mathbb R$ of degree at most $q$ and every $p\geq2$,
\[
\|Q(X)-\E Q(X)\|_p\leq C_q\sum_{r=1}^qK^r\sum_{\mcl J\in \Pi_r}p^{|\mcl J|/2}\|\E\nabla^rQ(X)\|_{\mcl J}.
\]
\end{lemma}

\begin{proof}[Proof of \cref{prop:homo}]
First, note that $\mcl M(f)\leq\|f\|^2=1/q!\leq1/2$. Hence $|\log\mcl M(f)|\geq\log2$ and $p\mcl M(f)\leq p\sqrt{\mcl M(f)}\leq1$. 
%We divide the proof into three steps.
\medskip

\textbf{Step 1. The exchangeable pair.}
Let $X^*=(X^*_1,\dots,X^*_n)$ be an independent copy of $X:=(X_1,\dots,X_n)$. 
Also, let $I\sim \text{Unif}[n]$ be an index independent of $X$ and $X^*$. 
Define $X'=(X'_1,\dots,X'_n)$ by
\be{
X_{i}'=
\begin{cases}
X_{i}^*, & \text{if}\ i=I,\\
X_{i}, & \text{otherwise}.
\end{cases}
}
Then we set
\[
W'=\sum_{i_1,\dots,i_q=1}^nf(i_1,\dots,i_q)X'_{i_1}\cdots X'_{i_q}.
\]
It is easy to check $\mathcal{L}(X, X')=\mathcal{L}(X', X)$; hence, $\mathcal{L}(W, W')=\mathcal{L}(W', W)$. 
Moreover, 
\ba{
D:=W'-W
&=\sum_{\begin{subarray}{c}
i_1,\dots,i_q=1\\
\exists r:i_r=I
\end{subarray}}^nf(i_1,\dots,i_q)(X'_I-X_I)\prod_{r=1:i_r\neq I}^qX_{i_r}\\
&=q(X_I'-X_I)Q_I(X),
}
where, for every $i=1,\dots,n$, $Q_i$ is an $n$-variate polynomial defined as
\[
Q_i(x_1,\dots,x_n):=\sum_{i_2,\dots,i_{q}=1}^nf(i,i_2,\dots,i_{q})x_{i_2}\cdots x_{i_{q}}.
\]
Hence
\ba{
\E[D\mid X]=-\frac{q}{n} W.
}
Therefore, by \cref{c1}
\ben{\label{homo-est}
\mcl{W}_p(W,Z)\leq C\sqrt{p} \norm{E}_p+C p \sqrt{\frac{n}{q}\norm{\E[D^4\mid X]}_p}
=:H_1+H_2,
}
where 
\[
E=\frac{n}{2q}\E[D^2\mid X]-1.
\]

\medskip

\textbf{Step 2. Bounding $H_1$.} Observe that
%$\frac{n}{2q}\E D^2=\E W^2=1$
\ba{
\frac{n}{2q}\E[D^2\mid X]&=\frac{q}{2}\sum_{i=1}^n(1+X_i^2)Q_i(X)^2.
}
Define an $n$-variate polynomial $Q$ as
\[
Q(x_1,\dots,x_n)=\frac{q}{2}\sum_{i=1}^n(1+x_i^2)Q_i(x_1,\dots,x_n)^2.
\]
Observe that $Q$ has total degree $2q$ and degree 2 in $x_i$ for every $i\in[n]$; the latter follows from the fact that $f$ is vanishing on diagonals. Using the latter property, one can easily verify that, with $G\sim N(0,I_n)$, $\E Q(X)=\E Q(G)$ and $\E\nabla^rQ(X)=\E\nabla^rQ(G)$ for all $r=1,\dots,2q$. 
Hence, by Lemmas \ref{lem:latala} and \ref{lem:aw-est},
\ben{\label{homo-E-gauss}
H_1\leq C_q\sqrt pK^{2q}\norm{Q(G)-1}_p.
}
Let $e_1,\dots,e_n$ be the standard basis of $\mathbb R^n$. Without loss of generality, we may assume that $G_i=\mathbf G(e_i)$ $(i=1,\dots,n)$ for some isonormal Gaussian process $\mathbf G$ over $\mf H=\mathbb R^n$. Then, for every $i=1,\dots,n$, we have
\[
Q_i(G)=I_{q-1}(\mathbf f_i),
\]
where $I_q$ denotes the $q$-th multiple Wiener--It\^ointegral with respect to $\mathbf G$ and
\[
\mathbf f_i:=\sum_{i_2,\dots,i_{q}=1}^nf(i,i_2,\dots,i_{q})e_{i_2}\otimes\cdots\otimes e_{i_q}.
\]
Thus we obtain
\ba{
Q(G)-1
&=\frac{q}{2}\sum_{i=1}^n(1+G_i^2)I_{q-1}(\mathbf f_i)^2-1\\
&=\frac{q}{2}\sum_{i=1}^n(G_i^2-1)I_{q-1}(\mathbf f_i)^2
+\left\{q\sum_{i=1}^nI_{q-1}(\mathbf f_i)^2-1\right\}
=:H_{11}+H_{12}.
}

To evaluate $H_{11}$, observe that $G_i^2-1=I_2(e_i^{\otimes2})$ by Theorem 2.7.7 in \cite{NoPe12}. Also, by the product formula for multiple Wiener--It\^ointegrals (cf.~Theorem 2.7.10 in \cite{NoPe12}), 
\be{
I_{q-1}(\mathbf f_i)^2=\sum_{r=0}^{q-1}r!\binom{q-1}{r}^2I_{2q-2-2r}(\mathbf f_i\wt{\otimes}_r\mathbf f_i).
}
%\blue{equation label was changed to the next one}
Using the product formula again and noting that $f(i,i_2,\dots,i_q)=0$ if $i_r=i$ for some $r$ as well as $e_i\cdot e_j=0$ if $i\ne j$, we obtain
\ben{\label{homo-E-I-expression}
H_{11}=\frac{q}{2}\sum_{r=0}^{q-1}r!\binom{q-1}{r}^2I_{2q-2r}\left(\sum_{i=1}^ne_i^{\otimes2}\wt\otimes(\mathbf f_i\wt{\otimes}_r\mathbf f_i)\right).
}
Let $r\in\{0,1,\dots,q-1\}$ be fixed. By \cref{lem:wiener},
\ben{\label{eq:mwi-est}
\lnorm{I_{2q-2r}\left(\sum_{i=1}^ne_i^{\otimes2}\wt\otimes(\mathbf f_i\wt{\otimes}_r\mathbf f_i)\right)}_p
\leq C_q\sum_{\mcl J\in\Pi_{2q-2r}}p^{|\mcl J|/2}\lnorm{\sum_{i=1}^ne_i^{\otimes2}\wt\otimes(\mathbf f_i\wt{\otimes}_r\mathbf f_i)}_{\mcl J}.
}
Observe that 
\ba{
\lnorm{\sum_{i=1}^ne_i^{\otimes2}\wt\otimes(\mathbf f_i\wt{\otimes}_r\mathbf f_i)}_{\{1\},\dots,\{2q-2r\}}
&\leq\sup_{u\in\mathbb R^n:|u|\leq1}\sum_{i=1}^nu_i^2\norm{\mathbf f_i\wt{\otimes}_r\mathbf f_i}_{\mf H^{\otimes(2q-2r-2)}}
}
and
\ba{
\lnorm{\sum_{i=1}^ne_i^{\otimes2}\wt\otimes(\mathbf f_i\wt{\otimes}_r\mathbf f_i)}_{\mcl J}
\leq\lnorm{\sum_{i=1}^ne_i^{\otimes2}\wt\otimes(\mathbf f_i\wt{\otimes}_r\mathbf f_i)}_{\mf H^{\otimes(2q-2r)}}
\leq\sqrt{\sum_{i=1}^n\norm{\mathbf f_i\wt{\otimes}_r\mathbf f_i}_{\mf H^{\otimes(2q-2r-2)}}^2}
}
for any $\mcl J\in\Pi_{2q-2r}$. 
By the Cauchy--Schwarz inequality,
\ba{
\norm{\mathbf f_i\wt{\otimes}_r\mathbf f_i}_{\mf H^{\otimes(2q-2r-2)}}
%\leq\sqrt{\sum_{j_1,\dots,j_{2q-2-2r}=1}^nf\otimes_rf(j_1,\dots,j_{2q-2-2r})^2}
\leq\sum_{i_2,\dots,i_q=1}^nf(i,i_2,\dots,i_q)^2.
}
Hence we obtain
\ba{
\lnorm{\sum_{i=1}^ne_i^{\otimes2}\wt\otimes(\mathbf f_i\wt{\otimes}_r\mathbf f_i)}_{\{1\},\dots,\{2q-2r\}}
&\leq\mcl M(f)
}
and
\ba{
\lnorm{\sum_{i=1}^ne_i^{\otimes2}\wt\otimes(\mathbf f_i\wt{\otimes}_r\mathbf f_i)}_{\mcl J}
\leq\sqrt{\mcl M(f)\sum_{i_1,\dots,i_q=1}^nf(i_1,\dots,i_q)^2}
=\sqrt{\frac{1}{q!}\mcl M(f)}
}
for any $\mcl J\in\Pi_{2q-2r}$. Inserting these estimates into \eqref{eq:mwi-est}, we deduce
\[
\lnorm{I_{2q-2r}\left(\sum_{i=1}^ne_i^{\otimes2}\wt\otimes(\mathbf f_i\wt{\otimes}_r\mathbf f_i)\right)}_p
\leq C_q\left(p^{q-r-1/2}\sqrt{\mcl M(f)}+p^{q-r}\mcl M(f)\right).
\]
Combining this bound with \eqref{homo-E-I-expression} and $p\mcl M(f)\leq1$, we obtain
\ben{\label{homo-I-est}
\|H_{11}\|_p\leq C_qp^{q-1/2}\sqrt{\mcl M(f)}.
}

To evaluate $H_{12}$, observe that $I_{q-1}(\mathbf f_i)=q^{-1}\mathbf D I_q(\mathbf f)\cdot e_i$ for every $i=1,\dots,n$, where $\mathbf D$ denotes the Malliavin derivative with respect to $\mathbf G$ and
\[
\mathbf f:=\sum_{i_1,\dots,i_{q}=1}^nf(i_1,\dots,i_{q})e_{i_1}\otimes\cdots\otimes e_{i_q}.
\]
Hence
\ba{
H_{12}=q^{-1}\sum_{i=1}^n(\mathbf D I_q(\mathbf f)\cdot e_i)^2-1=q^{-1}\norm{\mathbf D I_q(\mathbf f)}_{\mf H}^2-1.
}
Therefore, by the proof of \cref{thm:wiener},
\be{%\label{homo-II-est}
\|H_{12}\|_p\leq C_q\sum_{r=1}^{q-1}\sum_{\mcl J\in\Pi_{2q-2r}}p^{|\mcl J|/2}\|\mathbf f\wt\otimes_r\mathbf f\|_{\mcl J}
%=C_q\sum_{r=1}^{q-1}\sum_{\mcl J\in\Pi_{2q-2r}}p^{|\mcl J|/2}\|f\wt\otimes_rf\|_{\mcl J}.
\leq C_qp^{q-1}\max_{r\in[q-1]}\|f\otimes_rf\|,
}
where, for every $r\in[q]$, the function $f\otimes_rf:[n]^{2q-2r}\to\mathbb R$ is defined as
\[
f\otimes_rf(i_1,\dots,i_{2q-2r})=\sum_{j_1,\dots,j_r=1}^nf(i_1,\dots,i_{q-r},j_1,\dots,j_r)f(i_{q-r+1},\dots,i_{2q-2r},j_1,\dots,j_r).
\]
%\red{For $\|f\otimes_rf\|$, subindex is unnecessary because of the notation introduced at the beginning of this section.}
Combining this with Lemma 2.1 in \cite{Ko22},
%\red{Lemma 2.1 of Koike says $M$ instead of $M^q$?} \blue{The statement of this lemma contains a typo ... $M^q$ is correct, cf. Section 7.4 of \cite{Ko22}} 
we obtain
\ben{\label{homo-II-est}
\|H_{12}\|_p\leq C_qp^{q-1}\sqrt{|\E W^4-3|+M^q\mcl M(f)}.
}
By \eqref{homo-E-gauss}, \eqref{homo-I-est} and \eqref{homo-II-est}, we conclude
\ben{\label{homo-h1-est}
H_1\leq 
%C_qK^{2q}\left(
%p^{q}\sqrt{\mcl M(f)}
%+\max_{r\in[q-1]}\max_{\mcl J\in\Pi_{2q-2r}}p^{(1+|\mcl J|)/2}\|f\wt\otimes_rf\|_{\mcl J}
%\right).
C_qp^qK^{2q}\sqrt{|\E W^4-3|+M^q\mcl M(f)}.
}

%\ba{
%\frac{n}{2q}\E[D^2\mid X]=\frac{q}{2}\sum_{i_1,\dots,i_{2q-2},k=1}^nf(i_1,\dots,i_{q-1},k)f(i_q,\dots,i_{2q-2},k)(1+X_k^2)X_{i_1}\cdots X_{i_{2q-2}}
%}
%
%\[
%Q(x_1,\dots,x_n)=\frac{q}{2}\sum_{i_1,\dots,i_{2q-2},k=1}^nf(i_1,\dots,i_{q-1},k)f(i_q,\dots,i_{2q-2},k)(1+x_k^2)x_{i_1}\cdots x_{i_{2q-2}}
%\]
%
%\ba{
%&\E\partial_{j_1,j_2}Q(X_1,\dots,X_n)=1_{\{j_1=j_2\}}q\sum_{i_1,\dots,i_{2q-2}=1}f(i_1,\dots,i_{q-1},j_1)f(i_q,\dots,i_{2q-2},j_1)\E[X_{i_1}\cdots X_{i_{2q-2}}]\\
%&+q(q-1)^2\sum_{i_1,\dots,i_{2q-4},k=1}^nf(i_1,\dots,i_{q-2},j_1,k)f(i_{q-1},\dots,i_{2q-4},j_2,k)\E[(1+X_k^2)X_{i_1}\cdots  X_{i_{2q-4}}]\\
%&=1_{\{j_1=j_2\}}q!f\otimes_{q-1} f(j_1,j_1)
%+2q!(q-1)f\otimes_{q-2} f(j_1,j_2)
%}

\medskip

\textbf{Step 3. Bounding $H_2$.} First, by \cref{lem:aw-est}
\ba{
\norm{Q_i(X)}_s
\leq C_qK^{q-1}s^{(q-1)/2}\sqrt{\Inf_i(f)}
}
for any $i\in[n]$ and $s\geq2$, where 
\[
\Inf_i(f):=\sum_{i_2,\dots,i_q=1}^nf(i,i_2,\dots,i_q)^2.
\]
Hence we have (cf.~Lemma A.4 in \cite{Ko22})
\[
P(|Q_i(X)|\geq t)\leq C_q\exp\left(-\left(\frac{t}{C'_qK^{q-1}\sqrt{\Inf_i(f)}}\right)^{2/(q-1)}\right)
\]
for all $t>0$, where $C_q'>0$ is a constant depending only on $q$. 
Let 
\[
\delta_i:=C'_qK^{q-1}\sqrt{\Inf_i(f)}|pq\log\mcl M(f)|^{(q-1)/2}.
\]
Then, by Lemma 6.1 in \cite{Ko22},
\be{
\E[|Q_i(X)|^s1_{\{|Q_i(X)|>\delta_i\}}]\leq C_q\left(1+\frac{2s-2/(q-1)}{s-2/(q-1)}\right)\{s(q-1)\}^{s(q-1)/2}\delta_i^s\mcl M(f)^{pq}
}
for any $s>2/(q-1)$. Since $2/(q-1)\leq2$, we can apply this inequality with $s=4p$ and then obtain
\ben{\label{eq:q-tail}
\norm{Q_i(X)^41_{\{|Q_i(X)|>\delta_i\}}}_p\leq C_qp^{2(q-1)}\delta_i^4\mcl M(f)^{q}.
}
Now we bound $\frac{n}{q}\E[D^4\mid X]$ as
\besn{\label{homo-h2-decomp}
\frac{n}{q}\E[D^4\mid X]
&=q^3\sum_{i=1}^n\E[(X'_i-X_i)^4\mid X]Q_i(X)^4\\
&\leq q^3\sum_{i=1}^n\E[(X'_i-X_i)^4\mid X]\delta_i^4
+q^3\sum_{i=1}^n\E[(X'_i-X_i)^4\mid X]Q_i(X)^41_{\{|Q_i(X)|>\delta_i\}}\\
&=:H_{21}+H_{22}.
}
We bound $\|H_{21}\|_p$ as
\ba{
\|H_{21}\|_p\leq q^3\sum_{i=1}^n\E[(X'_i-X_i)^4]\delta_i^4+q^3\lnorm{\sum_{i=1}^n\{\E[(X'_i-X_i)^4\mid X]-\E(X_i'-X_i)^4\}\delta_i^4}_p.
}
For the first term, we have
\ba{
q^3\sum_{i=1}^n\E[(X'_i-X_i)^4]\delta_i^4\leq C_q K^4\sum_{i=1}^n\delta_i^4
\leq C_qp^{2q-2}K^{4q}\mcl M(f)|\log\mcl M(f)|^{2(q-1)}.
}
To bound the second term, note that $\|\E[(X'_i-X_i)^4\mid X]\|_{\psi_{1/2}}\leq CK^4$. Therefore, by \cref{lem:weibull},
\ba{
&\lnorm{\sum_{i=1}^n\{\E[(X'_i-X_i)^4\mid X]-\E(X_i'-X_i)^4\}\delta_i^4}_p\\
&\leq CK^4\left(\sqrt{p\sum_{i=1}^n\delta_i^8}+p^2\max_{1\leq i\leq n}\delta_i^4\right)\\
&\leq C_qK^{4q}(p^{2q-3/2}\mcl M(f)^{3/2}+p^{2q}\mcl M(f)^2)|\log\mcl M(f)|^{2(q-1)}\\
&\leq C_qK^{4q}p^{2q-2}\mcl M(f)|\log\mcl M(f)|^{2(q-1)},
}
where in the second inequality we used $\sum_{i=1}^n \Inf_i(f)=1/q!$ and the last inequality follows from the condition $p\mcl M(f)\leq p\sqrt{\mcl M(f)}\leq1$. All together, we obtain
\ben{\label{homo-III-est}
\|H_{21}\|_p\leq C_qK^{4q}p^{2q-2}\mcl M(f)|\log\mcl M(f)|^{2(q-1)}.
}
In the meantime, noting that $(X_i,X_i')$ and $Q_i(X)$ are independent, we have
\ba{
\norm{H_{22}}_p\leq q^3 \sum_{i=1}^n\norm{(X'_i-X_i)^4}_p\norm{Q_i(X)^41_{\{|Q_i(X)|>\delta_i\}}}_p.
}
Using \eqref{eq:q-tail} and $p\sqrt{\mcl M(f)}\leq1$, we obtain
\ben{\label{homo-IV-est}
\norm{H_{22}}_p\leq C_qK^4p^{2q}\mcl M(f)^{q}\sum_{i=1}^n\delta_i^4
\leq C_qK^{4q}p^{2q-2}\mcl M(f)|\log\mcl M(f)|^{2(q-1)}.
}
Combining \eqref{homo-III-est} and \eqref{homo-IV-est} with \eqref{homo-h2-decomp} gives
\ben{\label{homo-h2-est}
H_2\leq Cp\sqrt{H_{21}+H_{22}}
\leq C_qK^{2q}p^{q}\sqrt{\mcl M(f)}|\log\mcl M(f)|^{q-1}.
}
By \eqref{homo-est}, \eqref{homo-h1-est} and \eqref{homo-h2-est}, we complete the proof. 
\end{proof}

\subsection{Proof for moderate deviations in multi-dimensions}\label{appendix:multi}

\begin{proof}[Proof of \cref{thm:multip}]
The proof is almost identical to the arguments leading to \eq{034}, except that we view $Y_i^{\otimes 2}$ as a $d^2$-vector, use $\norm{Y_i^{\otimes 2}}_{H.S.}=|Y_i|^2$ and \cref{lem:weibull} for independent random vectors in $\mathbb{R}^{d^2}$. The factor $d^{1/4}$ comes from the $\sqrt{d}$ term in \eq{abst-est}.
\end{proof}

\begin{proof}[Proof of \cref{thm:multiMD}]
In this proof, we use $C:=C_{A, \alpha, B_1, B_2}$ to denote positive constants, which depend only on $\alpha$, $A$, $B_1$ and $B_2$ and may be different in different expressions.
Let $f(x):=f(x;d)$ denote the density of the chi-distribution with $d$ degrees of freedom, i.e.,
\be{
f(x)=\frac{1}{\kappa(d)} x^{d-1} e^{-x^2/2},\quad \kappa(d):=2^{(d/2)-1}\Gamma(d/2).
}
Note that $\log(\kappa(d))\leq C d \log d$.
For $d\geq 2$ and $x>0$, we have
\be{
\int_x^\infty y^{d-1} e^{-y^2/2}dy=x^{d-2} e^{-x^2/2}+\int_x^{\infty} (d-2) y^{d-3} e^{-y^2/2}dy\geq x^{d-2} e^{-x^2/2}.
}
Therefore,
\ben{\label{031}
\frac{f(x)}{P(|Z|>x)}\leq x.
}
First we prove the claim when $\Delta<1/e$.
Set 
\be{
p=|\log \Delta|+\log (\kappa(d))+\frac{x^2}{2}, \quad \eps=A p^\alpha \Delta e.
} 
Because of the condition $|\log \Delta|\leq p_0/4$, $\log(\kappa(d))\leq p_0/4$ and $x\leq \sqrt{p_0}$, we have $p\leq p_0$.
From the upper bound on $\mcl{W}_p(W, Z)$, we can couple $W$ and $Z$ such that $\norm{W-Z}_p\leq A p^\alpha \Delta$.
We have
\ba{
P(|W|>x)
&\leq P(|Z|>x-\eps)+P(|W-Z|>\eps)\\
&=P(|Z|>x)+P(x-\eps<|Z|\leq x)+P(|W-Z|>\eps).
}
Since
\ba{
P(x-\eps<|Z|\leq x)
=\int_{(x-\eps)\vee 0}^x f(z)dz%\leq\phi((x-\eps)\vee0)\eps
}
and
\ba{
P(|W-Z|>\eps)\leq \eps^{-p}\norm{W-Z}^p_p \leq (A p^\alpha \Delta/\eps)^p=e^{-p}=\Delta\frac{1}{\kappa(d)} e^{-x^2/2},
}
we obtain
\ba{
P(|W|>x)
&\leq P(|Z|>x)+\int_{(x-\eps)\vee 0}^x f(z)dz+\Delta\frac{1}{\kappa(d)} e^{-x^2/2}.
}
Similarly, we deduce
\ba{
P(|Z|>x)
&=P(|Z|>x+\eps)+P(x<|Z|\leq x+\eps)\\
&\leq P(|W|>x)+P(|W-Z|>\eps)+P(x<|Z|\leq x+\eps)\\
&\leq P(|W|>x)+\int_{x}^{x+\eps} f(z)dz+\Delta\frac{1}{\kappa(d)} e^{-x^2/2}.
}
Consequently, we obtain
\be{
|P(|W|>x)-P(|Z|>x)|\leq\int_{(x-\eps)\vee 0}^{x+\eps} f(z)dz+\Delta\frac{1}{\kappa(d)} e^{-x^2/2}.
}
Note that \eq{029} implies $d (\log d) \Delta^{2/(2\alpha+1)}\leq C$.
Therefore, using $x\leq \Delta^{-1/(2\alpha+1)}$, we have
\besn{\label{032}
\eps&\leq C\Delta(|\log\Delta|^{\alpha}+\log^\alpha (\kappa(d))+x^{2\alpha})
\leq C\Delta(|\log\Delta|^{\alpha}+\Delta^{-2\alpha/(2\alpha+1)})\\
&\leq C\Delta^{1-2\alpha/(2\alpha+1)}
=C\Delta^{1/(2\alpha+1)}.
}
%By decreasing $\Delta$ and increasing $A$, we assume in the following without loss of generality that $\eps\leq 1$. 
Note that $\eps\leq C$ and for $0\leq x\leq \Delta^{-1/(2\alpha+1)}$, we have $x\eps\leq C$.
Also note that \eq{029} implies $d\Delta |\log \Delta|^\alpha\leq C d \Delta^{2/(2\alpha+1)}\leq C$ if $\alpha>1/2$.
If $x\geq 1$, 
we have, from \eq{031},
\bes{
&\frac{\int_{(x-\eps)\vee 0}^{x+\eps} f(z)dz}{P(|Z|>x)}\leq C\eps\frac{f(x)}{P(|Z|>x)} e^{x\eps} (\frac{x+\eps}{x})^{d-1}\leq C x\eps e^{d\eps/x}\\
\leq & C x\eps \exp\left\{C\left( d\Delta|\log \Delta|^\alpha + d(d\log d)^\alpha\Delta+d\Delta+d \Delta^{2/(2\alpha+1)}1_{\{\alpha>1/2\}}\right)     \right\}\\
\leq & C x\eps,
}
where we used $1\leq x\leq \Delta^{-1/(2\alpha+1)}$, \eq{029} and \eq{030}.
Therefore,
\be{
\left|\frac{P(|W|>x)}{P(|Z|>x)}-1\right|\leq Cx\eps+\Delta \leq  C(1+x)(|\log\Delta|+d \log d+x^2)^{\alpha}\Delta.
}
If $x<1$, the conclusion follows from $1/P(|Z|>x)\leq C$ and 
\be{
|P(|W|>x)-P(|Z|>x)|\leq C(\eps+\Delta).
}

It remains to prove \eq{036} when $\Delta\geq 1/e$. In this case, we have $x\leq e$ and thus $1/P(|Z|>x)$ is bounded. Hence \eq{036} holds with a sufficiently large $C_{A, \alpha, B_1, B_2}$.
\end{proof}

\subsection{Proof for local dependence}\label{appendix:local}

\begin{proof}[Proof of \cref{thm:local3}]
We adapt the proof of \cref{t1} and use the notation therein. Let $\mcl{G}=\sigma(X_1,\dots, X_n)$.
Let
$I$ be a uniform random index from $\{1,\dots, n\}$ and independent of everything else. Let $D=-n^{-1/2}\sum_{j\in A_I} X_j$. Because $|X_j|\leq b_n$ and $|A_I|\leq \theta_1$, we have $|D|\leq \theta_1 b_n/\sqrt{n}$.
Because $X_i$ is independent of $\{X_j: j\notin A_i\}$, we have 
\be{
\E[(-\sqrt{n}X_I) T_t h(W+D)]=0,
}
and hence,
\be{
\E\left[ (-\sqrt{n}X_I) \left\{ T_t h(W)+\langle \nabla T_t h(W), D\rangle+\sum_{k=2}^\infty \frac{1}{k!} \langle \nabla^k T_t h(W), D^{\otimes k} \rangle    \right\}   \right]=0.
}
Let 
\be{
\tau_t=\E\left[(-\sqrt{n}X_I)\left(1-\frac{\langle\nabla \phi(Z),D\rangle}{\phi(Z)\sqrt{e^{2t}-1}}
+\sum_{k=2}^\infty \frac{1}{k!}\frac{(-1)^k\langle\nabla^k \phi(Z),D^{\otimes k}\rangle}{\phi(Z)(e^{2t}-1)^{k/2}}\right)\mid \mcl{G}\vee\sigma(Z)\right].
}
Following the same argument leading to \eq{step1-eq1}, 
we have
\begin{multline*}
\|\rho_t(F_t)\|_p
\leq e^{-t}\left(\frac{1}{\sqrt{e^{2t}-1}}\|E Z\|_p\right.\\
\left.+\sum_{k=2}^\infty \frac{1}{k! (e^{2t}-1)^{k/2}}\left\|\E\left[(-\sqrt{n} X_I)\frac{\langle\nabla^k \phi(Z),D^{\otimes k}\rangle}{\phi(Z)}\mid \mcl{G}\vee\sigma(Z)\right]\right\|_p\right),
\end{multline*}
where $E=\E[(-\sqrt{n}X_I)\otimes D|\mcl{G}]-I_d$.
Following the same argument as in the proof of \cref{p1}, with $\gamma_t=\sqrt{e^{2t}-1}$,
the first term is bounded by $\frac{C e^{-t}}{\gamma_t} \sqrt{p} \norm{E}_{p}$.
The second term with $k=2$, is bounded by
\be{
 \frac{C e^{-t} p \theta_1^2 b_n^3}{\sqrt{n} \gamma_t^2}. 
}
The second term with $k\geq 3$, if $\gamma_t\geq \theta_1 b_n \sqrt{p/n}$, is bounded by
\be{
\frac{Ce^{-t} p^{3/2}\theta_1^3 b_n^4}{n\gamma_t^3}. 
}
Note that
\be{
d=\E(W^T W)\leq \frac{n\theta_1}{n}b_n^2=\theta_1 b_n^2.
}
Let $t_0$ be such that $\sqrt{e^{2t_0}-1}=\theta_1 b_n\sqrt{p/n}$ and assume it is $\leq c$ for a sufficiently small constant $c>0$ as in the condition \eq{037}. Then, with $W_0:=e^{-t_0} W+\sqrt{1-e^{-2t_0}}Z$, we have
%[The approximation argument modified from the exchangeable pair case:] for the approximation argument, we let
%\be{
%\tilde W_0=e^{-t_0} \tilde W+\sqrt{1-e^{-2t_0}} Z,
%}
%and note that we don't need to project $W$ to $B(0,R)$ because $W$ is assumed to be bounded,
%\be{
%\tilde W=J Z'+(1-J) (U+W),
%}
%\be{
%\tilde W+\tilde D=J Z'+(1-J) (U+W-\frac{1}{\sqrt{n}} X_j).
%}
%We start with the equation
%\be{
%\E[(-\sqrt{n} X_I) T_t h(\tilde W+\tilde D)]=0,
%}
%and let 
%\be{
%\tilde \tau_t=\E[(-\sqrt{n} X_I)(1-\frac{\tilde D}{}-\sum_{k=2}^\infty\frac{\tilde D}{})|\mcl{G}\vee \sigma(Z)\vee \sigma(J)].
%}
%Control $\rho_t(\tilde F_t)$ and show that $\int_{t_0}^\infty \rho_t(\tilde F_t) dt\leq ...+error$, where error tends to 0 as $\eps\to 0$ in $J\sim Ber(\eps)$.
\bes{
\mcl{W}_p(W_0, Z)\leq& \int_{t_0}^\infty \norm{\rho_t(F_t)}_p dt\\
\leq& C\sqrt{p}\norm{E}_p+      C\int_{t_0}^\infty \frac{e^{-t}}{e^{2t}-1} dt\frac{p\theta_1^2 b_n^3}{\sqrt{n}}+C\int_{t_0}^\infty \frac{e^{-t}}{(e^{2t}-1)^{3/2}} dt \frac{p^{3/2} \theta_1^3 b_n^4}{n}  \\
\leq& C\sqrt{p}\norm{E}_p+\frac{Cp \theta_1^2 b_n^3\log n}{\sqrt{n}}. 
}
This implies
\bes{
\mcl{W}_p(W, Z)=&e^{t_0} \mcl{W}_p(e^{-t_0}W, e^{-t_0}Z)\\
\leq& e^{t_0} \mcl{W}_p(e^{-t_0}W, W_0) +e^{t_0} \mcl{W}_p(W_0, Z)+e^{t_0} \mcl{W}_p(Z, e^{-t_0}Z)\\
\leq & C\sqrt{p}\norm{E}_p+\frac{Cp \theta_1^2 b_n^3\log n}{\sqrt{n}}+\frac{C p d^{1/2}  \theta_1 b_n}{\sqrt{n}}\\
\leq & C\sqrt{p}\norm{E}_p+\frac{Cp \theta_1^2 b_n^3\log n}{\sqrt{n}},
}
where we used $\norm{\sqrt{1-e^{-2t_0}}Z}_p\leq C\theta_1 b_n\sqrt{p/n} \sqrt{dp}$ (cf. \cref{hyper}) in the second inequality and $d\leq \theta_1 b_n^2$ in the last inequality.
Note that
\be{
E=\frac{1}{n}\sum_{i=1}^n \sum_{j\in A_i} (X_i\otimes X_j)-I_d.
}
Denote the $(u,v)$-entry of the $d\times d$ matrix $E$ by $E_{uv}$. Then, for $p\geq 2$,
\be{
\norm{E}_p=\left[\E(\sum_{u,v=1}^d E^2_{uv})^{p/2}\right]^{1/p}\leq d\max_{u,v} \norm{E_{uv}}_p.
}
Write $X_i=(X_{i1},\dots, X_{id})^T$ and, from $\E(E_{uv})=0$,
\bes{
E_{uv}=&\frac{1}{n}\sum_{i=1}^n X_{iu} \sum_{j\in A_i} X_{jv} -\delta_{uv}=\frac{2(\theta_1 \theta_2)^{1/2} b_n'^2}{\sqrt{n}} \sum_{i=1}^n \sum_{j\in A_i} \left[\frac{X_{iu} X_{jv} -\E(X_{iu} X_{jv})}{2(\theta_1 \theta_2)^{1/2} b_n'^2 \sqrt{n}} \right]\\
=:&\frac{2(\theta_1 \theta_2)^{1/2} b_n'^2}{\sqrt{n}} \sum_{i=1}^n \sum_{j\in A_i} X^{uv}_{ij}=:\frac{2(\theta_1 \theta_2)^{1/2} b_n'^2}{\sqrt{n}} V_{uv}.
}
In the remainder of this proof, we show that if $2\leq p\leq \theta_1 n /\theta_2$ as in the condition \eq{037}, then
\ben{\label{041}
\norm{V_{uv}}_p\leq C\sqrt{p},
}
and hence conclude \eq{025}.

Let $V^{(ij)}_{uv}=V_{uv}-\sum_{(k,l)\in B_{ij}}X^{uv}_{kl}$.
Then $|V_{uv}-V_{uv}^{(ij)}|\leq \sqrt{\theta_2}/ \sqrt{\theta_1 n}$, and, for bounded $|t|\sqrt{\theta_2}/ \sqrt{\theta_1 n}$ and using the local dependence assumption in the first equation below,
\be{
h'(t)=\sum_{i=1}^n \sum_{j\in A_i} \E X^{uv}_{ij} [e^{tV_{uv}}-e^{t V_{uv}^{(ij)}}] \leq C|t| h(t), \ \text{where}\ h(t)=Ee^{tV_{uv}}.
}
This implies 
\ben{\label{042}
h(t)=Ee^{tV_{uv}}\leq e^{Ct^2}\  \text{for bounded}\  |t|\sqrt{\theta_2}/ \sqrt{\theta_1 n}.
}
\eqref{042} means that $V_{uv}$ is sub-gamma with variance factor $C$ and scale parameter $\sqrt{\theta_2}/\sqrt{\theta_1n}$ in the sense of \cite[Section 2.4]{BLM13}. Then, by Theorem 2.3 in \cite{BLM13} and Stirling's formula, 
\[
\norm{V_{uv}}_p\leq C(\sqrt{p}+p\sqrt{\theta_2}/\sqrt{\theta_1n})\leq C\sqrt p,
\]
where the last inequality follows by \eqref{037}. This proves \eq{041}.
\end{proof}

\begin{proof}[Proof of \cref{thm:local2}]
We use $C$ to denote positive absolute constants, which may differ in different expressions.
We first do truncation.
Let $\tilde X_{ij}:=X_{ij} 1_{\{|X_{ij}|\leq b\log n\}} -\E X_{ij} 1_{\{|X_{ij}|\leq  b\log n\}}$, $\tilde X_i=(\tilde X_{i1},\dots, \tilde X_{id})^T$, $W^{(l)}=n^{-1/2}\sum_{i\in g_l} X_i$, $\tilde W^{(l)}=n^{-1/2}\sum_{i\in g_l} \tilde X_i$ and $\tilde W=\sum_{l=1}^L \tilde W^{(l)}=n^{-1/2}\sum_{i=1}^n \tilde X_i$.

From $\norm{X_{ij}}_{\psi_1}\leq b$ and \cite[Lemma 5.4]{Ko21}, we have, for every positive integer $p$,
\bes{
&\E|n^{-1/2}(X_{ij}-\tilde X_{ij})|^p\leq n^{-p/2} 2^{p-1} \E[|X_{ij}|^p 1_{\{|X_{ij}|>b\log n\}}]\\
\leq & n^{-p/2}2^{p-1} p! 2 e^{-b\log n/b} (b\log n+b)^p\\
=& \frac{p!}{2} \left(\frac{2b\log n+2b}{\sqrt{n}}\right)^{p-2} \cdot \frac{8 (b\log n+b)^2}{ n^2}.
}
Using the independence of the $X$'s within each group $g_l$ and the Bernstein inequality (\cite[Theorem 2.10]{BLM13}), we obtain
\be{
P(W^{(l)}_j-\tilde W^{(l)}_j>\sqrt{2v_0t}+c_0 t)\vee P(-(W^{(l)}_j-\tilde W^{(l)}_j)>\sqrt{2v_0t}+c_0 t)\leq e^{-t},\quad \forall\ t>0,
}
where $v_0=8(b\log n+b)^2/n$, $c_0=(2b\log n+2b)/\sqrt{n}$ and $W_j$ denotes the $j$th component of $W$.
Therefore, by \cite[Theorem 2.3]{BLM13} we obtain, for $p\geq 1$,
\be{
\norm{W^{(l)}_j-\tilde W^{(l)}_j}_{p}\leq C(\sqrt{pv_0}+p c_0)\leq  \frac{C p b \log n}{\sqrt{n}},
}
\be{
\norm{W_j-\tilde W_j}_{p}\leq \sum_{l=1}^L \norm{W^{(l)}_j-\tilde W^{(l)}_j}_{p} \leq \frac{C p L b\log n}{\sqrt{n}},
}
and
\be{
\norm{W-\tilde W}_{p}\leq \sum_{j=1}^d\norm{W_j-\tilde W_j}_{p}\leq \frac{C p d L b\log n}{\sqrt{n}}.
}
%\blue{If we cite \cite{BLM13}, it may be better to cite Thms. 2.10 and 2.3 of \cite{BLM13} instead of \cite{VW96} to directly get the moment bound. Of course, this is inessential.}
Using the triangle inequality, we have
\be{
\mcl{W}_p(W, Z)\leq \mcl{W}_p(W, \tilde W)+\mcl{W}_p(\tilde W, \tilde Z)+\mcl{W}_p(\tilde Z, Z),
}
where $\tilde Z\sim N(0, \Var(\tilde W))$. 
Note that
\bes{
&|\E[\tilde W_j \tilde W_k]-\E[W_j W_k]|\\
=& \left| \frac{1}{n}\sum_{i=1}^n\sum_{i'\in A_i} \left\{\E[X_{ij}1_{\{|X_{ij}|\leq  b\log n\}} X_{i'k}1_{\{|X_{i'k}|\leq  b\log n\}}] -  \E [X_{ij} X_{i'k}] \right\}   \right|\\
=& \left| \frac{1}{n}\sum_{i=1}^n\sum_{i'\in A_i} \left\{-\E[X_{ij}1_{\{|X_{ij}|\leq  b\log n\}} X_{i'k}1_{\{|X_{i'k}|>  b\log n\}}] -  \E [X_{ij}1_{\{|X_{ij}|>  b\log n\}} X_{i'k}] \right\}   \right|\\
\leq& 2 b \log n \max_i\sum_{i'\in A_i} \E[|X_{i'k}| 1_{\{|X_{i'k}|\geq  b\log n\}}]\\
&+\max_i\sum_{i'\in A_i} \sqrt{\E[X_{ij}^21_{\{|X_{ij}|> b\log n\}}] \E[X_{i'k}^21_{\{|X_{i'k}|> b\log n\}}]}\\
\leq & \frac{C \theta_1 b^2 \log^2 n}{n},
}
where we used \cite[Lemma 5.4]{Ko21} in the last inequality.
This implies, from the $p$-Wasserstein bound via Stein kernels by \cite[Proposition 3.4(ii)]{LeNoPe15},
\be{
\mcl{W}_p(\tilde Z, Z)\leq \frac{C p^{1/2} d \theta_1 b^2 \log^2 n}{n}\ \text{for all}\ p\geq 2.
}
By the eigenvalue stability inequality $|\lambda_i(A+B)-\lambda_i(A)|\leq \norm{B}_{op}$, the eigenvalues of $\Var(\tilde W)$ differ from 1 by at most $C d\theta_1 b^2\log^2 n/n$.
Therefore, assuming $d^{1/2}\theta_1^{1/2} b\log n/\sqrt{n}$ to be sufficiently small as in the condition of \cref{thm:local2}, subject to the truncation error $C p d L b \log n/\sqrt{n}+C p^{1/2} d\theta_1 b^2 \log^2 n/n$, and by a renormalization, the problem reduces to the setting of \cref{thm:local2} with the additional assumption that 
%$(I-A)^{-1}=I+A+A^2+..., (I+A)^{1/2}=I+A/2-...$
\be{
|X_{ij}|\leq b_n':=C  b\log n,\quad
|X_i|\leq b_n:=C d^{1/2} b \log n\ \text{for all}\ i, j.
}
Using \cref{thm:local3} and $1\leq C \theta_1 b^2$, there exist positive absolute constants $c$ and $C$ such that, if
\be{
2\leq  p\leq \min\{\frac{\theta_1}{\theta_2}, \frac{c}{\theta_1^2 b_n^2}\} n,
}
then
\be{
\mcl{W}_p(W, Z)\leq C p \Delta_d.
}
The upper bounds \eq{026} and \eq{033} then follows from \cref{t4} and \cref{thm:multiMD} respectively by a similar argument as in the proof of \cref{thm:local1}.
\end{proof}

\section{Supplementary material}

\subsection{Proof of \cref{thm:qf}}

\cref{thm:qf} is a straightforward consequence of the following $p$-Wasserstein bound and \cref{t4}:

\begin{proposition}\label{prop:qf}
Under the assumptions of \cref{thm:qf}, for any $2\leq p\leq2\|F\|_{op}^{-2/3}$,
\[
\mcl W_p(W,Z)\leq CpK^4\|F\|_{op},
\]
where $C$ is a positive absolute constant.
\end{proposition}

\begin{proof}[Proof of \cref{thm:qf}]
We first note that $\|F\|_{op}\leq\|F\|_{H.S.}=1/\sqrt 2$. 
We apply \cref{t4} with $r_0=\alpha_1=1$, $\Delta_1=\|F\|_{op}$ and $p_0=2\|F\|_{op}^{-2/3}$. Then it remains to check $\log\|F\|_{op}^{-1}\leq \|F\|_{op}^{-2/3}$. This follows from the fact that $\log x\leq x^{2/3}$ for all $x>0$. 
\end{proof}

\begin{proof}[Proof of \cref{prop:qf}]
We construct an exchangeable pair $(W,W')$ in the same way as in the proof of \cref{prop:homo}. So we obtain the bound \eqref{homo-est}. We derive refined bounds for $H_1$ and $H_2$ using the assumption $q=2$ and the boundedness of $X_i$. 
In the proof, a symmetric function $g:[n]^r\to\mathbb R$ is also regarded as an element of $(\mathbb R^n)^{\odot r}$. In particular, given a partition $\mcl J\in\Pi_r$, we define the mixed injective norm $\|g\|_{\mcl J}$ as in \cref{sec:wiener}. Note that, if two partitions $\mcl J_1,\mcl J_2\in\Pi_r$ are such that any element of $\mcl J_1$ is contained in an element of $\mcl J_2$, then $\|g\|_{\mcl J_1}\leq\|g\|_{\mcl J_2}$ by definition. 
Note also that $\norm{F}_{op}\leq\norm{F}_{H.S.}=1/\sqrt 2<1$. 
Also, we will freely use tensor notations introduced in \cref{sec:proof-p-wass}.
\medskip

\textbf{Step 1. Bounding $H_1$.} We decompose $E$ as
\ben{\label{qf-E-decomp}
E=\sum_{i=1}^n(X_i^2-1)Q_i(X)^2
+\left\{2\sum_{i=1}^nQ_i(X)^2-1\right\}
=:E_1+E_2.
}
Define an $n$-variate polynomial $\wt Q$ as
\[
\wt Q(x_1,\dots,x_n)=\sum_{i=1}^n(x_i^2-1)Q_i(x_1,\dots,x_n)^2
=\sum_{i=1}^n(x_i^2-1)\left(\sum_{i'=1}^nf(i,i')x_{i'}\right)^2.
\]
By \cref{lem:aw-est},
\ben{\label{qf-e1-est0}
\norm{E_1}_p\leq C\sum_{r=1}^4K^r\sum_{\mcl J\in \Pi_r}p^{|\mcl J|/2}\|\E\nabla^r\wt Q(X)\|_{\mcl J}.
}
We bound summands of $\sum_{r=1}^4$ in the following way.
\smallskip

\ul{Case 1: $r=1$}. Since $\E\nabla\wt Q(X)=0$, we have
\[
K\sum_{\mcl J\in \Pi_1}p^{|\mcl J|/2}\lnorm{\E\nabla\wt Q(X)}_{\mcl J}=0.
\]

\ul{Case 2: $r=2$}. For $j,k\in\{1,\dots,n\}$,
\[
\E\partial_{j,k}\wt Q(X)=2\sum_{i=1}^nf(i,j)^21_{\{j=k\}}.
\]
Hence, using $\|f\|=1/\sqrt{2}$ by standardization and $\sqrt{\mcl M(f)}\leq \|F\|_{op}$ (we will use these two facts implicitly in the remainder of the proof),
\ba{
\norm{\E\nabla^2\wt Q(X)}_{\{1,2\}}
=2\sqrt{\sum_{j=1}^n\left(\sum_{i=1}^nf(i,j)^2\right)^2}
\leq2\sqrt{\mcl M(f)}\|f\|\leq\sqrt2\|F\|_{op}
}
and
\ba{
\norm{\E\nabla^2\wt Q(X)}_{\{1\},\{2\}}
=\norm{\E\nabla^2\wt Q(X)}_{op}
=2\mcl M(f)\leq2\|F\|_{op}^2.
}
Therefore,
\[
K^2\sum_{\mcl J\in \Pi_2}p^{|\mcl J|/2}\lnorm{\E\nabla^2\wt Q(X)}_{\mcl J}
\leq CK^2(\sqrt p\|F\|_{op}+p\|F\|_{op}^2).
\]

\ul{Case 3: $r=3$}. Since $\E\nabla^3\wt Q(X)=0$,
\[
K^3\sum_{\mcl J\in \Pi_3}p^{|\mcl J|/2}\lnorm{\E\nabla^3\wt Q(X)}_{\mcl J}=0.
\]

\ul{Case 4: $r=4$}. Define a function $f_1:[n]^4\to\mathbb R$ as $f_1(j,k,l,m)=1_{\{j=k\}}f(j,l)f(j,m)$ for $j,k,l,m\in[n]$. Then, for $j,k,l,m\in[n]$,
\[
\E\partial_{j,k,l,m}\wt Q(X)=4!\wt f_1(j,k,l,m),
\]
where $\wt f_1$ is the symmetrization of $f_1$. 
%\ba{
%\E\partial_{j,k,l,m}\wt Q(X)
%%&=41_{\{j=k\}}f(j,l)f(j,m)+41_{\{j=l\}}f(j,k)f(j,m)+41_{\{j=m\}}f(j,l)f(j,k)+41_{\{k=l\}}f(j,k)f(k,m)+41_{\{k=m\}}f(j,k)f(k,l)+41_{\{l=m\}}f(j,l)f(k,l)\\
%%&=4(1_{\{j=k\}}+1_{\{l=m\}})f(j,l)f(k,m)+4(1_{\{j=l\}}+1_{\{j=m\}}+1_{\{k=l\}}+1_{\{k=m\}})f(j,k)f(l,m)\\
%&=4(1_{\{j=k\}}+1_{\{l=m\}})f(j,l)f(k,m)+4(1_{\{j\in\{l,m\}\}}+1_{\{k\in\{l,m\}\}})f(j,k)f(l,m)
%}

\begin{enumerate}[label=(\roman*)]

\item Case $|\mcl J|=1$. In this case, we have
\ba{
\lnorm{\E\nabla^4\wt Q(X)}_{\mcl J}
&\leq C\sqrt{\sum_{j,l,m=1}^nf(j,l)^2f(j,m)^2}
\leq C\|F\|_{op}\|F\|_{H.S.}\leq C\|F\|_{op}.
}

\item Case $|\mcl J|=2$. Observe that
\ba{
\lnorm{\E\nabla^4\wt Q(X)}_{\mcl J}
\leq\lnorm{\E\nabla^4\wt Q(X)}_{\{1,2\},\{3,4\}}\vee\lnorm{\E\nabla^4\wt Q(X)}_{\{1,2,3\},\{4\}}.
}
Since $f$ is symmetric, we have
\bm{
\lnorm{\E\nabla^4\wt Q(X)}_{\{1,2\},\{3,4\}}
\leq C\sup_{U,V\in(\mathbb R^n)^{\otimes2}:|U|\vee|V|\leq1}\left|\sum_{j,l,m=1}^nU_{jj}V_{lm}f(j,l)f(j,m)\right|\\
+C\sup_{U,V\in(\mathbb R^n)^{\otimes2}:|U|\vee|V|\leq1}\left|\sum_{j,k,m=1}^nU_{jk}V_{jm}f(j,k)f(j,m)\right|
}
and
\bm{
\lnorm{\E\nabla^4\wt Q(X)}_{\{1,2,3\},\{4\}}
\leq C\sup_{U\in(\mathbb R^n)^{\otimes3},v\in\mathbb R^n:|U|\vee|v|\leq1}\left|\sum_{j,l,m=1}^nU_{jjl}v_{m}f(j,l)f(j,m)\right|\\
+C\sup_{U\in(\mathbb R^n)^{\otimes3},v\in\mathbb R^n:|U|\vee|v|\leq1}\left|\sum_{j,k,l=1}^nU_{jkl}v_{j}f(j,l)f(j,k)\right|.
}
For any $U,V\in(\mathbb R^n)^{\otimes2}$,
\ba{
\left|\sum_{j,l,m=1}^nU_{jj}V_{lm}f(j,l)f(j,m)\right|
&=\left|\sum_{j=1}^nU_{jj}(FVF)_{jj}\right|\\
&\leq\|U\|_{H.S.}\|FVF\|_{H.S.}
\leq\|F\|_{op}^2\|U\|_{H.S.}\|V\|_{H.S.}
}
and
\ba{
\left|\sum_{j,k,m=1}^nU_{jk}V_{jm}f(j,k)f(j,m)\right|
&=\left|\sum_{j=1}^n(UF)_{jj}(VF)_{jj}\right|\\
&\leq\|UF\|_{H.S.}\|VF\|_{H.S.}
\leq\|F\|_{op}^2\|U\|_{H.S.}\|V\|_{H.S.}.
}
Hence
\ben{\label{qf-e-22}
\lnorm{\E\nabla^4\wt Q(X)}_{\{1,2\},\{3,4\}}\leq C\|F\|_{op}^2.
}
In the meantime, for any $U\in(\mathbb R^n)^{\otimes3}$ and $v\in\mathbb R^n$,
\ba{
\left|\sum_{j,l,m=1}^nU_{jjl}v_{m}f(j,l)f(j,m)\right|
&=\left|\sum_{j=1}^n\left(\sum_{l=1}^nU_{jjl}f(j,l)\right)(Fv)_j\right|\\
&\leq\sqrt{\sum_{j=1}^n\left(\sum_{l=1}^nU_{jjl}f(j,l)\right)^2}|Fv|
%&\leq\sqrt{\sum_{j=1}^n\left(\sum_{l=1}^nU_{jjl}^2\right)\left(\sum_{l=1}^nf(j,l)^2\right)}\|F\|_{op}|v|
\leq\|F\|_{op}^2|U||v|
}
and, with $U_j=(U_{jkl})_{1\leq k,l\leq n}$,
\ba{
\left|\sum_{j,k,l=1}^nU_{jkl}v_{j}f(j,l)f(j,k)\right|
&=\left|\sum_{j=1}^n(FU_{j}F)_{jj}v_{j}\right|
\leq\|F\|_{op}^2\sum_{j=1}^n\|U_j\|_{H.S.}|v_j|\\
&\leq\|F\|_{op}^2|U||v|.
}
Hence
\[
\lnorm{\E\nabla^4\wt Q(X)}_{\{1,2,3\},\{4\}}\leq C\|F\|_{op}^2.
\]
Consequently, 
\[
\lnorm{\E\nabla^4\wt Q(X)}_{\mcl J}\leq C\|F\|_{op}^2.
\]

\item Case $|\mcl J|=3$. In this case, we have
\ba{
\lnorm{\E\nabla^4\wt Q(X)}_{\mcl J}
=\lnorm{\E\nabla^4\wt Q(X)}_{\{1,2\},\{3\},\{4\}}.
}
Hence, by \eqref{qf-e-22},
\[
\lnorm{\E\nabla^4\wt Q(X)}_{\mcl J}\leq C\|F\|_{op}^2.
\]

\item Case $|\mcl J|=4$. In this case we have $\mcl J=\{\{1\},\{2\},\{3\},\{4\}\}$. Therefore, by \eqref{qf-e-22},
\[
\lnorm{\E\nabla^4\wt Q(X)}_{\mcl J}\leq C\|F\|_{op}^2.
\]
%By Proposition 2.1 in \cite{Ra19a},
%\ba{
%\lnorm{\E\nabla^4\wt Q(X)}_{\mcl J}
%=\sup_{u\in\mathbb R^n:|u|\leq1}\left|\langle\E\nabla^4\wt Q(X),u^{\otimes4}\rangle\right|.
%}
%Then, for any $u\in\mathbb R^n$,
%\ba{
%\left|\langle\E\nabla^4\wt Q(X),u^{\otimes4}\rangle\right|
%&=24\sum_{j,k,l,m=1}^nf_1(j,k,l,m)u_ju_ku_lu_m\\
%&=24\sum_{j,l,m=1}^nf(j,l)f(j,m)u_j^2u_lu_m\\
%&=24\sum_{j=1}^nu_j^2\left(\sum_{l=1}^nf(j,l)u_l\right)^2
%\leq24\|F\|_{op}^2|u|^2.
%}
%Hence
%\[
%\lnorm{\E\nabla^4\wt Q(X)}_{\mcl J}\leq24\|F\|_{op}^2.
%\]

\end{enumerate}
All together, we obtain
\ba{
K^4\sum_{\mcl J\in \Pi_4}p^{|\mcl J|/2}\lnorm{\E\nabla^4\wt Q(X)}_{\mcl J}
\leq CK^4(\sqrt p\norm{F}_{op}+p^2\|F\|_{op}^2).
}

Combining these bounds with \eqref{qf-e1-est0} gives
\ben{\label{qf-E1-est}
\norm{E_1}_p\leq CK^4(\sqrt p\norm{F}_{op}+p^2\|F\|_{op}^2).
}

%A straightforward computation shows that, for $j,k,l,m\in[n]$,
%\ba{
%\E\partial_{j}\wt Q(X)&=0,&
%\E\partial_{j,k}\wt Q(X)&=2\sum_{i=1}^nf(i,j)^21_{\{j=k\}},\\
%\E\partial_{j,k,l}\wt Q(X)&=0,&
%\E\partial_{j,k,l,m}\wt Q(X)&=4!\wt f_1(j,k,l,m).
%}
%Therefore, 
%\ba{
%\norm{E_1}_p\leq C\left(K^2\sqrt{p}\norm{\E\nabla^2\wt Q(X)}_{\{1,2\}}
%+K^2p\norm{\E\nabla^2\wt Q(X)}_{\{1\},\{2\}}+
%K^4\max_{\mcl J\in \Pi_4}p^{|\mcl J|/2}\|\wt f_1\|_{\mcl J}\right).
%}
%Since
%\ba{
%\norm{\E\nabla^2\wt Q(X)}_{\{1,2\}}
%=2\sqrt{\sum_{j=1}^n\left(\sum_{i=1}^nf(i,j)^2\right)^2}
%\leq2\mcl M(f)\|f\|\leq\sqrt2\|F\|_{op}
%}
%and
%\ba{
%\norm{\E\nabla^2\wt Q(X)}_{\{1\},\{2\}}
%=\norm{\E\nabla^2\wt Q(X)}_{op}
%=2\mcl M(f)\leq2\|F\|_{op}^2,
%}
%we conclude
%\ban{
%\norm{E_1}_p&\leq C\left(K^2\sqrt{p}\|F\|_{op}+K^2p\|F\|_{op}^2+
%K^4\max_{\mcl J\in \Pi_4}p^{|\mcl J|/2}\|\wt f_1\|_{\mcl J}\right),\notag\\
%&\leq C\left(K^2\sqrt{p}\|F\|_{op}+
%K^4\max_{\mcl J\in \Pi_4}p^{|\mcl J|/2}\|\wt f_1\|_{\mcl J}\right),\label{qf-E1-est}
%}
%where the second line follows from $\sqrt p\leq\|F\|_{op}^{-1}$. 
In the meantime, by a similar argument to the proof of \cref{prop:homo} (cf.~\eq{homo-E-gauss} and the bound on $\|H_{12}\|_p$ therein),
\[
\norm{E_2}_p\leq CK^{2}\max_{\mcl J\in\Pi_{2}}p^{|\mcl J|/2}\|f\wt\otimes_1f\|_{\mcl J}.
\]
%\blue{$K^2$ is still OK because $E_2$ is a quadratic polynomial in $X$}
Observe that
\ba{
\left(f\wt\otimes_1f(j,k)\right)_{1\leq j,k\leq n}=\left(\sum_{i=1}^nf(i,j)f(i,k)\right)_{1\leq j,k\leq n}=F^2.
}
Hence we have
\ba{
\|f\wt\otimes_1f\|_{\{1,2\}}=\|F^2\|_{H.S.}\leq\|F\|_{op}\|F\|_{H.S.}=\|F\|_{op}/\sqrt 2
}
and
\ba{
\|f\wt\otimes_1f\|_{\{1\},\{2\}}=\|F\|_{op}^2.
}
Consequently,
\ben{\label{qf-E2-est}
\|E_2\|_p\leq CK^{4}\max\{\sqrt p\|F\|_{op},p\|F\|_{op}^2\}
\leq CK^{4}\sqrt p\|F\|_{op}.
}
Combining \eqref{qf-E-decomp}, \eqref{qf-E1-est} and \eqref{qf-E2-est} gives
\ben{\label{qf-h1-est}
H_1\leq C\sqrt p(\norm{E_1}_p+\norm{E_2}_p)
\leq CK^4\left(p\|F\|_{op}+p^{5/2}\|F\|_{op}^2\right)
\leq CpK^4\|F\|_{op},
}
where the last inequality follows by the condition $p\|F\|_{op}^{2/3}\leq2$.
\medskip

\textbf{Step 2. Bounding $H_2$.} Since $|X'_i-X_i|\leq 2K$ a.s.,
\ben{\label{qf-d4-bound}
\frac{n}{2}\E[D^4\mid X]
\leq 8 (2K)^4\sum_{i=1}^nQ_i(X)^4.
}
By \cref{lem:aw-est},
\ben{\label{qi-mean}
\sum_{i=1}^n\E Q_i(X)^4\leq CK^4\sum_{i=1}^n\left(\sum_{j=1}^nf(i,j)^2\right)^2
\leq CK^4\mcl M(f)\|f\|^2\leq CK^4\|F\|_{op}^2
}
and
\ben{\label{qi-var}
\lnorm{\sum_{i=1}^nQ_i(X)^4-\sum_{i=1}^n\E Q_i(X)^4}_{p}
\leq C\sum_{r=1}^4K^r\sum_{\mcl J\in \Pi_r}p^{|\mcl J|/2}\lnorm{\sum_{i=1}^n\E\nabla^r Q_i^4(X)}_{\mcl J}.
}
We bound summands of $\sum_{r=1}^4$ in the following way.

\ul{Case 1: $r=1$}. For $j\in\{1,\dots,n\}$,
\[
\E\partial_{j} Q_i^4(X)=4f(i,j)\E\left(\sum_{i'=1}^nf(i,i')X_{i'}\right)^3.
\]
Therefore, with
\[
v:=\left(\E\left(\sum_{i'=1}^nf(1,i')X_{i'}\right)^3,\dots,\E\left(\sum_{i'=1}^nf(n,i')X_{i'}\right)^3\right)^{T},
\]
we have
\[
\lnorm{\sum_{i=1}^n\E\nabla Q_i^{4}(X)}_{\{1\}}=4|Fv|\leq4\|F\|_{op}|v|.
\]
By \cref{lem:aw-est},
\ba{
|v|^2=\sum_{i=1}^n\left|\E\left(\sum_{i'=1}^nf(i,i')X_{i'}\right)^3\right|^2
\leq CK^6\sum_{i=1}^n\left(\sum_{i'=1}^nf(i,i')^2\right)^3
\leq CK^6\|F\|_{op}^4.
}
Hence
\[
K\sum_{\mcl J\in \Pi_1}p^{|\mcl J|/2}\lnorm{\sum_{i=1}^n\E\nabla Q_i(X)}_{\mcl J}\leq CK^4\sqrt p\|F\|_{op}^3.
\]

\ul{Case 2: $r=2$}. For $j,k\in\{1,\dots,n\}$,
\[
\E\partial_{jk} Q_i^4(X)=12f(i,j)f(i,k)\sum_{i'=1}^nf(i,i')^2.
\]
Hence
\ba{
\sum_{i=1}^n\E\nabla^r Q_i^4(X)
=12F\diag\left(\sum_{i'=1}^nf(1,i')^2,\dots,\sum_{i'=1}^nf(n,i')^2\right)F.
}
Therefore,
\ba{
\lnorm{\sum_{i=1}^n\E\nabla^r Q_i^4(X)}_{\{1\},\{2\}}
=\lnorm{\sum_{i=1}^n\E\nabla^r Q_i^4(X)}_{op}
\leq12\|F\|_{op}^2\mcl M(f)\leq12\|F\|_{op}^4
}
and
\ba{
\lnorm{\sum_{i=1}^n\E\nabla^r Q_i^4(X)}_{\{1,2\}}
&=\lnorm{\sum_{i=1}^n\E\nabla^r Q_i^4(X)}_{H.S.}
\leq12\|F\|_{op}^2\sqrt{\sum_{i=1}^n\left(\sum_{i'=1}^nf(i,i')^2\right)^2}\\
&\leq6\sqrt 2\|F\|_{op}^3.
}
Hence
\[
K^2\sum_{\mcl J\in \Pi_2}p^{|\mcl J|/2}\lnorm{\sum_{i=1}^n\E\nabla^2 Q_i^4(X)}_{\mcl J}
\leq CK^2(\sqrt p\|F\|_{op}^3+p\|F\|_{op}^4).
\]

\ul{Case 3: $r=3$}. Since $\E\partial_{jkl} Q_i^4(X)=0$ for all $j,k,l\in\{1,\dots,n\}$, 
\[
K^3\sum_{\mcl J\in \Pi_3}p^{|\mcl J|/2}\lnorm{\sum_{i=1}^n\E\nabla^3 Q_i^4(X)}_{\mcl J}=0.
\]

\ul{Case 4: $r=4$}. For $j,k,l,m\in\{1,\dots,n\}$,
\[
\E\partial_{jklm} Q_i^4(X)=24f(i,j)f(i,k)f(i,l)f(i,m).
\]

\begin{enumerate}[label=(\roman*)]

\item Case $|\mcl J|=1$. In this case, we have
\ba{
\lnorm{\sum_{i=1}^n\E\nabla^4 Q_i^4(X)}_{\mcl J}
&=24\sqrt{\sum_{j,k,l,m=1}^n\left(\sum_{i=1}^nf(i,j)f(i,k)f(i,l)f(i,m)\right)^2}\\
&=24\sqrt{\sum_{i,i'=1}^n|(F^2)_{ii'}|^4}
\leq24\|F^2\|_{op}\|F^2\|_{H.S.}\\
&\leq24\|F\|_{op}^3\|F\|_{H.S.}
=12\sqrt 2\|F\|_{op}^3.
}

\item Case $|\mcl J|=2$. Observe that
\ba{
\lnorm{\sum_{i=1}^n\E\nabla^4 Q_i^4(X)}_{\mcl J}
\leq\lnorm{\sum_{i=1}^n\E\nabla^4 Q_i^4(X)}_{\{1,2\},\{3,4\}}\vee\lnorm{\sum_{i=1}^n\E\nabla^4 Q_i^4(X)}_{\{1,2,3\},\{4\}}.
}
For any $U,V\in(\mathbb R^n)^{\otimes2}$,
\ba{
\sum_{i=1}^n\langle\E\nabla^4 Q_i^4(X),U\otimes V\rangle
&=24\sum_{i,j,k,l,m=1}^nf(i,j)f(i,k)f(i,l)f(i,m)U_{jk}V_{lm}\\
&=24\sum_{i=1}^n(FUF)_{ii}(FVF)_{ii}
\leq24\|FUF\|_{H.S.}\|FVF\|_{H.S.}\\
&\leq24\|F\|_{op}^4\|U\|_{H.S.}\|V\|_{H.S.}.
}
Hence
\ben{\label{qf-d4-22}
\lnorm{\sum_{i=1}^n\E\nabla^4 Q_i^4(X)}_{\{1,2\},\{3,4\}}\leq24\|F\|_{op}^4.
}
In the meantime, for any $U\in(\mathbb R^n)^{\otimes3}$ and $v\in\mathbb R^n$,
\ba{
\sum_{i=1}^n\langle\E\nabla^4 Q_i^4(X),U\otimes v\rangle
&=24\sum_{i,j,k,l,m=1}^nf(i,j)f(i,k)f(i,l)f(i,m)U_{jkl}v_{m}\\
&=24\sum_{i,j=1}^nf(i,j)(FU_{j}F)_{ii}(Fv)_i,
}
where $U_j=(U_{jkl})_{1\leq k,l\leq n}$. Thus, by the Cauchy--Schwarz inequality,
\ba{
\sum_{i=1}^n\langle\E\nabla^4 Q_i^4(X),U\otimes v\rangle
&\leq24\sqrt{\sum_{i,j=1}^nf(i,j)^2|(Fv)_i|^2\sum_{i,j=1}^n(FU_{j}F)_{ii}^2}\\
&\leq24\|F\|_{op}^2\sqrt{\sum_{i=1}^n|(Fv)_i|^2\sum_{j=1}^n\|U_{j}\|_{H.S.}^2}\\
&\leq24\|F\|_{op}^3|v||U|.
}
Hence
\[
\lnorm{\sum_{i=1}^n\E\nabla^4 Q_i^4(X)}_{\{1,2,3\},\{4\}}\leq24\|F\|_{op}^3.
\]
Consequently, 
\[
\lnorm{\sum_{i=1}^n\E\nabla^4 Q_i^4(X)}_{\mcl J}\leq24\|F\|_{op}^3.
\]

\item Case $|\mcl J|=3$. In this case, we have
\ba{
\lnorm{\sum_{i=1}^n\E\nabla^4 Q_i^4(X)}_{\mcl J}
=\lnorm{\sum_{i=1}^n\E\nabla^4 Q_i^4(X)}_{\{1,2\},\{3\},\{4\}}.
}
%For any $U=(u_{jk})\in(\mathbb R^n)^{\otimes2}$ and $v,w\in\mathbb R^n$,
%\ba{
%\sum_{i=1}^n\langle\E\nabla^4 Q_i(X),U\otimes v\otimes w\rangle
%&=24\sum_{i,j,k,l,m=1}^nf(i,j)f(i,k)f(i,l)f(i,m)u_{jk}v_{l}w_m\\
%&=24\sum_{i=1}^n(FUF)_{ii}(Fv)_i(Fw)_i\\
%&\leq24\|FUF\|_{H.S.}\sqrt{\sum_{i=1}^n|(Fv)_i(Fw)_i|^2}\\
%&\leq24\|FUF\|_{H.S.}|Fv||Fw|
%\leq24\|F\|_{op}^4\|U\|_{H.S.}|u||w|.
%}
%Hence
Therefore, by \eqref{qf-d4-22},
\[
\lnorm{\sum_{i=1}^n\E\nabla^4 Q_i^4(X)}_{\mcl J}\leq24\|F\|_{op}^4.
\]

\item Case $|\mcl J|=4$. %By Proposition 2.1 in \cite{Ra19a},
%\ba{
%\lnorm{\sum_{i=1}^n\E\nabla^4 Q_i(X)}_{\mcl J}
%=\sup_{u\in\mathbb R^n:|u|\leq1}\left|\sum_{i=1}^n\langle\E\nabla^4 Q_i(X),u^{\otimes4}\rangle\right|.
%}
%Then, for any $u\in\mathbb R^n$,
%\ba{
%\left|\sum_{i=1}^n\langle\E\nabla^4 Q_i(X),u^{\otimes4}\rangle\right|
%&=24\sum_{i=1}^n\left(\sum_{j=1}^nf(i,j)u_j\right)^4
%\leq24\|F\|_{op}^2|u|^2\sum_{i=1}^n\left(\sum_{j=1}^nf(i,j)u_j\right)^2\\
%&=24\|F\|_{op}^2|u|^2\sum_{i,j,k=1}^nf(i,j)f(i,k)u_ju_k\\
%&=24\|F\|_{op}^2|u|^2(u^\top F^2u)
%\leq24\|F\|_{op}^4|u|^4.
%}
%Hence
In this case we have $\mcl J=\{\{1\},\{2\},\{3\},\{4\}\}$. 
Therefore, by \eqref{qf-d4-22},
\[
\lnorm{\sum_{i=1}^n\E\nabla^4 Q_i^4(X)}_{\mcl J}\leq24\|F\|_{op}^4.
\]

\end{enumerate}
All together, we obtain
\ba{
K^4\sum_{\mcl J\in \Pi_4}p^{|\mcl J|/2}\lnorm{\sum_{i=1}^n\E\nabla^4 Q_i^4(X)}_{\mcl J}
\leq CK^4(p\norm{F}_{op}^3+p^2\|F\|_{op}^4).
}

Combining these bounds with \eqref{qi-var} and the condition $p\leq2\|F\|_{op}^{-1}$, we obtain
\ben{\label{qi-var2}
\lnorm{\sum_{i=1}^nQ_i(X)^4-\sum_{i=1}^n\E Q_i(X)^4}
\leq CK^4\|F\|_{op}^2.
}
By \eqref{qf-d4-bound}, \eqref{qi-mean} and \eqref{qi-var2}, we conclude
\ben{\label{qf-h2-est}
H_2
%\leq CK^4p\sqrt{\|F\|_{op}^2
%+\max_{\mcl J\in \Pi_4}p^{|\mcl J|/2}\|f_2\|_{\mcl J}}
\leq CK^4p\|F\|_{op}.
}
Combining \eqref{homo-est}, \eqref{qf-h1-est} and \eqref{qf-h2-est}, we complete the proof.
\end{proof}

\subsection{Removing the extra assumptions in derivation of (\ref{ov-est})}\label{appendix:ov-est}

In the literature, the bound \eqref{ov-est} was formally established only when $W$ has a bounded $C^\infty$ density $h$ with respect to $N(0,I_d)$ such that $h\geq\eta$ for some constant $\eta>0$ and $|\nabla h|$ is bounded. 
In this appendix, we show this assumption can be replaced with $\E|W|^p<\infty$. Our argument is largely the same as in Section 8 of \cite{Bo20}. 
Below we assume $W$ and $Z$ are independent without loss of generality.
\medskip

\textbf{Step 1}. In this step, we prove \eqref{ov-est} when $W$ has a compactly supported $C^\infty$ density $f$. 
%Take $\eta\in(0,1)$ arbitrarily and let $I^\eta$ be a random variable independent of $W$ and $Z$ with $P(I^\eta=1)=1-P(I^\eta=0)=\eta$. 
Let $U$ be a uniform random variable on $[0,1]$ independent of $W$ and $Z$. Also, let $Z'\sim N(0,I_d)$ be independent of everything else. Take $\eta\in(0,1)$ arbitrarily, and define $I^\eta:=1_{\{U\leq\eta\}}$ and $W^\eta:=I^\eta Z'+(1-I^\eta)W$. Then, for any bounded measurable function $g:\mathbb R^d\to\mathbb R$,
\ba{
\E g(W^\eta)=\eta\E g(Z')+(1-\eta)\E g(W)
=\eta\int_{\mathbb R^d}g(x)\phi(x)dx+(1-\eta)\int_{\mathbb R^d}g(x)f(x)dx.
}
Hence $\eta+(1-\eta)f/\phi$ is a density of $W^\eta$ with respect to $N(0,I_d)$. In this case we already have
\ben{\label{ov-step1}
\mcl W_p(W^\eta,Z)\leq\int_0^\infty\|\rho^\eta_t(F^\eta_t)\|_pdt,
}
where $F^\eta_t:=e^{-t}W^\eta+\sqrt{1-e^{-2t}}Z$ and $\rho_t^\eta$ is the score of $F^\eta_t$ with respect to $N(0,I_d)$. By the triangle inequality for the $p$-Wasserstein distance, we have
\ba{
|\mcl W_p(W,Z)-\mcl W_p(W^\eta,Z)|
&\leq\mcl W_p(W,W^\eta)
\leq\norm{W-W^\eta}_p\\
&=(\E I^\eta|W-Z'|^p)^{1/p}
=\eta^{1/p}\norm{W-Z'}_p.
}
Hence $|\mcl W_p(W,Z)-\mcl W_p(W^\eta,Z)|\to0$ as $\eta\downarrow0$.

%Meanwhile, by Lemma IV.1 in \cite{NoPeSw14}, 
%\red{the second equality below may not be correct because $\sigma(F_t^\eta)\not\subset \sigma(F_t, I^\eta, Z')$? If so, to show $\|\E[e^{-t}W-\frac{e^{-2t}}{\sqrt{1-e^{-2t}}}Z | F_t^\eta]\|_p\to \|\E[e^{-t}W-\frac{e^{-2t}}{\sqrt{1-e^{-2t}}}Z | F_t]\|_p$, maybe we let $I^\eta=1_{\{U\leq \eta\}}$, where $U\sim Unif(0,1)$ and go through the arguments as in Step 3?} \blue{Sorry for many mistakes. Yes, we should do as in Step 3.}
%\ba{
%\rho^\eta_t(F^\eta_t)
%&=\E\left[e^{-t}W^\eta-\frac{e^{-2t}}{\sqrt{1-e^{-2t}}}Z\mid F^\eta_t\right]\\
%&=\E\left[\E\left[e^{-t}W^\eta-\frac{e^{-2t}}{\sqrt{1-e^{-2t}}}Z\mid F_t,I^\eta\right]\mid F^\eta_t\right]\\
%&=\E\left[\E\left[e^{-t}W-\frac{e^{-2t}}{\sqrt{1-e^{-2t}}}Z\mid F_t,I^\eta\right]
%+\E\left[e^{-t}I^\eta(Z-W)\mid F_t,I^\eta\right]\mid F^\eta_t\right]\\
%&=\E[\rho_t(F_t)\mid F^\eta_t]
%+\E\left[e^{-t}I^\eta(Z-W)\mid F^\eta_t\right],
%}
%where we used the independence between $(W,Z)$ and $I^\eta$ in the last line.
%Hence
%\ba{
%\int_0^\infty\norm{\rho^\eta_t(F^\eta_t)}_pdt
%&\leq\int_0^\infty\norm{\red{\rho_t}(F_t)}_pdt
%+\int_0^\infty\norm{e^{-t}I^\eta(Z-W)}_pdt\\
%&=\int_0^\infty\norm{\rho_t(F_t)}_pdt+\eta^{1/p}\norm{W-Z}_p.
%}

Meanwhile, by Lemma IV.1 in \cite{NoPeSw14},
\ben{\label{fr-score-eta}
\rho^\eta_t(F^\eta_t)
=\E\left[e^{-t}W^\eta-\frac{e^{-2t}}{\sqrt{1-e^{-2t}}}Z\mid F^\eta_t\right].
}
In particular, \ba{
\norm{\rho^\eta_t(F^\eta_t)}_p
\leq e^{-t}(\norm{Z'}_p+\|W\|_p)+\frac{e^{-2t}}{\sqrt{1-e^{-2t}}}\|Z\|_p.
}
Hence, by the reverse Fatou lemma,
\[
\limsup_{\eta\downarrow0}\int_0^\infty\|\rho^\eta_t(F^\eta_t)\|_pdt
\leq\int_0^\infty\limsup_{\eta\downarrow0}\|\rho^\eta_t(F^\eta_t)\|_pdt.
\]
Therefore, we complete the proof once we show that $\|\rho^\eta_t(F^\eta_t)\|_p\to\|\rho_t(F_t)\|_p$ as $\eta\downarrow0$ for any fixed $t>0$. The latter follows once we verify the following two statements:
\begin{enumerate}[label=(\roman*)]

\item $\rho^\eta_t(F^\eta_t)\to\rho_t(F_t)$ as $\eta\downarrow0$ a.s.

\item $\{|\rho^\eta_t(F^\eta_t)|^p:\eta\in(0,1)\}$ is uniformly integrable. 

\end{enumerate}

\ul{Proof of (i)}. For any bounded measurable function $g:\mathbb R^d\to\mathbb R$,
\ba{
\E g(e^{-t}W^\eta+\sqrt{1-e^{-2t}}Z)&=\eta\E g(e^{-t}Z'+\sqrt{1-e^{-2t}}Z)+(1-\eta)\E g(F_t)\\
&=\eta\int_{\mathbb R^d}g(x)\phi(x)dx+(1-\eta)\int_{\mathbb R^d}g(x)f_t(x)\phi(x)dx,
}
where $f_t$ is the density of $F_t$ with respect to $N(0,I_d)$. Hence $\eta+(1-\eta)f_t$ is the smooth density of $F^\eta_t$ with respect to $N(0,I_d)$, and thus 
\[
\rho^\eta_t(F^\eta_t)=(1-\eta)\nabla f_t(F^\eta_t)/(\eta+(1-\eta)f_t(F^\eta_t)). 
\]
Since $f_t$ is smooth and $F^\eta_t\to F_t$ as $\eta\downarrow0$ a.s., we have $\rho^\eta_t(F^\eta_t)\to\nabla f_t(F_t)/f_t(F_t)=\rho_t(F_t)$ as $\eta\downarrow0$ a.s.

\ul{Proof of (ii)}. %\blue{text in parentheses was removed} 
Let
\[
G_t:=e^{-t}(|W|+|Z'|)+\frac{e^{-2t}}{\sqrt{1-e^{-2t}}}|Z|.
\]
Then we have $|\rho_t^\eta(F_t^\eta)|^p\leq\E[G_t^p\mid F_t^\eta]$ for any $\eta\in(0,1)$ by \eqref{fr-score-eta} and Jensen's inequality. Hence, for any $K>0$,
\[
\E[|\rho^\eta_t(F^\eta_t)|^p;|\rho_t(F^\eta_t)|^p>K]\leq\E[\E[G_t^p\mid F_t^\eta];\E[G_t^p\mid F_t^\eta]>K].
\]
Since $\E G_t^p<\infty$, $\{\E[G_t^p\mid F_t^\eta]:\eta\in(0,1)\}$ is uniformly integrable by Theorem 13.4 in \cite{Wi91}. Hence $\{|\rho^\eta_t(F^\eta_t)|^p:\eta\in(0,1)\}$ is uniformly integrable as well.
%Consequently, letting $\eta\downarrow0$ in \eqref{ov-step1}, we obtain \eqref{ov-est}. 
\medskip

\textbf{Step 2}. In this step, we prove \eqref{ov-est} when $W$ is bounded. Let $N$ be a random variable independent of $W$ and $Z$ and such that $N$ has a $C^\infty$ density $\psi$ and takes values in the unit ball in $\mathbb R^d$. Take $\eps>0$ arbitrarily and define $W^\eps:=W+\eps N$. Then, for any bounded measurable function $g:\mathbb R^d\to\mathbb R$,
\ba{
\E g(W^\eps)=\int_{\mathbb R^d}\E[g(W+\eps x)]\psi(x)dx
=\eps^{-d}\int_{\mathbb R^d}g(y)\E[\psi((y-W)/\eps)]dy.
}
Hence $f(y)=\eps^{-d}\E[\psi((y-W)/\eps)]$ is a density of $W^\eps$. Since $\psi$ is $C^\infty$ and compactly supported, $f$ is $C^\infty$. Also, since $W$ is bounded, $f$ is compactly supported. Thus, by Step 1,
\ben{\label{ov-step2}
\mcl W_p(W^\eps,Z)\leq\int_0^\infty\|\rho^\eps_t(F^\eps_t)\|_pdt,
}
where $F^\eps_t:=e^{-t}W^\eps+\sqrt{1-e^{-2t}}Z$ and $\rho_t^\eps$ is the score of $F^\eps_t$ with respect to $N(0,I_d)$. By the triangle inequality for the $p$-Wasserstein distance, we have
\ba{
|\mcl W_p(W,Z)-\mcl W_p(W^\eps,Z)|
\leq\mcl W_p(W,W^\eps)
\leq\norm{W-W^\eps}_p
=\eps\norm{N}_p.
}
Meanwhile, by Lemma IV.1 in \cite{NoPeSw14},
\ba{
\rho^\eps_t(F^\eps_t)
&=\E\left[e^{-t}W^\eps-\frac{e^{-2t}}{\sqrt{1-e^{-2t}}}Z\mid F^\eps_t\right]\\
&=\E\left[\E\left[e^{-t}W^\eps-\frac{e^{-2t}}{\sqrt{1-e^{-2t}}}Z\mid F_t,N\right]\mid F^\eps_t\right]\\
&=\E\left[\E\left[e^{-t}W-\frac{e^{-2t}}{\sqrt{1-e^{-2t}}}Z\mid F_t,N\right]
+\E\left[e^{-t}\eps N\mid F_t,N\right]\mid F^\eps_t\right]\\
&=\E[\rho_t(F_t)\mid F^\eps_t]
+\eps\E\left[e^{-t}N\mid F^\eps_t\right],
}
where we used the independence between $(W,Z)$ and $N$ in the last line.
Hence
\ba{
\int_0^\infty\norm{\rho^\eps_t(F^\eps_t)}_pdt
&\leq\int_0^\infty\norm{\rho_t(F_t)}_pdt
+\eps\norm{N}_p.
}
Consequently, letting $\eps\downarrow0$ in \eqref{ov-step2}, we obtain \eqref{ov-est}. 
\medskip

\textbf{Step 3}. In this step, we prove \eqref{ov-est} when $\E |W|^p<\infty$. Take $R>0$ arbitrarily and define $W^R:=W1_{\{|W|\leq R\}}$. Since $W^R$ is bounded, we have by Step 2
\ben{\label{ov-step3}
\mcl W_p(W^R,Z)\leq\int_0^\infty\|\rho^R_t(F^R_t)\|_pdt,
}
where $F^R_t:=e^{-t}W^R+\sqrt{1-e^{-2t}}Z$ and $\rho^R_t$ is the score of $F^R_t$ with respect to $N(0,I_d)$. By the triangle inequality for the $p$-Wasserstein distance, we have
\ba{
|\mcl W_p(W,Z)-\mcl W_p(W^R,Z)|
\leq\mcl W_p(W,W^R)
\leq\norm{W-W^R}_p
=(\E[|W|^p1_{\{|W>R|\}}])^{1/p}.
}
Since $\E|W|^p<\infty$, we obtain $|\mcl W_p(W,Z)-\mcl W_p(W^R,Z)|\to0$ as $R\to\infty$ by the dominated convergence theorem. 
Meanwhile, by Lemma IV.1 in \cite{NoPeSw14},
\ben{\label{fr-score}
\rho^R_t(F^R_t)
=\E\left[e^{-t}W^R-\frac{e^{-2t}}{\sqrt{1-e^{-2t}}}Z\mid F^R_t\right]
}
and
\ben{\label{f-score}
\rho_t(F_t)
=\E\left[e^{-t}W-\frac{e^{-2t}}{\sqrt{1-e^{-2t}}}Z\mid F_t\right].
}
In particular, 
\ba{
\norm{\rho^R_t(F^R_t)}_p
\leq e^{-t}\|W\|_p+\frac{e^{-2t}}{\sqrt{1-e^{-2t}}}\|Z\|_p.
}
Hence, by the reverse Fatou lemma,
\[
\limsup_{R\to\infty}\int_0^\infty\|\rho^R_t(F^R_t)\|_pdt
\leq\int_0^\infty\limsup_{R\to\infty}\|\rho^R_t(F^R_t)\|_pdt.
\]
Therefore, we complete the proof once we show that $\|\rho^R_t(F^R_t)\|_p\to\|\rho_t(F_t)\|_p$ as $R\to\infty$ for any fixed $t>0$. The latter follows once we verify the following two statements:
\begin{enumerate}[label=(\roman*)]

\item $\rho^R_t(F^R_t)\to \rho_t(F_t)$ as $R\to\infty$ a.s.

\item $\{|\rho^R_t(F^R_t)|^p:R>0\}$ is uniformly integrable. 

\end{enumerate}

\ul{Proof of (i)}. 
For any $u\in\mathbb R^d$, 
\ben{\label{w-cf-est}
|\E[W^R e^{\sqrt{-1}u\cdot F^R_t}]|
=|\E[W^R e^{\sqrt{-1}u\cdot e^{-t}W^R}]\E[e^{\sqrt{-1}u\cdot\sqrt{1-e^{-2t}} Z}]|
\leq \E|W|e^{-(1-e^{-2t})u^2/2}
}
and
\ben{\label{z-cf-est}
|\E[Z e^{\sqrt{-1}u\cdot F^R_t}]|
=|\E[e^{\sqrt{-1}u\cdot e^{-t}W^R}]\E[Ze^{\sqrt{-1}u\cdot\sqrt{1-e^{-2t}} Z}]|
\leq |u|  e^{-(1-e^{-2t})u^2/2}.
}
%\blue{We may drop $\sqrt{1-e^{-2t}}$ because it is bounded by 1.}
Hence, we can define a function $g_R:\mathbb R^d\to\mathbb C$ as 
\[
g_R(x)=\frac{1}{f_R(x)(2\pi)^d}\int_{\mathbb R^d}e^{-\sqrt{-1}u\cdot x}\E\left[\left(e^{-t}W^R-\frac{e^{-2t}}{\sqrt{1-e^{-2t}}}Z\right) e^{\sqrt{-1}u\cdot F^R_t}\right]du,~ x\in\mathbb R^d,
\]
where $f_R(x)=(1-e^{-2t})^{-d/2}\E[\phi((x-e^{-t}W^R)/\sqrt{1-e^{-2t}})]$ is the density of $F^R_t$. 
Similarly, we can define a function $g:\mathbb R^d\to\mathbb C$ as 
\[
g(x)=\frac{1}{f(x)(2\pi)^d}\int_{\mathbb R^d}e^{-\sqrt{-1}u\cdot x}\E\left[\left(e^{-t}W-\frac{e^{-2t}}{\sqrt{1-e^{-2t}}}Z\right) e^{\sqrt{-1}u\cdot F_t}\right]du,~ x\in\mathbb R^d,
\]
where $f(x)=(1-e^{-2t})^{-d/2}\E[\phi((x-e^{-t}W)/\sqrt{1-e^{-2t}})]$ is the density of $F_t$. 
By Theorem 2 in \cite{Ye74} and \eqref{fr-score}--\eqref{f-score}, we have $g_R(F^R_t)=\rho^R_t(F^R_t)$ a.s.~and  $g(F_t)=\rho_t(F_t)$ a.s. 
Moreover, by \eqref{w-cf-est}, \eqref{z-cf-est} and the dominated convergence theorem, $g_R(x)\to g(x)$ as $R\to\infty$ for any $x\in\mathbb R^d$. 
Hence $\rho^R_t(F^R_t)\to \rho_t(F_t)$ as $R\to\infty$ a.s.

%\ul{Proof of (ii)}. Let
%\[
%G^R_t:=e^{-t}W^R-\frac{e^{-2t}}{\sqrt{1-e^{-2t}}}Z.
%\]
%Take $K>0$ arbitrarily. Observe that $|\E[G^R_t\mid F^R_t]|\leq\E[|G^R_t|\mid F^R_t]\leq K^{1/p}$ if $|G^R_t|\leq K^{1/p}$. Hence,
%\ba{
%\E[|\rho^R_t(F^R_t)|^p;|\rho_t(F^R_t)|^p>K]
%&\leq\E[\E[|G^R_t|^p\mid F^R_t];|\E[G^R_t\mid F^R_t]|^p>K]~(\text{Jensen})\\
%&=\E[|G^R_t|^p;|\E[G^R_t\mid F^R_t]|^p>K]\\
%&\leq\E[|G^R_t|^p;|G^R_t|^p>K].
%}

\ul{Proof of (ii)}. Let
\[
G_t:=e^{-t}|W|+\frac{e^{-2t}}{\sqrt{1-e^{-2t}}}|Z|.
\]
Then we have $|\rho_t^R(F_t^R)|^p\leq\E[G_t^p\mid F_t^R]$ for any $R>0$ by \eqref{fr-score} and Jensen's inequality. Hence, for any $K>0$,
\[
\E[|\rho^R_t(F^R_t)|^p;|\rho_t(F^R_t)|^p>K]\leq\E[\E[G_t^p\mid F_t^R];\E[G_t^p\mid F_t^R]>K].
\]
Since $\E G_t^p<\infty$, $\{\E[G_t^p\mid F_t^R]:R>0\}$ is uniformly integrable by Theorem 13.4 in \cite{Wi91}. Hence $\{|\rho^R_t(F^R_t)|^p:R>0\}$ is uniformly integrable as well.\qed
%$\{|G^R_t|^p;R>0\}$ is uniformly integrable. Hence $\{|\rho^R_t(F^R_t)|^p:R>0\}$ is uniformly integrable as well.\qed

\section*{Acknowledgements}

Fang X. was partially supported by Hong Kong RGC GRF 14302418, 14305821, a CUHK direct grant and a CUHK start-up grant. Koike Y. was partly supported by JST CREST and JSPS KAKENHI Grant Number JP19K13668.

\end{document}